\documentclass[preprint,reqno,12pt]{amsart}
\usepackage{amssymb}
\usepackage{amsmath, amsfonts, amsthm, cite}
\usepackage{color}
\input xy
\xyoption{all}
\CompileMatrices

\usepackage{eepic}
\usepackage{amssymb}
\usepackage{amsmath}
\usepackage{amsfonts}
\usepackage{indentfirst}
\usepackage{a4wide}
\usepackage{amstext}
\usepackage{amsthm}
\newtheorem{theorem}{Theorem}[section]

\newtheorem{lemma}[theorem]{Lemma}
\newtheorem{prop}[theorem]{Proposition}
\newtheorem{cor}[theorem]{Corollary}

\theoremstyle{definition}
\newtheorem{defn}[theorem]{Definition}

\newtheorem{example}[theorem]{Example}

\theoremstyle{remark}

\newtheorem{remark}[theorem]{Remark}

\numberwithin{equation}{section}

\def\N{{\mathbb N}}

\def\U{{\mathcal U}}

\def\R{{\mathbb R}}
\def\T{{\mathbb T}}

\def\H{{\mathcal H}}
\def\K{{\mathcal K}}
\def\W{{\mathcal W}}

\def\Z{{\mathbb Z}}
\def\G{{\mathcal G}}
\def\Br{\operatorname{Br}}

\begin{document}

\baselineskip=1.2\baselineskip
 
\title[A Chern-Weil Isomorphism for the Equivariant Brauer Group]{A Chern-Weil 
Isomorphism for the\\ Equivariant Brauer Group}

\author[P Bouwknegt]{Peter Bouwknegt}

\address[P Bouwknegt]{Department of Mathematics, Mathematical Sciences Institute, and
Department of Theoretical Physics, Research School of Physics and Engineering,
Australian National University, Canberra ACT~0200, Australia}
\email{peter.bouwknegt@anu.edu.au}

\author[A Carey]{Alan Carey}
\address[A Carey]{Mathematical Sciences Institute, Australian National University, Canberra
ACT~0200, Australia}
\email{alan.carey@anu.edu.au}

\author[R Ratnam]{Rishni Ratnam}

\address[R Ratnam]{Department of Mathematics, 
Mathematical Sciences Institute, Australian National University, Canberra
ACT~0200, Australia}
\email{rishni.ratnam@anu.edu.au, rishniratnam@yahoo.com.au}

\thanks{This research was supported under Australian Research Council's Discovery
Projects funding scheme (project numbers DP0559415 and DP0878184).}

\begin{abstract}
In this paper we construct a Chern-Weil isomorphism for the equivariant 
Brauer group of $\R^n$ actions on a principal torus bundle, where the target 
for this isomorphism is a ``dimensionally reduced" \v Cech cohomology group. 
{}From this point of view, the usual forgetful functor takes the form of a connecting 
homomorphism in a long exact sequence in dimensionally reduced cohomology.
\end{abstract}

\maketitle

% ==========================================================================
\section{Introduction}

For a second countable locally compact Hausdorff space $X$, we define the 
Brauer group $\Br(X)$ to be the group of $C_0(X)$-isomorphism classes 
of \emph{stable} separable continuous trace algebras with spectrum $X$. 
Dixmier-Douady \cite{DixDou63} showed that this group is isomorphic to the 
sheaf cohomology group $\check{H}^2(X,\mathcal{S})$ {{}(where $\mathcal{S}$ 
is the sheaf of germs of continuous $\mathbb T$-valued functions, and 
$\mathbb T = \mathbb R/\mathbb Z$ is the unit circle).} Given an 
action of a locally compact group $G$ on $X$, one can extend the Dixmier-Douady 
result to outer conjugacy classes of pairs $(A(X),\alpha)$, where $A(X)$ is a separable 
continuous trace C*-algebra with spectrum $X$, and $\alpha$ is a $G$-action 
on $A(X)$ that covers the $G$-action on $X$ \cite{CroKumRaeWil97}. 
Importantly, if the pairs $(A(X),\alpha)$ and $(B(X),\beta)$ are outer conjugate then 
they will have isomorphic crossed products $A(X)\rtimes_\alpha G$ and 
$B(X)\rtimes_\beta G$ \cite{RaeWill98}.
The group of outer conjugacy classes of pairs $(A(X),\alpha)$ 
is known as the equivariant Brauer group $\Br_G(X)$.

There has been a resurgence of interest in this equivariant Brauer group as a result
of applications to T-duality in string theory. {In particular, it was shown in 
\cite{MatRos05, MatRos06} that the equivariant Brauer group, modulo an imprecise 
homotopy equivalence, provides a natural setting for computing the T-duals of principal 
$\T^n$-bundles with H-flux, even when there does not exist a so-called ``classical" T-dual bundle}.
 Recently in \cite{Tu}  an isomorphism between $\Br_G(X)$ and a certain 
\v Cech-type cohomology group $\check{H}^2(G\ltimes X,\mathcal{S})$ for the 
transformation-group groupoid $G\ltimes X$ was established as an analogue of the 
Dixmier-Douady result. Our aim is to extend \cite{Tu} to shed  new light on the structure of $\Br_G(X)$.

The structure of  $\Br_G(X)$ is well known to be very  difficult to analyse except in special cases. 
For example, when $G$ acts freely and properly on $X$ we have 
$\Br_G(X)\cong \check{H}^2(G\backslash X,\mathcal{S})$. At the other extreme, when $G$ acts trivially on $X$:
\[\Br_G(X)\cong \check{H}^2(X,\mathcal{S})\oplus \check{H}^1(X,\mathcal{G}_{ab})
\oplus C(X,H^2_M(G,\T)),\]
where $\mathcal{G}_{ab}$ is the sheaf of germs of continuous $G_{ab}$-valued functions 
(and $G_{ab}$ is the abelianisation of $G$), and $H_M$ is Moore's cohomology for groups 
\cite{Moo76}. When there is a subgroup $N\subset G$ such that $N$ acts trivially on $X$, and 
the $G$ action induces a principal $G/N$-bundle structure $X\to G\backslash X$  progress was 
made in \cite{RaeWil93ddclasses} and \cite{PacRaeWill96}, where the subgroup of $\Br_G(X)$ 
with trivial Mackey obstruction was shown to be isomorphic to the degree 2 cohomology of a 
two-column double complex. This is still insufficient for the applications to T-duality which study 
the case where $\pi:X\to Z$ is a principal $\R^n/\Z^n$ bundle.

Now, it is well known that such principal bundles are classified by the cohomology 
group $\check{H}^2(Z,\underline{\Z}^n)$, and we call the cohomology class associated 
to $\pi:X\to Z$ the \emph{Euler vector}. In this paper we provide an isomorphism from all 
of $\Br_{\R^n}(X)$ to the degree 2 cohomology of a three column complex. 
This is not an obvious generalisation of \cite{RaeWil93ddclasses, PacRaeWill96}. 
The innovation is in the fact that we do not use the ordinary horizontal and vertical 
differentials to constitute the total differential. Instead, we mimic the $E_2$ page of the 
Leray-Serre spectral sequence by combining the vertical differential with the cup product 
with the Euler vector.

To assist the reader we summarise in Sections 2, 3, 4, and 6, respectively,
the equivariant Brauer group (and the case arising in $T$-duality), dimensionally reduced
cohomology, Tu's groupoid cohomology and the Raeburn-Williams equivariant cohomology.
Extensions of these last two papers needed for our results are in Sections 5 and 6
and the main theorem is proved in Section 8.
We explain the main result of the paper in detail in the next subsection. Some results have also 
appeared in \cite{BouCarRat1}.

% ----------------------------------------------------------------------------------------------------------------------------------------
\subsection{The main theorem}

We begin by recalling the Raeburn-Williams ``equivariant cohomology" from 
\cite{RaeWil93ddclasses, PacRaeWill96}. Consider the case where $G$ is a second 
countable locally compact Hausdorff abelian group acting on a locally compact space 
$X$ with orbit space $Z=G\backslash X$. Suppose also that $N\subset G$ is a closed subgroup 
such that $G\to N$ and $\hat{G}\to\hat{N}$ have local sections, and $\pi:X\to Z$ is a 
principal $G/N$-bundle. As the Dixmier-Douady theorem implies there is an isomorphism 
$\Br_{G}(X)\cong \check{H}^3(X,\underline{\Z})$, we denote by $CT(X,\delta)$ the unique 
(up to $C_0(X)$-isomorphism) stable separable continuous trace C*-algebra with 
\emph{Dixmier-Douady class} $\delta\in\check{H}^3(X,\underline{\Z})$. This definition implies 
that for any element $[CT(X,\delta),\alpha]\in\Br_G(X)$ there is an open cover $\{U_{\lambda_0}\}$ 
of $X$ and isomorphisms $\Phi_{\lambda_0}:CT(X,\delta)|_{U_{\lambda_0}}\to 
C_0(U_{\lambda_0},\K)$. It follows by \cite[Prop 2.1]{EchNes01} that 
$\Phi_{\lambda_0}\circ\alpha|_N\circ\Phi_{\lambda_0}$ is locally inner. 
In \cite{RaeWil93ddclasses} and \cite{PacRaeWill96} the authors develop a 
cohomology theory $H^k_{G}(X,{\mathcal S})$ that in degree 2 is isomorphic to 
the subgroup of $\mbox{Br}_G(X)$ such that the restriction of the actions to $N$ are 
locally \emph{unitary}. Note that, the failure of $\Phi_{\lambda_0}\circ\alpha|_N
\circ\Phi_{\lambda_0}$ to be locally unitary at $x\in X$ is measured by a class 
$[M(\cdot,\cdot,x)]$ in $H_M^2(N,\T)$, which is independent of $\lambda_0$ and 
constant on orbits of $X$. We then obtain a map $\operatorname{M}:\Br_G(X)\mapsto 
C(G\backslash X,H_M^2(N,\T))$, called the \emph{Mackey obstruction map}, which has trivial image 
if and only if the restriction of $\alpha$ to $N$ is locally unitary for all $x\in X$. 
Packer, Raeburn and Williams call such systems \emph{$N$-principal}.

Let $\U=\{U_{\lambda_0}\}$ be an open cover of $X$ by $G$-invariant sets. 
We define a two column cochain complex $C^{kj}_G(\U,{\mathcal S})$ as follows:
$$
C^{k0}_G(\U,{\mathcal S}):=\check{C}^k(\U,{\mathcal S}),\quad
C^{(k-1)1}_G(\U,{\mathcal S}):= Z_G^1(G,\check{C}^{k-1}(\U,{\mathcal S})),
$$
where $Z_G^1(G,\check{C}^{k-1}(\U,{\mathcal S}))$ denotes the set of \emph{continuous} group
cohomology 1-\emph{cocycles} from $G$ into the $G$-module $\check{C}^{k-1}(\U,{\mathcal S})$ 
with the obvious $G$ action. In other words, an element $\eta\in C^{(k-1)1}_G(\U,{\mathcal S})$ 
is a collection of continuous functions $\eta_{\lambda_0\dots\lambda_{k-1}}: 
G\times U_{\lambda_0\dots\lambda_{k-1}}\to \T$ such that, for all $g_0,g_1\in G$ 
and $x\in U_{\lambda_0\dots\lambda_{k-1}}$, the following holds:
\[\eta_{\lambda_0\dots\lambda_{k-1}}(g_1,x)\eta_{\lambda_0\dots
\lambda_{k-1}}(g_0g_1,x)^*\eta_{\lambda_0\dots\lambda_{k-1}}(g_0,g_1^{-1}x)=1.\]
This complex has as horizontal differential the usual group cohomology differential 
$\partial_G$, and as vertical differential the usual \v{C}ech differential $\check{\partial}$. 
The cohomology group $H^k_{G}(\U,{\mathcal S})$ is then the cohomology 
$Z^k_{G}(\U,{\mathcal S})/B^k_{G}(\U,{\mathcal S})$ of the total complex
\[\left(\bigoplus_{j=0}^1 C^{(k-j)j}_G(\U,{\mathcal S}), \check\partial\oplus (-1)^{k-j}\partial_G\right).\]
After showing refinement maps induce canonical maps on cohomology 
\cite[Sect 2]{PacRaeWill96}, the authors define
$H^k_{G}(X,{\mathcal S}):=\varinjlim_{\U}H^k_{G}(\U,{\mathcal S}).$

The importance of these groups to our work is in the following two results:

\begin{lemma}[{\cite[Lemma 1.3]{PacRaeWill96}}]\label{IsotoKerM}
Suppose that $G$ is a locally compact abelian group, and $N$ a closed 
subgroup such that $\hat{G}\to\hat{N}$ has local sections. Suppose $\pi:X\to Z$ 
is a locally trivial principal $G/N$-bundle over a paracompact space $Z$ and let 
$\operatorname{M}:\operatorname{Br}_G(X)\to C(G\backslash X,H^2_M(N,X))$ denote the 
Mackey Obstruction map. Then $H_G^2(X,{\mathcal S})\cong \ker \operatorname{M}$.
\end{lemma}

In order to generalise Lemma \ref{IsotoKerM} to all of $\Br_G(X)$, one may choose to 
study a three-column complex defined by setting
$C^{k0}_G(\U,{\mathcal S}):=\check{C}^k(\U,{\mathcal S}),$ and
$$C^{(k-1)1}_G(\U,{\mathcal S}):= C_G^1(G,\check{C}^{k-1}(\U,{\mathcal S})),\quad
C^{(k-2)2}_G(\U,{\mathcal S}):= Z_G^2(G,\check{C}^{k-2}(\U,{\mathcal S})).$$
However, such a complex is unable to go beyond $\ker\operatorname{M}$, 
because the horizontal differential requires $\U$ to consist of $G$-invariant sets, and 
\cite[Cor 5.18]{PacRaeWill96} \footnote{The corollary is incorrect as the isomorphism
claimed there is only a surjection.} implies continuous trace algebras trivialisable 
over $G$-invariant sets have trivial Mackey obstruction.

On the other hand, the following theorem provides the appropriate direction:
\begin{theorem}[{\cite[Thm 4.1]{PacRaeWill96}}]\label{PRWGysin}
Let $G$ be a locally compact abelian group and $N$ a closed subgroup such that 
$G\to N$ and $\hat{G}\to\hat{N}$ have local sections. Let ${\mathcal N}$ and $\mathcal{\hat{N}}$ 
denote the sheaves of germs of continuous $N$ and $\hat{N}$-valued functions respectively. 
Then for any principal $G/N$-bundle $\pi:X\to Z$ with Euler vector 
$c\in \check{H}^2(Z,{\mathcal N})$ there is a long exact sequence
\[\dots\to \check{H}^k(Z,{\mathcal S})\stackrel{\pi^*_G}{\to} 
H^k_G(X,{\mathcal S})\stackrel{\pi_*}{\to} \check{H}^{k-1}
(Z,\mathcal{\hat{N}})\stackrel{\cup c}{\to}\check{H}^{k+1}(Z,{\mathcal S})\to\dots\]
The sequence starts with $\check{H}^1(Z,{\mathcal S})$, and $\pi^*_G:
\check{H}^1(Z,{\mathcal S})\to \check{H}^1_G(X,{\mathcal S})$ is injective.
\end{theorem}

The maps $\pi^*_{G}, \pi_*$ and $\cup c$ of the theorem are defined as follows.
If $[\phi]\in  \check{H}^k(Z,{\mathcal S})$ then $\pi^*_{G}[\phi]=[\pi^*(\phi),1]$, where $\pi^*$ 
is the  pullback $\pi^*: \check{Z}^k(Z,{\mathcal S})\to \check{Z}^k(X,{\mathcal S})$. 
To define the ``integration" map $\pi_*$, let $(\nu,\eta)\in \check{Z}^k_{G}(\pi^{-1}(\W),
{\mathcal S})$ for some open cover $\W$ of $Z$. Now a calculation, exploiting the fact 
that $(\nu,\eta)$ is a cochain and $G$ is abelian, implies for any $t\in G$, $m\in N$ 
and $x\in X$ that
\begin{equation}\label{etainvariantonorbits}
\eta_{\mu_0\dots\mu_{k-1}}(m,-t\cdot x)=\eta_{\mu_0\dots\mu_{k-1}}(m,x).
\end{equation}
Then, using the cocycle property of  $(\nu,\eta)$, we find
$
\check\partial \eta(m,\cdot)_{\mu_0\dots\mu_{k-1}}(x)=1.
$
Therefore we may define $\pi_*(\eta)\in \check{Z}^{k-1}(\W,\hat{{\mathcal N}})$ by
\[\pi_*(\eta)(z)_{\mu_0\dots\mu_{k-1}}(m):=\eta_{\mu_0\dots\mu_{k-1}}(m,x),\quad \pi(x)=z.\]
For the last map $\cup c:\check{H}^{k-1}(Z,\hat{{\mathcal N}})\to
\check{H}^{k+1}(Z,{\mathcal S})$, let $\psi\in \check{Z}^{k-1}(\W,\hat{{\mathcal N}})$, 
and choose a representative $F\in\check{Z}^2(\W,\mathcal{N})$ of $c$ (this may 
require taking a common refinement). Then $[\psi]\cup c$ is by definition the class with representative
\[W_{\mu_0\dots\mu_{k+1}}\ni z\mapsto (-1)^k\psi_{\mu_0\dots\mu_{k-1}}(F_{\mu_{k-1}
\mu_{k}\mu_{k+1}}(z),z).\]
This is well-defined because one can show the image of $[\psi]$ under $\cup c$ is 
independent of the choice of representatives $\psi$ and $F$.

It is the definition of these maps and the exactness of the sequence from 
Theorem \ref{PRWGysin} that tells us how to proceed. Indeed 
\cite[Lemma 4.2]{PacRaeWill96} gives a function $\mu:Z^k_G(\pi^{-1}(\W),\mathcal{S})\to 
\check{C}^{k}(\W,\mathcal{S})$ such that for any cocycle $(\nu,\eta)\in Z^k_G(\pi^{-1}(\W),\mathcal{S})$ 
we have $\check{\partial}[\mu(\nu,\eta)]=\pi_*(\eta)\cup F$. We can then define a two column complex
\[C^{k0}_F(\W,{\mathcal S}):=\check{C}^k(\W,{\mathcal S}),\quad
C^{(k-1)1}_F(\W,{\mathcal S}):= \check{C}^{k-1}(\W,\hat{{\mathcal N}})\]
with differential $\check{\partial}_F(\phi^{k0},\phi^{(k-1)1})=(\check\partial \phi^{k0}+(-1)^{k+1}\phi^{(k-1)1}
\cup F,\check{\partial}\phi^{(k-1)1})$ and cohomology $H^k_F(\W,\mathcal{S})$. If we choose $\W$ to 
be ``good", the Five Lemma shows that the map $[\nu,\eta]\mapsto [\mu(\nu,\eta),\pi_*(\eta)]$ is an 
isomorphism of $\check{H}^k_G(Z,\mathcal{S})$ with $H^k_F(\W,\mathcal{S})$. To accommodate 
non-trivial Mackey obstructions then, we need to extend $H^k_F(\W,\mathcal{S})$ 
to a three column complex.

We denote by $\check{H}^3(X,\underline{\Z})|_{\pi^{0,3}=0}$ the kernel of the Leray-Serre spectral 
sequence projection $\pi^{0,3}:\check{H}^3(X,\underline{\Z})\to C(Z,\check{H}^3(\T^n,\underline{\Z}))$. 
We now state our main theorem, which applies to the case $G=\R^n$, $N=\Z^n$.

\begin{theorem}\label{MainSquare}
Let $\pi:X\to Z$ be a $C^\infty$ principal $\T^n$-bundle over a Riemannian manifold $Z$. 
Then there exists an open cover $\U$ of $X$, and a cocycle 
$F\in \check{Z}^2(\pi(\U),\underline{\Z}^n)$ such that
\begin{itemize}
\item[(1)] the image of $[F]\in \check{H}^2(\pi(\U),\underline{\Z}^n)$ in $\check{H}^2(Z,\underline{\Z}^n)$ 
is the Euler vector of  $\pi:X\to Z$; and
\item[(2)] every stable continuous trace C*-algebra over $X$ is trivialised over $\U$.\\
Moreover, for all $k\geq 0$ there exist groups $\mathbb{H}^k_{F}(\pi(\U),\mathfrak{G})$, being the 
cohomology of a three column complex, where $\mathfrak{G}$ is either the sheaf $\mathcal{S}$ 
or $\underline{\Z}$, such that there is a commutative diagram
\end{itemize}\medskip

\centerline{\xymatrix{
\operatorname{Br}_{\R^n}(X)\ar[r]\ar[d]^{\cong}&\check{H}^3(X,\underline{\Z})|_{\pi^{0,3}=0}\ar[d]^{\cong}\\
\mathbb{H}^2_{F}(\pi(\U),\mathcal{S})\ar[r]&\mathbb{H}^3_{F}(\pi(\U),\underline{\Z}).
}}
\end{theorem}

There is actually substantially more motivation for the above theorem than just 
\cite{PacRaeWill96}. Indeed, the nature of the cohomology groups 
$\mathbb{H}^k_{F}(\pi(\U),\underline{\Z})$ and the right vertical isomorphism 
was predicted by the de Rham cohomology version of Theorem \ref{PRWGysin} 
contained in \cite{BouHanMat05}.

% ==========================================================================
\section{Preliminaries}\label{Mackeysectionchapter0}

\subsection[Continuous Trace C*-Algebras]{Continuous Trace C*-Algebras}\label{ctstracesection}

In this paper we are only concerned with separable and stable algebras and so our discussion 
is a greatly restricted version of the usual one for which the reader should consult \cite{RaeWill98}.
With $X$ as in the introduction, let $p:E\to X$ be a (locally trivial) vector bundle over $X$ 
with fibre $\K$ (the compact operators
on a separable, infinite dimensional Hilbert space $\H$), and structure group 
$\operatorname{Aut}\K$ (with the point-norm topology).  Let $\Gamma(E,X)$ 
be the $*$-algebra of all continuous sections of $p:E\to X$. Then
\[\Gamma_0(E,X):=\{f\in \Gamma(E,X): x\mapsto ||f(x)|| \mbox{ vanishes at } \infty\}\]
is a $C^*$-algebra with respect to point-wise operations and the sup norm \cite[Prop 4.89]{RaeWill98}.

\begin{defn}
A separable C*-algebra $A$ with {spectrum} 
 $X$ is called \emph{continuous trace} if it is $C_0(X)$-isomorphic to $\Gamma_0(E,X)$, for some 
(locally trivial) vector bundle $p: E\to X$ with fibre $\K$, and structure group $\operatorname{Aut}\K$.
\end{defn}

{{} Now we recall the Brauer group. The product of algebras $A$ and $B$ with spectrum $X$ is given 
by taking the ideal $I_X$ of $A\otimes B$ generated by the set
\[\{(f\cdot a)\otimes b-a\otimes (f\cdot b):f\in C_0(X), a\in A, b\in B\} \,, \]
and defining the \emph{balanced tensor product} $A\otimes_{C_0(X)} B$ by
$A\otimes_{C_0(X)} B:= (A\otimes B)/I_X$.
(Note that  C*-algebras with Hausdorff spectra are nuclear \cite[Cor B.44]{RaeWill98}).
Thus if $p:E\to X$ and $p':E'\to X$ are bundles with fibre $\K$ then
\[\Gamma_0(E,X)\otimes_{C_0(X)} \Gamma_0(E',X)\cong \Gamma_0(E\otimes_X E',X),\]
where $E\otimes_X E'$ is the restriction of the tensor product bundle $E\otimes E\to X\times X$ to 
the diagonal $\{(x,x):x\in X\}$ \cite[Sect 3]{KumMuhRenWil98}. Thus $A\otimes_{C_0(X)} B$ is a 
continuous trace C*-algebra  also with spectrum $X$.}
Then the \emph{Brauer group} $\mbox{Br}(X)$ is the group of $C_0(X)$-isomorphism classes of 
continuous trace C*-algebras with spectrum $X$ where the zero element is the $C_0(X)$-isomorphism 
class of $C_0(X,\K)$ and the group operation is the balanced tensor product.

By \cite[Prop 4.53]{RaeWill98}, isomorphism classes of vector bundles with fibre 
$\K$ and structure group $\operatorname{Aut}\K$ are in one-to-one correspondence 
with $\check{H}^1(X,{\mathcal A})$, where ${\mathcal A}$ is the sheaf of germs of 
continuous $\operatorname{Aut}\K$-valued functions. If we equip the unitary operators 
$U(\H)$ with the strong operator topology, then there is an exact sequence of topological groups
\begin{equation}\label{Uexactsequence}
1\to \T\to U(\H)\to \operatorname{Aut}\K\to 1,
\end{equation}
such that $U(\H)\to \operatorname{Aut}\K$ has local continuous sections \cite[Chapter 1]{RaeWill98}. 
Consequently, with ${\mathcal S}$ as in the introduction, the long exact sequence in sheaf cohomology, 
combined with the fact that $U(\H)$ is contractible (in the strong operator topology) 
(see, e.g., \cite[Thm 4.72]{RaeWill98}) implies that 
$\check{H}^1(X,{{\mathcal A}})\cong \check{H}^2(X,{{\mathcal S}})$. From the exact 
sequence $0\to \Z\to \R\to \T\to 1$,
we then obtain
$\check{H}^1(X,{{\mathcal A}})\cong \check{H}^2(X,{{\mathcal S}})\cong \check{H}^3(X,\underline{\Z}).$
As continuous trace C*-algebras are $C_0(X)$-isomorphic if and only if they are 
$C_0(X)$-isomorphic to the algebra of sections of isomorphic bundles, we obtain the 
\emph{Dixmier-Douady} classification \cite{DixDou63}.
Namely,
if $A$ is a continuous trace C*-algebra with spectrum $X$, $C_0(X)$-isomorphic to
continuous sections of a bundle $p:E\to X$, then the isomorphism $\mathrm{Br}(X)\to 
\check{H}^3(X,\underline{\Z})$ is defined by the map induced by sending $A$ to
its \emph{Dixmier-Douady class}, the image $\delta$, of $p:E\to X$ in $\check{H}^3(X,\underline{\Z})$.
The corresponding class of continuous trace C*-algebras will be denoted $CT(X,\delta)$.
We move on now to the \emph{equivariant} Brauer group. 

\begin{defn}
Fix a locally compact group $G$ and a second countable locally compact 
Hausdorff (left) $G$-space $X$. Denote the induced $G$-action on $C_b(X)$ by $\tau$. That is
$\tau_g(f)(x)=f(g^{-1}x),$ for $ f\in C_b(X).$
Let $A$ be a continuous trace C*-algebra with spectrum $X$, and $\alpha$ an action of 
$G$ on $A$. Then $\alpha$ is said to \emph{preserve the (given) action on the spectrum} 
if for all $g\in G$, $a\in A$ and $f\in C_b(X)$:
$\alpha_g(f\cdot a)=\tau_g(f)\cdot \alpha_g(a).$
\end{defn}
We write $\mathfrak{Br}_G(X)$ for the collection of pairs $(A,\alpha)$, where $A$ is a 
continuous trace C*-algebra with spectrum a $G$-space $X$ and $\alpha$ is an action 
of $G$ on $A$ that preserves the given action on the spectrum.
If $A$ is a C*-algebra, we denote by $M(A)$ and $UM(A)$ the \emph{multiplier algebra} 
of $A$ and unitary elements of $M(A)$, respectively. 
Now, given elements $(A,\alpha)$ and $(A,\beta)$ of $\mathfrak{Br}_G(X)$, we say the 
actions $\alpha$ and $\beta$ are \emph{exterior equivalent} if there is a (strictly) continuous 
map $w:G\to UM(A)$ such that
\begin{align}
\label{exteriorcondition1}
\beta_g(a)&=w_g\alpha_g(a)w_g^* \quad \mbox{ for all } a\in 
A \mbox{ and } g\in G,\\
\label{exteriorcondition2}w_{gh}&=w_g\overline{\alpha}_g(w_h)\quad \mbox{ for all } g,h\in G.
\end{align}
In this case, one says that $w$ is a \emph{unitary} $\alpha$ \emph{cocycle} implementing 
the equivalence. Two pairs $(A,\alpha)$ and $(B,\beta)$ are \emph{outer conjugate} if and 
only if there is a $C_0(X)$-isomorphism $\phi:A\to B$ such that $\alpha$ and 
$\phi^{-1}\circ \beta\circ \phi$ are exterior equivalent.
 {{} If $(A,\alpha)$ and $(B,\beta)$ are in $\mathfrak{Br}_G(X)$ then by 
\cite[Prop B.13]{RaeWill98} there is an action $\alpha\otimes \beta$ of 
$G$ on $A\otimes B$ given by
\begin{equation*}
(\alpha\otimes \beta)_g(a\otimes b)=\alpha_g(a)\otimes \beta_g(b)\,, 
\end{equation*}
and as 
\[(\alpha\otimes \beta)_g(f\cdot a\otimes b -a\otimes f\cdot b)=\tau_g(f)\cdot 
\alpha_g(a)\otimes \beta_g(b) -\alpha_g(a)\otimes \tau_g(f)\cdot \beta_g(b)\]
then $\alpha\otimes \beta$ preserves the ideal $I_X$ of $A\otimes B$, and therefore 
induces an action $\alpha\otimes_X\beta$ of $G$ on $A\otimes_{C_0(X)}B$. 
With the notation of the previous definition we have:
\begin{defn}
The \emph{equivariant Brauer group} $\mbox{Br}_G(X)$  is defined as 
the group of equivalence classes in $\mathfrak{Br}_G(X)$
under outer conjugacy. The zero element is the equivalence class of 
$(C_0(X,\K),\tau)$. The binary operation is
\begin{equation*}
[A,\alpha]\cdot[B,\beta]=[A\otimes_{C_0(X)}B,\alpha\otimes_X \beta] \,,
\end{equation*}
and the inverse of $[A,\alpha]$ is
the conjugate algebra $[\overline{A},\overline{\alpha}]$ (we use the canonical
bijection of sets $\flat:A\to \overline{A}$, to define $\overline{\alpha}_g(\flat(a)):=\flat(\alpha_g(a))$.
\end{defn}
Later on we will use the ``forgetful homomorphism" $F:\mbox{Br}_G(X)\to 
\mbox{Br}(X)$ that sends $[CT(X,\delta),\alpha]$ to $[CT(X,\delta)]$. 
This map is neither surjective nor injective in general.

% ----------------------------------------------------------------------------------------------------------------------------------------
\subsection{Actions of $\Z^n$ on $C_0(X,\K)$.}\label{Mackeysectiononeoneone}

With $X$  as above we now review the
 Mackey obstruction map for actions $\alpha$ of $\Z^n$ on $C_0(X,\K)$.
We call $\alpha$ \emph{locally inner} if there exists an open cover $\{U_{\lambda_0}\}_{\lambda_0\in 
\mathcal{I}}$ of $X$ and functions $u_{\lambda_0}:\Z^n\to C(U_{\lambda_0},U(\H))$ such that
\[(\alpha_{m}h)(x)=u_{\lambda_0}(m,x)h(x)u_{\lambda_0}(m,x)^*,\quad h\in C_0(X,\K),\ m\in \Z^n.\]
and we note from {\cite[Prop 2.1]{EchNes01}}  that every action of $\Z^n$ on $C_0(X,\K)$ locally inner.

Fix generators $e_i$ of $\Z^n$, and let the action of $e_i$ at $x\in U_{\lambda_0}$ be 
implemented by $u_{\lambda_0}^i(x) = u_{\lambda_0}(e_i,x)$:
\begin{equation}\label{implemented}
(\alpha_{e_i}h)(x)=u_{\lambda_0}^i(x)h(x)u_{\lambda_0}^i(x)^*.
\end{equation}
For $m,l\in \Z^n$, let $m_i$ denote the $i^{th}$ component of $m$, and let $\T$ be the centre 
of $U(\H)$. If we define
$u^m_{\lambda_0}(x):=(u_{\lambda_0}^1(x))^{m_1}(u_{\lambda_0}^2(x))^{m_2}\dots 
(u_{\lambda_0}^n(x))^{m_n}$,
then there is a function $M_{\lambda_0}:\Z^n\times \Z^n\times U_{\lambda_0}\to\T$ defined by
\begin{equation}\label{mdefn}
u^l_{\lambda_0}(x)u^m_{\lambda_0}(x)=M_{\lambda_0}(m,l,x)u^{m+l}_{\lambda_0}(x).
\end{equation}
An easy calculation shows that $M_{\lambda_0}$ satisfies the (Moore cohomology 
\cite{Moo76}) cocycle identity
\[M_{\lambda_0}(l,k,x)M_{\lambda_0}(m+l,k,x)^*M_{\lambda_0}(m,l+k,x)M_{\lambda_0}(m,l,x)^*=1.\]
Thus the map $x\mapsto ((m,l)\mapsto [M_{\lambda_0}(m,l,x)])$ defines an element of
$\prod_{\lambda_0} C(U_{\lambda_0},H_M^2(\Z^n,\T)).$
For two sets $U_{\lambda_0}$ and $U_{\lambda_1}$, the unitary $u_{\lambda_0}^i(x)$ 
differs from $u_{\lambda_1}^i(x)$ by an element of $\T$, from which it follows that 
$x\mapsto ((m,l)\mapsto [M_{\lambda_0}(m,l,x)])$ defines a \emph{global} continuous 
function in $C(X,H_M^2(\Z^n,\T))$. If $M_n^u(\T)$ is the group (under addition of matrices) 
of $n\times n$ strictly upper triangular matrices with entries in $\T$, then a result from 
\cite{Klep65} says $H_M^2(\Z^n,\T)\cong M_n^u(\T)$. Thus we have defined a continuous 
function $f:X\to M_n^u(\T)$ called the \emph{Mackey obstruction} of $(C_0(X,\K),\alpha)$. 
The map $\operatorname{M}:(C_0(X,\K),\alpha)\to f$ is called the \emph{Mackey obstruction map}.

It is possible to extract the function $f$ from the cocycle $M$. To see this 
define $f_{\lambda_0}:U_{\lambda_0}\to \T$ by\footnote{We thank Iain Raeburn 
for communicating this definition.}
\begin{equation}\label{fdefn}
f_{\lambda_0}(x)_{ij}u_{\lambda_0}^i(x)u_{\lambda_0}^j(x)=u_{\lambda_0}^j(x)u_{\lambda_0}^i(x)\,,
\qquad i\leq j\,.
\end{equation}
Since $f_{\lambda_0}$ is independent of $\lambda_0$ (for the same reason that $M$ gave a 
global function), we have a well-defined map $f:X\to M_n^u(\T)$ with $ij^{th}$ entry
$f(x)_{ij}=f_{\lambda_0}(x)_{ij}, \ x\in U_{\lambda_0}.$
{}From the definitions one may see that
\begin{equation}\label{mfequation}
M_{\lambda_0}(m,l,x)=\prod_{1\leq i<j\leq n} f_{\lambda_0}(x)_{ij}^{m_il_j}.
\end{equation}
Therefore $M_{\lambda_0}$, defined as in Eqn.~(\ref{mdefn}), is independent of $\lambda_0$ 
even at the level of cocycles (as opposed to Moore cohomology classes). In particular, 
if $f$ satisfies $f(x)_{ij}=1$ for all $x\in X$ then
$u^l_{\lambda_0}(x)u^m_{\lambda_0}(x)=u^{m+l}_{\lambda_0}(x)$.
\begin{remark}
That $f$ and $M$ are independent of the choice of $u_{\lambda_0}:\Z^n\to C(U_{\lambda_0},U(\H))$  
follows respectively from Equation (\ref{fdefn}) and Equation (\ref{mfequation}).
\end{remark}

% ----------------------------------------------------------------------------------------------------------------------------------------
\subsection{Actions of $\R^n$ on a Continuous Trace C*-algebra over a Principal $\T^n$-bundle}
\label{Mackeysectiononeonetwo}

In the application we are primarily concerned with, the action of $\Z^n$ comes from the restriction of 
an action of $\R^n$ on $CT(X,\delta)$ which covers the fibre action for a principal $\T^n$-bundle 
$\pi:X\to Z$. This data gives an element $[CT(X,\delta),\alpha]$ of the equivariant Brauer group 
$\mbox{Br}_{\R^n}(X)$.

In this case, we proceed as above by first choosing an open cover 
$\{U_{\lambda_0}\}_{\lambda_0\in \mathcal{I}}$ of $X$ so that there are isomorphisms
$\Phi_{\lambda_0}:CT(X,\delta)\vline_{U_{\lambda_0}}
\stackrel{\cong}{\longrightarrow} C_0(U_{\lambda_0},\K).$
These isomorphisms, together with the exact sequence (\ref{Uexactsequence}),  
imply that there exist functions $v_{\lambda_0\lambda_1}\in C(U_{\lambda_0\lambda_1},U(\H))$ 
such that $\Phi_{\lambda_1}\circ \Phi_{\lambda_0}^{-1}=\operatorname{Ad} v_{\lambda_0\lambda_1}$, 
and a representative for $\delta$ is given by the image of the map
$x\mapsto v_{\lambda_1\lambda_2}(x)v_{\lambda_0\lambda_2}(x)^*v_{\lambda_0\lambda_1}(x)$
under the connecting isomorphism $\Delta:\check{H}^2(X,{{\mathcal S}})\to \check{H}^3(X,\underline{\Z})$. 
Since $\Z^n$ acts trivially on the spectrum $X$, the actions 
$\Phi_{\lambda_0}\circ\alpha|_{\Z^n}\circ\Phi^{-1}_{\lambda_0}$ on $C_0(U_{\lambda_0},\K)$ 
are locally inner. Refine, if necessary, the sets $\{U_{\lambda_0}\}$ so that there exists 
$u_{\lambda_0}^{i}\in C(U_{\lambda_0},U(\H))$ that implements the actions:
\begin{equation}\label{chapter0eq0}
\left(\Phi_{\lambda_0}\circ\alpha_{e_i}\circ\Phi^{-1}_{\lambda_0}\right)(h)(x)=
u_{\lambda_0}^i(x)h(x)u_{\lambda_0}^i(x)^* \quad h\in C_0(U_{\lambda_0},\K),
\end{equation}
(cf. (\ref{implemented}) above, where the local isomorphisms $\Phi_\bullet$ are restrictions). 
We define the Mackey obstruction at $x\in X$ as follows. Since (\ref{chapter0eq0}) is equivalent to
\[(\Phi_{\lambda_0}\circ\alpha_{e_i}(a))(x)=u_{\lambda_0}^i(x)\Phi_{\lambda_0}(a)(x)
u_{\lambda_0}^i(x)^*,  \quad a\in CT(X,\delta),\]
if we define for any $m\in \Z^n$,
$u^m_{\lambda_0}(x):=(u_{\lambda_0}^1(x))^{m_1}(u_{\lambda_0}^2(x))^{m_2}
\dots (u_{\lambda_0}^n(x))^{m_n}$,
the fact that $\alpha$ is a homomorphism implies
\begin{equation}\label{alphahomomorphism}
(\Phi_{\lambda_0}\circ\alpha_{m}(a))(x)=u_{\lambda_0}^m(x)\Phi_{\lambda_0}(a)(x)
u_{\lambda_0}^m(x)^*, \quad a\in CT(X,\delta).
\end{equation}
Then, the Mackey obstruction at $x$ is  the class $[M_{\lambda_0}(\cdot,\cdot,x)]\in H^2(\Z^n,\T)$ given by
\begin{equation}\label{fullmackey}
u^l_{\lambda_0}(x)u^m_{\lambda_0}(x)=M_{\lambda_0}(m,l,x)u^{m+l}_{\lambda_0}(x).
\end{equation}
Now, we may still define
$f_{\lambda_0}(x)_{ij}u_{\lambda_0}^i(x)u_{\lambda_0}^j(x)=u_{\lambda_0}^j(x)u_{\lambda_0}^i(x),$
but then it is not obvious that $f$ is a global function on $X$. In order to prove that it is, 
we present a simple lemma that we will use repeatedly  throughout this paper.
\begin{lemma}\label{cyclicpermute}
Let $k\in\N$ and $u_1,u_2\dots,u_k\in U(\H)$ be unitaries such that 
$u_1u_2\dots u_k\in ZU(\H)\cong \T$. Then
$u_1u_2\dots u_k=u_ku_1u_2\dots u_{k-1}.$
\end{lemma}
\begin{proof}
Suppose that $t\in\T$ is such that $t=u_1u_2\dots u_k$. Then
\[u_1u_2\dots u_k=t=u_ktu_k^*=u_ku_1u_2\dots u_{k-1}.\]
\end{proof}
\begin{prop}
The function $x\mapsto f_{\lambda_0}(x)_{ij}$ is independent of $\lambda_0$.
\end{prop}
\begin{proof}
One may see that for any $\lambda_0,\lambda_1\in \mathcal{I}$ and 
appropriate sections in $CT(X,\delta)\vline_{U_{\lambda_0\lambda_1}}$
\begin{align*}
&\Phi_{\lambda_1}^{-1}\circ \operatorname{Ad}u_{\lambda_1}^i\circ \Phi_{\lambda_1}=
\Phi_{\lambda_0}^{-1}\circ \operatorname{Ad}u_{\lambda_0}^i\circ \Phi_{\lambda_0}
\implies  \Phi_{\lambda_1}\circ\Phi_{\lambda_0}^{-1}\circ \operatorname{Ad}u_{\lambda_0}^i\circ 
\Phi_{\lambda_0}\circ \Phi_{\lambda_1}^{-1}=\operatorname{Ad}u_{\lambda_1}^i.
\end{align*}
This means that, if $h$ is a section in $C_0(U_{\lambda_0\lambda_1},\K)$ then
\begin{align*}
\left(\Phi_{\lambda_1}\circ\Phi_{\lambda_0}^{-1}\circ \operatorname{Ad}
u_{\lambda_0}^i\circ \Phi_{\lambda_0}\circ \Phi_{\lambda_1}^{-1}\right)(h)(x)
&=v_{\lambda_0\lambda_1}(x)u_{\lambda_0}^i(x)v_{\lambda_0\lambda_1}^*(x)h(x)
v_{\lambda_0\lambda_1}(x)u_{\lambda_0}^i(x)^*v_{\lambda_0\lambda_1}^*(x)\\
&=u_{\lambda_1}^i(x)h(x)u_{\lambda_1}^i(x)^*.
\end{align*}
Hence there are continuous functions $\eta_{\lambda_0\lambda_1}^i:U_{\lambda_0\lambda_1}\to \T$ satisfying
\begin{equation*}
u_{\lambda_1}^i(x):=\eta_{\lambda_0\lambda_1}^i(x)v_{\lambda_0\lambda_1}(x)
u_{\lambda_0}^i(x)v_{\lambda_0\lambda_1}(x)^*.
\end{equation*}
Therefore, using Lemma \ref{cyclicpermute} between the second and third equality, we have
\begin{align*}
f_{\lambda_1}(x)_{ij}&=u_{\lambda_1}^j(x)^*u_{\lambda_1}^i(x)^*u_{\lambda_1}^j(x)u_{\lambda_1}^i(x)\\
%\end{align*}\begin{align*}
&=\eta_{\lambda_0\lambda_1}^j(x)^*v_{\lambda_0\lambda_1}(x)u_{\lambda_0}^j(x)^*v_{\lambda_0\lambda_1}(x)^*
\eta_{\lambda_0\lambda_1}^i(x)^*v_{\lambda_0\lambda_1}(x)u_{\lambda_0}^i(x)^*v_{\lambda_0\lambda_1}(x)^*\\
&\quad\times\eta_{\lambda_0\lambda_1}^j(x)v_{\lambda_0\lambda_1}(x)u_{\lambda_0}^j(x)
v_{\lambda_0\lambda_1}(x)^*
\eta_{\lambda_0\lambda_1}^i(x)v_{\lambda_0\lambda_1}(x)u_{\lambda_0}^i(x)
v_{\lambda_0\lambda_1}(x)^*\\%\end{align*}\begin{align*}
&=u_{\lambda_0}^j(x)^*u_{\lambda_0}^i(x)^*u_{\lambda_0}^j(x)u_{\lambda_0}^i(x)
=f_{\lambda_0}(x)_{ij}.
\end{align*}
\end{proof}
It follows that
$M_{\lambda_0}(m,l,x)=\prod_{1\leq i<j\leq n} f_{\lambda_0}(x)^{m_il_j}_{ij},$
is independent of $\lambda_0$, and therefore globally continuous. Hence, we obtain a map
$\operatorname{M}:\mathfrak{Br}_{\R^n}(X)\to C(X, M_n^u(\T)),$ from 
$M:(CT(X,\delta),\alpha)\mapsto f.$
It is a standard result \cite[Sect 1]{PacRaeWill96} that $\operatorname{M}$ 
is constant on outer conjugacy classes and constant on orbits of $X$, and hence
defines a group homomorphism (denoted by the same symbol)
$\operatorname{M}:\operatorname{Br}_{\R^n}(X)\to C(Z, M_n^u(\T)).$
In this case, we may write

\begin{equation}\label{mfequation2}
M_{\lambda_0}(m,l,x)=M(m,l,\pi(x))=\prod_{1\leq i<j\leq n} f(\pi(x))^{m_il_j}_{ij}.
\end{equation}
As before the map $\operatorname{M}:[C_0(X,\K),\alpha]\to f$ is called 
the \emph{Mackey obstruction map}, and the function $f$ is called the 
\emph{Mackey obstruction} of $[CT(X,\delta),\alpha]$.

% ==========================================================================
\section{Dimensionally Reduced Cohomology}\label{DimRedCohoSection}
{{}
In this Section we recall the ``dimensionally reduced cohomology" groups of \cite{BouRat1}.  
In degree 2 with ${\mathcal S}$ coefficients they are isomorphic to the equivariant 
Brauer groups $\operatorname{Br}_{\R^n}(X)$, for $X$ a principal $\T^n$-bundle (see 
Theorem \ref{MainSquare}).}

\begin{lemma}[{\cite[Sect 2]{Ste47}}]\label{commutecoboudary}
Let $Z$ be a topological space with an open cover $\W$ and $A,B\in \check{Z}^2(\W,\underline{\Z})$. Then
$A\cup B-B\cup A=\check\partial C$
where
\[C_{\lambda_0\lambda_1\lambda_2\lambda_3}(z):= A_{\lambda_0\lambda_1\lambda_2}
(z)B_{\lambda_0\lambda_2\lambda_3}(z)-A_{\lambda_1\lambda_2\lambda_3}
(z)B_{\lambda_0\lambda_1\lambda_3}(z).\]
\end{lemma}
{{}
Let $Z$ be a $C^\infty$ manifold and fix an open cover ${{\mathcal W}}=\{W_{\mu_0}\}$ of $Z$ 
together with a cocycle $F\in \check{Z}^2({{\mathcal W}},\underline{\Z}^n)$. 
Let $\mathcal{S}, \hat{\mathcal{N}}$ and $\mathcal{M}$ denote the sheaves of 
germs of continuous $\T,\hat{\Z}^n$ and $M_n^u(\T)$-valued functions respectively. 
We can think of a \v{C}ech cocycle in $\phi^{(k-2)2}\in \check{C}^{k-2}(\W,{\mathcal M})$ 
as an $n \choose 2$-tuple
$\{\phi^{(k-2)2}(\cdot)_{ij}\}_{1\leq i<j\leq n},$
where $\phi^{(k-2)2}(\cdot)_{ij}\in \check{C}^{k-2}(\W,\mathcal{S})$.

We define a cochain complex
$C^k_F({{\mathcal W}},{{\mathcal S}}),$
for $k\geq 2$ as all triples $(\phi^{k0},\phi^{(k-1)1},\phi^{(k-2)2})$ consisting of 
\v{C}ech cochains $\phi^{k0}\in \check{C}^{k}(\W,\mathcal{S})$, $\phi^{(k-1)1}\in 
\check{C}^{k-1}(\W,\hat{\mathcal{N}})$ and $\phi^{(k-2)2}\in \check{C}^{k-2}(\W,\mathcal{M})$. 
When $k=1$, we define a cochain to be a pair $(\phi^{10},\phi^{01
})$, where $\phi^{10}\in \check{C}^{1}(\W,\mathcal{S})$, $\phi^{01}\in 
\check{C}^{0}(\W,\hat{\mathcal{N}})$, whilst when $k=0$ a cochain is an 
element  $\phi^{00}\in \check{C}^{0}(\W,\mathcal{S})$.

Now, for any $A\in \check{C}^2(\W,\underline{\Z}^n)$ and $B\in 
\check{C}^3(\W,\underline{M_n^u(\Z)})$ we can define products
\begin{align*}
\cup_1 A: \check{C}^{k-1}(\W,\hat{\mathcal{N}})\to \check{C}^{k+1}(\W,\mathcal{S}),\quad
&\cup_1 A: \check{C}^{k-2}(\W,\mathcal{M}) \to \check{C}^{k}(\W,\hat{\mathcal{N}}),\quad \mbox{and}\\
\cup_2 B: \check{C}^{k-2}(\W,\mathcal{M}) &\to \check{C}^{k+1}(\W,\mathcal{S}),
\end{align*}
by the formulas}
\begin{align*}
(\phi^{(k-1)1}\cup_1 A)_{\lambda_0\dots\lambda_{k+1}}(z):=&
\phi^{(k-1)1}_{\lambda_0\dots\lambda_{k-1}}(A_{\lambda_{k-1}\lambda_k\lambda_{k+1}}(z),z),\\
(\phi^{(k-2)2}\cup_1 A)_{\lambda_0\dots\lambda_{k+1}}(m,z):=&\prod_{1\leq i<j\leq n}
\phi^{(k-2)2}_{\lambda_0\dots\lambda_{k-1}}(z)_{ij}^{A_{\lambda_{k-1}\lambda_k
\lambda_{k+1}}(z)_i(m)_j-(m)_iA_{\lambda_{k-1}\lambda_k\lambda_{k+1}}(z)_j},\\
(\phi^{(k-2)2}\cup_2 B)_{\lambda_0\dots\lambda_{k+1}}(z):=&\prod_{1\leq i<j\leq n}
\phi^{(k-2)2}_{\lambda_0\dots\lambda_{k-2}}(z)_{ij}^{B_{\lambda_{k-2}\lambda_{k-1}\lambda_{k}
\lambda_{k+1}}(z)_{ij}}.
\end{align*}
The \v{C}ech differential $\check\partial$ is a graded derivation with respect to these 
products \cite{BouRat1}.
Now, let $F_i$ denote the $i^{th}$ component of $F$, our fixed representative of the 
Euler vector of $\pi:X\to Z$. Applying Lemma \ref{commutecoboudary} with $A=F_i, B=F_j$ 
gives us a 3-cochain $C(F)\in \check{C}^3(\W,\underline{M_n^u(\Z)})$ defined by the formula:
\[C(F)_{\lambda_0\lambda_1\lambda_2\lambda_3}(z)_{ij}:=
F_{\lambda_0\lambda_1\lambda_2}(z)_iF_{\lambda_0\lambda_2\lambda_3}(z)_j-
F_{\lambda_1\lambda_2\lambda_3}(z)_iF_{\lambda_0\lambda_1\lambda_3}(z)_j.\]
Then,  $D_F: C^k_F({\mathcal W},{\mathcal S})\to C^{k+1}_F({\mathcal W},{\mathcal S})$ 
is defined as follows:
\begin{align}
\label{dfdefn}D_F(\phi^{k0}, \phi^{(k-1)1},\phi^{(k-2)2}):=&
(\check\partial\phi^{k0}\times(\phi^{(k-1)1}\cup_1 F)^{(-1)^{k+1}}\times(\phi^{(k-2)2}
\cup_2 C(F)))^{(-1)^{k+1}},\notag\\
&\quad\check\partial\phi^{(k-1)1}\times(\phi^{(k-2)2}\cup_1 F)^{(-1)^{k}},\check\partial\phi^{(k-2)2})\notag.
\end{align}
It is straightforward to see that $D_F^2=0$  so we have:
\begin{defn}
The $k^{th}$ dimensionally reduced \v{C}ech cohomology group of the covering $\W$ 
with coefficients in ${\mathcal S}$,
is the cohomology of $C^k_F({{\mathcal W}},{{\mathcal S}})$  under the differential $D_F$. 
This group is denoted 
${{\mathbb H}}^k_F(\W,{\mathcal S}).$
\end{defn}
We can also define similarly, ${\mathbb H}^k_F(\W,\underline{\Z})$ and 
${\mathbb H}^k_F(\W,{\mathcal  R})$, using integer  and real coefficients.
 Cochains in $C^k_F(\W,\underline{\Z})$ are triples $(\phi^{k0},\phi^{(k-1)1},\phi^{(k-2)2})$ 
 consisting of  \v{C}ech cochains $\phi^{k0}\in \check{C}^{k}(\W,\underline{\Z})$, $\phi^{(k-1)1}\in 
\check{C}^{k-1}(\W,\underline{\Z}^n)$ and $\phi^{(k-2)2}\in \check{C}^{k-2}(\W,\underline{M_n^u(\Z)})$. 
We define degree 0 and 1 cochains as before, by truncating the lower \v{C}ech cochains. 
To define the differential, let $m_l$ denote the $l^{th}$ component of $m\in\Z^n$.
Then we have maps
\begin{align*}
\cup_1 F: \check{C}^{k-1}(\W,\underline{\Z}^n)\to \check{C}^{k+1}(\W,\underline{\Z}),
\cup_1 F&: \check{C}^{k-2}(\W,\underline{M_n^u(\Z)}) \to \check{C}^{k}(\W,\underline{\Z}^n), 
\quad \mbox{and}\\
\cup_2 C(F): \check{C}^{k-2}(\W,\underline{M_n^u(\Z)})& \to \check{C}^{k+1}(\W,\underline{\Z}),
\end{align*}
with their integer cohomology analogues:
\begin{align*}
(\phi^{(k-1)1}\cup_1 F)_{\lambda_0\dots\lambda_{k+1}}(z)=&
\sum_{l=1}^n\phi^{(k-1)1}_{\lambda_0\dots\lambda_{k-1}}(z)_l 
F_{\lambda_{k-1}\lambda_k\lambda_{k+1}}(z)_l,\\
(\phi^{(k-2)2}\cup_1 F)_{\lambda_0\dots\lambda_{k}}(z)_l=
\sum_{1\leq i<j\leq n}\phi^{(k-2)2}_{\lambda_0
\dots\lambda_{k-2}}(z)_{ij}    &(F_{\lambda_{k-2}\lambda_{k-1}\lambda_{k}}(z)_i(e_l)_j-
(e_l)_iF_{\lambda_{k-2}\lambda_{k-1}\lambda_{k}}(z)_j),\\
\mbox{and} \quad
(\phi^{(k-2)2}\cup_2 C(F))_{\lambda_0\dots\lambda_{k+1}}(z):=&\sum_{1\leq i<j\leq n}
\phi^{(k-2)2}_{\lambda_0\dots\lambda_{k-2}}(z)_{ij} C(F)_{\lambda_{k-2}\lambda_{k-1}
\lambda_{k}\lambda_{k+1}}(z)_{ij},
\end{align*}
Then the differential in integer coefficients is
\begin{align*}
D_F(\phi^{k0}, \phi^{(k-1)1},\phi^{(k-2)2}):=&
(\check\partial\phi^{k0}+(-1)^{k+1}\phi^{(k-1)1}\cup_1 F+(-1)^{k+1}\phi^{(k-2)2}\cup_2 C(F),\\
&\quad\check\partial\phi^{(k-1)1}+(-1)^{k}\phi^{(k-2)2}\cup_1 F,\check\partial\phi^{(k-2)2}).
\end{align*}
\begin{defn}
Fix a cocycle $F\in\check{Z}^2(\W,\underline{\Z}^n)$. The \emph{$k^{th}$ dimensionally 
reduced \v{C}ech cohomology group of the cover $\W$ with coefficients in $\Z$} is the 
cohomology of $C^k_F({\mathcal W},\underline{\Z})$. 
\end{defn}
{{}
Similarly for real coefficients, cochains in $C^k_F(\W,{\mathcal  R})$ 
are triples $(\phi^{k0},\phi^{(k-1)1},\phi^{(k-2)2})$ consisting of  \v{C}ech 
cochains $\phi^{k0}\in \check{C}^{k}(\W,{\mathcal  R})$, $\phi^{(k-1)1}\in 
\check{C}^{k-1}(\W,{\mathcal  R}^n)$ and $\phi^{(k-2)2}\in 
\check{C}^{k-2}(\W,\mathcal{M}({\mathcal  R})),$ where ${\mathcal  R}$ 
denotes the sheaf of germs of continuous $\R$-valued functions, and 
$\mathcal{M}({\mathcal  R})$, the sheaf of germs of continuous $M_n^u(\R)$-valued functions.}
\begin{prop}[\cite{BouRat1}]\label{cohomologylongexactsequence}
Let $\W$ be a good open cover of a $C^\infty$ manifold $Z$, and fix a cocycle 
$F\in \check{Z}^2(\W,\underline{\Z}^n)$. Then there is a long exact sequence of cohomology groups
\[\cdots\to\mathbb{H}^k_F(\W,{\mathcal  R})\to\mathbb{H}^k_F(\W,{\mathcal S})\to 
{\mathbb H}^{k+1}_F(\W,\underline{\Z})\to {\mathbb H}^{k+1}_F(\W,{\mathcal  R})\to\cdots\]
\end{prop}

{{} Most of the groups ${\mathbb H}^k_F(\W,{\mathcal  R})$ are trivial. The proof
 is contained in \cite{BouRat1}.}

\begin{lemma}\label{Rcompute}
Let $\W$ be an open cover of of a $C^\infty$ manifold $Z$, and fix a 
cocycle $F\in \check{Z}^2(\W,\underline{\Z})$. Then we have group isomorphisms
\[{\mathbb H}^k_{F}(\W,{\mathcal  R})\cong\begin{cases}
C(Z,\R) & k=0\\
C(Z,\R^n)& k=1\\
C(Z,M_n^u(\R))& k=2\\
0& k\geq 3.
\end{cases}\]
\end{lemma}

Consequently, we have the:

\begin{cor}\label{dimreducedlongexact}
Let $\W$ be a good open cover of a $C^\infty$ manifold $Z$, and fix a cocycle 
$F\in \check{Z}^2(\W,\underline{\Z}^n)$. Then we have exact sequences
\begin{align*}
0&\to C(Z,\Z)\to C(Z,\R)\to C(Z,\T)\to  {\mathbb H}^1_{F}(\W,\underline{\Z})
\to C(Z,\R^n)\to {\mathbb H}^1_{F}(\W,{\mathcal S})\to {\mathbb H}^2_{F}(\W,\underline{\Z})\\
&\quad\quad \to C(Z,M_n^u(\R))\to {\mathbb H}^2_{F}(\W,{\mathcal S})\to 
{\mathbb H}^3_{F}(\W,\underline{\Z})\to 0,
\end{align*}
and
$\quad 0\to {\mathbb H}^k_{F}(\W,{\mathcal S})\to 
{\mathbb H}^{k+1}_{F}(\W,\underline{\Z})\to 0,\quad k\geq 3.$
\end{cor}

% ==============================================================================
\section{Covers of Groupoids}\label{Covers}

As mentioned in the introduction, in \cite{Tu}  Tu established a groupoid cohomology theory. 
It is isomorphic, in degree two, to the equivariant Brauer group.
An understanding of these cohomology groups $\check{H}^2(\R^n\ltimes X,{\mathcal S})$ and 
the isomorphism $\Upsilon:\operatorname{Br}_{\R^n}(X)\to \check{H}^2(\R^n\ltimes X,{\mathcal S})$ 
will be necessary to prove surjectivity of our map from the equivariant Brauer group to our 
dimensionally reduced cohomology group from Section \ref{DimRedCohoSection} 
(see Corollary \ref{surjectiveproof}).

\begin{defn}\label{simplicialdefn}
Let $\Delta$ be the \emph{simplicial category}, consisting of $(n+1)$-tuples $\{0,1,2,\dots,n\}$,
written $[n]$, $n\in\N^+$  and non-decreasing maps as morphisms. (A map $f:[k]\to [n]$ is 
\emph{non}-\emph{decreasing} if $f(i+1)\geq f(i)$ for all $i\in[k]$). The \emph{presimplicial 
category} $\Delta'$ is the category with the same objects and strictly increasing maps as 
morphisms. Denote the set of morphisms from $[k]$ to $[n]$ in the category  by $\Delta$ 
(resp. $\Delta'$) by $\mbox{hom}_{\Delta} ([k],[n])$  (resp. $\mbox{hom}_{\Delta'} ([k],[n])$).

A \emph{(pre)simplicial set} is a functor from the (pre)simplical category 
$\Delta$ (or $\Delta'$)  to the category $\mathcal{SET}$ of sets.
A \emph{(pre)simplicial (topological) space} is a functor from the 
(pre)simplical category $\Delta$ (or $\Delta'$)  to the category $\mathcal{TOP}$ of topological spaces.
\end{defn}

\begin{remark}\label{epsilonremark}
From, e.g.\ \cite{Mac63}, the image of the morphisms from $\Delta'$ (resp. $\Delta$) under 
a functor from the (pre)simplicial category is determined by the images of the collection of 
$\epsilon_i,i\in [n]$, the unique injective map in $\mbox{hom}_{\Delta}([n],[n+1])$ that 
omits the number $i$, and $\eta_i,i\in [n]$, the unique surjective map in 
$\mbox{hom}_{\Delta}([n],[n-1])$ that repeats the number $i$. 
Therefore, if $\mathcal{CAT}$ is an arbitrary category, and $F$ any functor 
$F:\Delta\to \mathcal{CAT}$, we  denote the images $F(\epsilon_i)$ and 
$F(\eta_i)$ of $\epsilon_i$ and $\eta_i$ by $\tilde{\epsilon}_i$ and $\tilde{\eta}_i$. 
We rely on context to make it clear which functor and which category is being used.
\end{remark}

\begin{remark}
We will often omit commas between indices to make reading indices printed in 
small fonts (e.g.\ when used as a subscript) easier on the eye.
\end{remark}

\begin{defn}
A \emph{groupoid} is a small category $\G$ in which every arrow is invertible. 
We denote the object and arrow spaces $\G^{(0)}$ and $\G^{(1)}$, and, if 
$g\in\G^{(1)}$, then $r(g),s(g)\in \G^{(0)}$ are the range and source of $g$ respectively. 
The symbol $\iota$ denotes the canonical inclusion $\iota:\G^{(0)}\to\G^{(1)}$. 
A groupoid is called a \emph{topological groupoid} if the sets $\G^{(0)}$ and $\G^{(1)}$ 
are topological spaces such that $r$, $s$, and composition of arrows are continuous maps.
\end{defn}

\begin{example}\label{groupoidsimplicial}
{Given a groupoid $\G$ there is an associated simplicial space whose object space is}
$\G^{(n)}:=\{(\gamma_0,\dots,\gamma_{n-1}): s(\gamma_{i+1})=r(\gamma_{i})\},$
the set of $n$-composable arrows. The morphisms are generated by compositions of maps 
$\tilde{\epsilon}_i$ and $\tilde{\eta}_i, i\in[n]$, where
\begin{equation}\label{epsilongammas}
\tilde{\epsilon}_i(\gamma_1,\dots,\gamma_n)=
\begin{cases}
(\gamma_1,\dots,\gamma_{n-1}) & i=0\\
(\gamma_0,\dots,\gamma_{n-2}) & i=n\\
(\gamma_0,\dots,\gamma_{i-1}\gamma_{i},\dots,\gamma_{n-1}) & \mbox{otherwise.}\\
\end{cases}
\end{equation}
and
\[
\tilde{\eta}_i(\gamma_0,\dots,\gamma_{n-1})=
\begin{cases}
(\iota(s(\gamma_0)),\gamma_0,\dots,\gamma_{n-1}) & \mbox{if } i=0\\
(\gamma_0,\dots,\gamma_{i-1},\iota(r(\gamma_{i-1})),\gamma_i,\dots,\gamma_{n-1})& \mbox{otherwise.} \\
\end{cases}
\]
\end{example}

From now on, all groupoids we deal with will be topological groupoids, so we omit   the 
``topological". In order to define Tu-\v{C}ech cohomology for a groupoid $\G$, we are going 
to need a \emph{pre-simplicial cover} of the simplicial space $\G^\bullet$ associated to $\G$. 
This will be generated from covers of $\G^{(0)}$ and $\G^{(1)}$ respectively, and will actually 
be a refinement in the sense discussed below.

Recall that if  $G$ is a group acting on a space $X$ then there is a \emph{transformation-group 
groupoid} $G\ltimes X$ with $(G\ltimes X)^{(0)}:=X$ and $(G\ltimes X)^{(1)}:=G\times X$. 
Later, for $g\in G, x\in X$, we write $(g,x)$ for elements in $G\ltimes X$, which satisfies 
$s((g,x)):=g^{-1}x$ and $r((g,x)):=x$.

\begin{example}\label{covergroupoid}
For any open cover $\U^0=\{U_{\lambda_{0}}^0\}_{\lambda_{0}\in\mathcal{I}^0}$ of the 
object space $\G^0$ of a groupoid $\G$, we may define the \emph{cover groupoid} $\G[\U^0]$ by:
\[\G[\U^0]^{(1)}:=\{(\lambda_{0},\gamma,\lambda_{1}): \gamma\in \G, s(\gamma)\in 
U^0_{\lambda_{0}}, r(\gamma)\in U^0_{\lambda_{1}}\}.\]
This groupoid has object space
$\G[\U^0]^{(0)}=\{(\lambda_{0},x,\lambda_{0}):x\in U^0_{\lambda_{0}}\}=\coprod U_{\lambda_{0}}^0$.
\end{example}

\begin{defn}
An \emph{open cover} of a presimplicial space $M^\bullet$ is a sequence of covers 
$\U^\bullet=(\U^n)_{n\in\N}$ such that $\U^n=(U^n_i)_{i\in \mathcal{I}^n}$ is an open cover 
of the space $M_n$.
The cover is said to be \emph{presimplicial} if $\mathcal{I}^\bullet=(\mathcal{I}^n)_{n\in \N}$ 
is a presimplicial set such that for all $f\in\mbox{hom}_{\Delta'}(k,n)$ and for all $i\in\mathcal{I}^n$ 
one has $\tilde{f}(U^n_i)\subset U^k_{\tilde{f}(i)}$. One defines \emph{simplicial} covers similarly.
\end{defn}
\begin{defn}
Let $\U^\bullet=\{\{U^n_i\}_{i\in \mathcal{I}^n}:n\in\N\}$ and 
$\mathcal{V}^\bullet=\{\{V^n_j\}_{j\in \mathcal{J}^n}:n\in\N\}$ be covers of the simplicial 
space $\G^\bullet$. Then we say $\mathcal{V}^\bullet$ is finer than $\U^\bullet$ if for all $n$ 
there exists a collection of refinement maps $\theta^n:\mathcal{J}^n\to \mathcal{I}^n$ 
such that $V^n_j\subset U^n_{\theta_n(j)}$ for all $j$. If the covers $\U^\bullet$ and 
$\mathcal{V}^\bullet$ are (pre)simplicial then the refinement map $\theta_\bullet$ is 
required to be (pre)simplicial, which means that for any index $j\in \mathcal{J}^n$ and 
$f:[k]\to [n]$ a non-decreasing (strictly increasing) map, one has 
$\theta_k(\tilde{f}(j))=\tilde{f}(\theta_n(j))$.
\end{defn}
\begin{remark}
It is necessary that the refinement maps respect the (pre)simplicial structure so that the 
Tu-\v{C}ech differential, defined later, commutes with refinements.
\end{remark}

Fix $n\in \N^+$, an open cover ${\mathcal U}^0=\{U^0_i\}_{i\in \mathcal{I}^0}$ and 
$\mathcal{U}^1=\{U^1_j\}_{j\in \mathcal{I}^1}$ of $\G^{(0)}$ and $\G^{(1)}$ respectively. 
We define $\Lambda_n'$ to be the set of all maps
$\lambda:\bigcup_{k\in[1]} \mbox{hom}_{\Delta'} ([k],[n])\to \bigcup_{k\in [1]} \mathcal{I}^k$
that satisfy
$\lambda(\mbox{hom}_{\Delta'}   ([k],[n]))\subset \mathcal{I}^k.$
We use Tu's notation for the elements of $\Lambda_n'$. That is, if $\lambda_l$ 
is the index in $\mathcal{I}^0$ given by $\lambda(f:0\mapsto l)$ and $\lambda_{lp}$ is 
the index in $\mathcal{I}^1$ given by $\lambda(f:0\mapsto l, 1\mapsto p)$, we write
$\lambda:=\lambda_0\lambda_1\dots\lambda_n\lambda_{01}\lambda_{02}\dots\lambda_{0n}
\lambda_{12}\lambda_{13}\dots \lambda_{(n-1)n}.$
If $n>0$, we define $U_{\lambda}^{n}$ to be the set of all 
$(\gamma_0,\dots,\gamma_{n-1})\in \G^{(n)}$ such that, for all $0\leq k\leq l\leq n-1$, 
all of the following hold: $s(\gamma_0)\in U^0_{\lambda_0}, r(\gamma_k)\in 
U^0_{\lambda_{k+1}}, \mbox{ and } \gamma_k\dots \gamma_l\in U^1_{\lambda_{k (l+1)}}$.
If $n=0$ we simply have $U_{\lambda}^0=U_{\lambda_0}^0$.

\begin{example}
In the case $n=2$ and $\mathcal{G}=G\ltimes X$ we have $(g_0,g_1,x)\in 
U^2_{\lambda_0\lambda_1\lambda_2\lambda_{01}\lambda_{02}\lambda_{12}}$ if 
and only if all of the following hold: $g_0^{-1}g_1^{-1}x\in U^0_{\lambda_0}, g_1^{-1}x\in
U^0_{\lambda_1},x\in U^0_{\lambda_2},(g_0,g^{-1}_1x)\in U^1_{\lambda_{01}},(g_0g_1, x)\in
U^1_{\lambda_{02}},\mbox{ and }(g_1,x)\in U^1_{\lambda_{12}}$.
\end{example}

\begin{remark}
The reader should be careful to note that, despite the notation above, we shall use 
simultaneously the standard notation $U^0_{\lambda_0\lambda_1\dots \lambda_n}$ to 
denote $ U^0_{\lambda_0}\cap U^0_{\lambda_1}\cap\dots\cap U^0_{\lambda_n}$. 
This shouldn't be confusing, since the superscript of 
$U^0_{\lambda_0\lambda_1\dots \lambda_n}$ is $0$ and the index 
$\lambda_0\lambda_1\dots\lambda_n$ is not in $\Lambda'_n$.
\end{remark}

\begin{lemma}\label{presimplicialcoverproof}
Given open covers $\mathcal{U}^0=\{U^0_i\}_{i\in \mathcal{I}^0}$ 
and $\mathcal{U}^1=\{U^1_j\}_{j\in \mathcal{I}^1}$ of a topological groupoid 
$\G$, the collection $\sigma\U^\bullet:=\{\{U^n_\lambda\}_{\lambda\in\Lambda_n'}\}_{n\in \N}$ 
forms a presimplicial cover of the simplicial space $\G^\bullet=\{\G^{(n)}\}_{n\in\N}$. 
Moreover, if for $n\geq 2$ we let $\U^{(n)}:=\G^{(n)}$, then $\U^\bullet$ is an open 
cover of $\G^\bullet$, and $\sigma\U^\bullet$ is a refinement of $\U^\bullet$.
\end{lemma}

\begin{proof}
We define the presimplicial structure on the index set $\Lambda'_n$ by taking 
$\lambda\in \Lambda'_n$ and $g\in \mbox{hom}_{\Delta'}([k],[n])$ and taking 
$\tilde{g}:\Lambda'_n\to \Lambda'_k$ to be
\begin{equation}\label{tildegdefn}
\tilde{g}(\lambda)(f):=\lambda(g\circ f).
\end{equation}
We must show for any $g\in \mbox{hom}_{\Delta'}(k,n)$ that 
$\tilde{g}(U_{\lambda}^{n})\subset U_{\tilde{g}(\lambda)}^{k}$. Now, 
Remark \ref{epsilonremark} implies it suffices to consider $g=\epsilon_i$, 
so that we need to check $\tilde{\epsilon_i}(U_{\lambda}^{n})\subset 
U_{\tilde{\epsilon_i}(\lambda)}^{n-1}$. We do this for the case $1\leq i\leq n-1$ 
(the cases $i=0$ and $i=n$ are similar). Fix $(\gamma_0,\dots,\gamma_{n-1})\in 
U_{\lambda}^{n}$, and define $(\delta_0,\dots,\delta_{n-2})$ by $(\delta_0,\dots,\delta_{n-2}):=
\tilde{\epsilon}_i(\gamma_0,\dots,\gamma_{n-1})$. Also define the index $\mu\in \Lambda'_{n-1}$ 
by $\mu=\tilde{\epsilon}_i(\lambda)$. Therefore we need to show that $s(\delta_0)\in U^0_{\mu_0}, 
r(\delta_k)\in U^0_{\mu_{k+1}}$ and $\delta_k\dots \delta_l\in U^1_{\mu_{k (l+1)}}$. By Equation 
(\ref{epsilongammas}) we have
$\delta_k=\gamma_k,\ \  k<i-1,\ \delta_{i-1}=
\gamma_{i-1}\gamma_{i},\ \ \delta_k=\gamma_{k+1},\ \  k\geq i$.
Therefore, for any $0\leq k\leq l\leq n-2$, we have
\[\delta_k\dots\delta_l=\begin{cases}
\gamma_k\dots\gamma_l & l<i-1\\
\gamma_k\dots\gamma_{l+1} & k< i-1\leq l\\
\gamma_{k+1}\dots\gamma_{l+1} & i-1<k.
\end{cases}
\]
Moreover, by Equation (\ref{tildegdefn}) we have
$\mu_k=\lambda_k,\  k<i$,  $\mu_k=\lambda_{k+1},\ k\geq i$, and
\begin{align*}
\mu_{kl}=&\begin{cases}
\lambda_{kl}& l<i,\\
\lambda_{k(l+1)} & k<i\leq l,\\
\lambda_{(k+1)(l+1)} & k\geq i.
\end{cases}
\end{align*}
Let us show that $\delta_k\dots\delta_l\in U^{n-1}_{\delta_{k(l+1)}}$. 
Certainly if $l<i-1$ then $\delta_k\dots\delta_l= \gamma_k\dots\gamma_l\in 
U^{1}_{\lambda_{k(l+1)}}=U^{1}_{\mu_{k(l+1)}}$. Also, if $k<i-1\leq l$ then 
$\delta_k\dots\delta_l= \gamma_k\dots\gamma_{l+1}\in U^{1}_{\lambda_{k(l+2)}}=
U^{1}_{\mu_{k(l+1)}}$. Finally, if $k\geq i-1$, then $\delta_k\dots\delta_l= 
\gamma_{k+1}\dots\gamma_{l+1}\in U^{1}_{\lambda_{(k+1)(l+2)}}=U^{1}_{\mu_{k(l+1)}}$. 
Proving that $s(\delta_0)\in U^0_{\mu_0}$ and $r(\delta_k)\in U^0_{\mu_{k+1}}$ is similar.
\end{proof}

\begin{remark}
One may check that if
$\lambda=\lambda_0\lambda_1\dots\lambda_n\lambda_{01}\lambda_{02}\dots\lambda_{(n-1)n},$
then $\tilde{\epsilon_i}(\lambda)$ can be obtained from $\lambda$ by deleting 
any occurrences of $\lambda_i, \lambda_{il}$ or $\lambda_{li}$ where $l$ is arbitrary. 
For example if
$\lambda=\lambda_0\lambda_1\lambda_2\lambda_{01}\lambda_{02}\lambda_{12}$,
then
$\tilde{\epsilon}_1(\lambda)=\lambda_0\lambda_2\lambda_{02}.$
\end{remark}

Given the covers $\mathcal{U}^0=\{U^0_i\}_{i\in \mathcal{I}^0}$ and 
$\mathcal{U}^1=\{U^1_j\}_{j\in \mathcal{I}^1}$ we can also generate a 
\emph{simplicial} cover of $\G^\bullet$, by modifying the definitions above. 
Let $\lambda\in\Lambda_n$, where $\Lambda_n$ is the set of all 
maps $\lambda:\bigcup_{k\in[1]} \mbox{hom}_{\Delta} ([k],[n])\to \bigcup \mathcal{I}^k$ 
that satisfy $\lambda(\mbox{hom}_{\Delta}([k],[n]))\subset \mathcal{I}^k.$
The difference between this cover and the presimplicial cover from 
Lemma \ref{presimplicialcoverproof} is that here we are working with 
$\mbox{hom}_{\Delta}([k],[n])$, as opposed to $\mbox{hom}_{\Delta'}([k],[n])$. 
In Tu's notation, elements of $\Lambda_n$ take the form
$\lambda:=\lambda_0\lambda_1\dots\lambda_n\lambda_{00}\lambda_{01}\dots\lambda_{0n}
\lambda_{11}\lambda_{12}\dots \lambda_{nn}.$
If $n>0$, we define $U_{\lambda}^{n}$ to be the set of all 
$(\gamma_0,\dots,\gamma_{n-1})\in \G^{(n)}$ such that, for 
all $0\leq k\leq l\leq n-1$, all of the following hold: $\iota(s(\gamma_0))\in \iota 
(U^0_{\lambda_0})\cap U^1_{\lambda_{00}}, \iota(r(\gamma_k))\in 
\iota(U^0_{\lambda_{k+1}})\cap U^1_{\lambda_{(k+1)(k+1)}}, \mbox{ and }
\gamma_k\dots \gamma_l\in U^1_{k (l+1)}$. Here we have used the 
canonical injection $\iota$ of $\G^{(0)}$ into $\G^{(1)}$ (from here on, 
we omit the $\iota$). If $n=0$ we have $U_{\lambda}^0=\{x\in G^{(0)}:x\in 
U_{\lambda_0}^0, \iota(x)\in U_{\lambda_{00}}^1\}$.
\begin{lemma}\label{simplicialrefinement}
Given the open covers $\mathcal{U}^0=\{U^0_i\}_{i\in \mathcal{I}_0}$ and 
$\mathcal{U}^1=\{U^1_j\}_{j\in \mathcal{I}_1}$ of a topological groupoid $\G$, 
the collection $\sigma_\infty\U^\bullet:=\{\{U^n_\lambda\}_{\lambda\in\Lambda^n}\}_{n\in \N}$ 
forms a simplicial cover of the simplicial space $\G^\bullet=\{\G^{(n)}\}_{n\in\N}$. 
Moreover, if for $n\geq 2$ we let $\U^{(n)}:=\G^{(n)}$, then $\U^\bullet$ is an 
open cover of $\G^\bullet$, and $\sigma_\infty\U^\bullet$ is a refinement of $\U^\bullet$.
\end{lemma}
\begin{proof}
{This proof is identical to that of Lemma \ref{presimplicialcoverproof} with the 
addition $\tilde{\eta_j}(U_{\lambda}^{n})\subset U_{\tilde{\eta_j}(\lambda)}^{n+1}$.}
\end{proof}

\begin{remark}
One can check that if
$\lambda=\lambda_0\lambda_1\dots\lambda_n\lambda_{00}\lambda_{01}\dots\lambda_{nn},$
then $\tilde{\eta}_i(\lambda)$ can be obtained from $\lambda$ as follows. 
Repeat the number $i$ in the ordered set $\{0,1,\dots n\}$ to get the ordered 
set $\{0,\dots,i,i,\dots,n\}$. Now, take all non-increasing maps from $[1]$ to this 
set as the subscripts of $\tilde{\eta_i}(\lambda)$.  For example, if $\lambda\in\Lambda_1$ is 
written as $\lambda=\lambda_0\lambda_1\lambda_{00}\lambda_{01}\lambda_{11}$, then
$$
\tilde{\eta}_0(\lambda)=\lambda_0\lambda_0\lambda_1\lambda_{00}
\lambda_{00}\lambda_{01}\lambda_{00}\lambda_{01}\lambda_{11},\quad \mbox{and}\quad
\tilde{\eta}_1(\lambda)=\lambda_0\lambda_1\lambda_1\lambda_{00}
\lambda_{01}\lambda_{01}\lambda_{11}\lambda_{11}\lambda_{11}.
$$
\end{remark}

% ==============================================================================
\section{Tu-\v{C}ech Cohomology}\label{cohomdefn}

\begin{defn}
Let $\mathcal{U}^0=\{U^0_i\}_{i\in \mathcal{I}^0}$ and $\mathcal{U}^1=
\{U^1_j\}_{j\in \mathcal{I}^1}$ be open covers of a topological groupoid 
$\G$, and let $\sigma\U^\bullet$ be the presimplicial cover from 
Lemma \ref{presimplicialcoverproof}. We denote by $\check{C}^n(\sigma\U^\bullet,{\mathcal S})$ 
the group of \emph{Tu-\v{C}ech} cochains of degree $n$, where a cochain 
$\varphi\in \check{C}^n(\sigma\U^\bullet,{\mathcal S})$ is a collection 
$\varphi=\{\varphi_\lambda\}_{\lambda\in\Lambda_n'}$ of continuous functions 
$\varphi_\lambda:U^n_{\lambda}\to\T$.
\end{defn}  
\begin{lemma}[{\cite[Sect 4.2]{Tu}}]
The  Tu-\v{C}ech differential, denoted $\partial_{Tu}$, is given by
\[(\partial_{Tu}\varphi)_{\mu}(\gamma_0,\gamma_1,\dots,\gamma_n)=\sum_{i=0}^{n+1}(-1)^i
\varphi_{\tilde{\epsilon_i}\mu}(\tilde{\epsilon}_i(\gamma_0,\dots,\gamma_n)),\]
where $\mu\in\Lambda_{n+1}'$ and $(\gamma_0,\gamma_1,\dots,\gamma_n)\in U^{n+1}_\mu$.
\end{lemma}
 \begin{remark}
We work over ${\mathcal S}$ 
because our cochains take values in $\T$. In \cite{Tu}, any abelian sheaf is allowed, 
but we do not need that generality here.
\end{remark}
We denote the groups of cocycles and coboundaries by 
$\check{Z}^n(\sigma\U^\bullet,{\mathcal S})$ and  $\check{B}^n(\sigma\U^\bullet,{\mathcal S})$ 
respectively and the  cohomology group relative to $\{\U^0,\U^1\}$ using the differential on 
$\check{C}^\bullet(\sigma \U^\bullet,{\mathcal S})$ as
$\check{H}^\bullet(\sigma\U^\bullet,{\mathcal S})$.
We need of course that the maps  are independent of the choice of refinement maps.  
Thus if $\theta^\bullet: \mathcal{J}^\bullet\to\mathcal{I}^\bullet$ is a refinement map then let
$
\theta^*:\check{C}^n(\sigma \mathcal{U}^\bullet,{\mathcal S})\to 
\check{C}^n(\sigma \mathcal{V}^\bullet,{\mathcal S})$ be given by
$(\theta^*\varphi)_\lambda=\mbox{restriction of } \varphi_{\theta(\lambda)} \mbox{ to } V^n_\lambda$,
where $\theta(\lambda)=\theta^0(\lambda_0)\dots\theta^0(\lambda_n)\theta^1(\lambda_{01})\dots
\theta^1(\lambda_{(n-1)n}).$  As $\theta^*$ commutes with the differentials, it defines a map $\theta^*:
\check{H}^\bullet(\sigma\U^\bullet,{\mathcal S})\to 
\check{H}^\bullet(\sigma\mathcal{V}^\bullet,{\mathcal S})$. We now have \cite[Lemma 4.5]{Tu}, 
for our context.
\begin{lemma}\label{refinement}
Let  $\U^\bullet$ and $\mathcal{V}^\bullet$ be open covers of $\G^\bullet$ such that 
$\mathcal{V}^\bullet$ is a simplicial cover. Suppose that $\mathcal{V}^\bullet$ is finer 
than  $\U^\bullet$ and that $\theta,\vartheta:\mathcal{J}^\bullet\to\mathcal{I}^\bullet$ are 
two refinements. Then there exists $H:\check{C}^n(\sigma \mathcal{U}^\bullet,{\mathcal S})\to 
\check{C}^{n-1}(\sigma \mathcal{V}^\bullet,{\mathcal S})$ such that $\vartheta^{*}-\theta^{*}=dH+Hd$.
\end{lemma}
\begin{remark}
We omit the definition of $H$ and the proof of the above lemma, as they are 
both highly non-trivial, and we do not use them.
\end{remark}
\begin{defn}\label{TuCechDefn}
The \emph{Tu-\v{C}ech cohomology groups $\check{H}^\bullet(\G,{\mathcal S})$ of 
a topological groupoid $\G$} are defined by
$\check{H}^\bullet(\G,{\mathcal S}):=\varinjlim_{\U} \check{H}^\bullet(\sigma\U^\bullet,{\mathcal S}),$
where the inductive limit is taken over all \emph{simplicial} covers of $\U^\bullet$ of $\G^\bullet$.
\end{defn}

Now, Lemmas \ref{refinement} and \ref{simplicialrefinement} imply that there is 
a canonical map $\check{H}^\bullet(\U^\bullet,{\mathcal S})\to \check{H}^\bullet(\G,{\mathcal S})$ 
for \emph{any} open cover $\U^\bullet$ of $\G^{\bullet}$, regardless of whether $\U^\bullet$ 
is simplicial or not. In practice, we usually work over a cover where, for $n\geq 2$, 
we have $\U^{(n)}=\G^{(n)}$, and will often assume this without comment.

\begin{example}
For later use we describe 1-cocycles and 1-coboundaries for the  groupoid 
$\G=G\ltimes X$. Choose open covers $\U^0=(U_i^0)_{i\in \mathcal{I}^0}$ and 
$\U^1=(U_j^1)_{j\in \mathcal{I}^1}$ of $\G^{(0)}= X$ and $\G^{(1)}=G\times X$. 
Then, for any index $\lambda=\lambda_0\lambda_1\lambda_{01}\in \Lambda_1'$, 
$U_{\lambda_{0}\lambda_{1}\lambda_{01}}^1$ is
the set of $(g,x)\in G\times X$ such that $g^{-1}x\in U_{\lambda_{0}}^0, x\in 
U_{\lambda_{1}}^0$ and $(g,x)\in U^1_{\lambda_{01}}$. Now, a 0-cochain is a 
collection $\psi=\{\psi_{\lambda_{0}}\}_{\lambda_0\in \mathcal{I}^0}$ of continuous functions 
$\psi_{\lambda_{0}}:U^0_{\lambda_0}\to \T$, whilst a 1-cochain 
$\varphi=\{\varphi\}_{\lambda\in\Lambda_1'}$ is a collection of continuous functions 
$\varphi_{\lambda_{0}\lambda_{1}\lambda_{01}}:U_{\lambda_{0}\lambda_{1}\lambda_{01}}^1\to\T$. 
Therefore, a 1-coboundary $\phi$ is a 1-cochain such that there exists a 0-cochain $\psi$ with
$\phi_{\lambda_{0}\lambda_{1}}(g,x)=\psi_{\lambda_{0}}(g^{-1}x)^*\psi_{\lambda_{1}}(x).$
\end{example}
To describe a 1-cocycle, we use the following lemma:
\begin{lemma}[{\cite[Sect 5.2]{Tu}}]\label{onecocycle}
Let $\varphi=(\varphi_\lambda)_{\lambda\in\Lambda^1}$ be a 1-cocycle. Then 
$\varphi_{\lambda_0\lambda_1\lambda_{01}}$ does not depend on the choice of $\lambda_{01}$.
\end{lemma}
\begin{proof}
The fact that $\partial_{Tu}\varphi = 1$ implies that
$$
1=\varphi_{\lambda_1\lambda_1\lambda_{12}}(e,x)\varphi_{\lambda_0
\lambda_1\lambda_{01}}(g,x)\varphi_{\lambda_0\lambda_1\lambda_{01}}(g,x)^*
=\varphi_{\lambda_1\lambda_1\lambda_{12}}(e,x).
$$
Therefore, using $\partial_{Tu}\varphi =1$ again implies that
$$
\varphi_{\lambda_0\lambda_1\lambda_{02}}(g,x)=
\varphi_{\lambda_1\lambda_1\lambda_{12}}(e,x)\varphi_{\lambda_0\lambda_1\lambda_{01}}(g,x)
=\varphi_{\lambda_0\lambda_1\lambda_{01}}(g,x).
$$
\end{proof}
\begin{remark}
Note that the proof does not use commutativity of $\T$, so that the lemma applies equally to 
maps taking values in $U(\H)$ that satisfy the 1-cocycle identity.
\end{remark}
Therefore, $\varphi$ is a 1-cocycle if for any $(g,h,x)=((g,h^{-1}x),(h,x))\in \G^{(2)}$ that satisfy
$s(g,h^{-1}x)=(gh)^{-1}x\in U_{\lambda_{0}}^0$,
$r(g,h^{-1}x)=h^{-1}x\in U_{\lambda_{1}}^0$,
$r(h,x)=x\in U_{\lambda_{2}}^0$,
we have
$\varphi_{\lambda_{1}\lambda_{2}}(h,x)\varphi_{\lambda_{0}\lambda_{2}}(gh,x)^{*}
\varphi_{\lambda_{0}\lambda_{1}}(g,h^{-1}x)=1$.

The rest of this Section is devoted to  proving:
\begin{theorem}[{\cite[Cor 5.9]{Tu}}]\label{tusbigtheorem}
Let $G$ be an abelian second countable locally compact Hausdorff topological 
group (with Haar measure) acting on a second countable locally compact 
Hausdorff space $X$. Then there is an isomorphism
$\Upsilon:\mathrm{Br}_G(X)\to \check{H}^2(G\ltimes X,{\mathcal S})$.
\end{theorem}
We need our own proof as Tu's does not provide the right approach for
our main result. We first outline the construction of $\Upsilon$ and then justify the steps.
Let $(CT(X,\delta),\alpha)\in \mathfrak{Br}_G(X)$, and choose an open cover 
$\U^0=\{U_{\lambda_0}^0\}_{\lambda_0\in\mathcal{I}^0}$ of $X$ such that there exist 
$C_0(U_{\lambda_0}^0)$-isomorphisms (local trivialisations)
$\Phi_{\lambda_0}:CT(X,\delta)|_{U_{\lambda_0}^0}\to C_0(U_{\lambda_0}^0,\K).$
Then, perhaps after a refinement, there are continuous maps 
$u_{\lambda_0\lambda_1}:U^0_{\lambda_0\lambda_1}\to U(\H)$ such that 
$\Phi_{\lambda_1}\circ \Phi_{\lambda_0}^{-1}=\operatorname{Ad}u_{\lambda_0\lambda_1}$.
Recall the ``cover groupoid" from Example \ref{covergroupoid}. Fix $T\in\K$ and 
$a\in CT(X,\delta)$ satisfying $\Phi_{\lambda_0}(a)(g^{-1}x)=T$. Then there is a 
continuous map $\beta^{\alpha,\Phi}:(G\ltimes X[\U^0])^{(1)}\to\operatorname{Aut}\K$:
$\beta^{\alpha,\Phi}_{(\lambda_0(g,x)\lambda_1)}(T)=\Phi_{\lambda_1}(\alpha_g(a))(x)$.
Lemma \ref{betawelledfined} below will show that this is well-defined while 
Lemma \ref{u1coverexists} shows that there exists an open cover 
$\U^1=\{U^1_{\lambda_{01}}\}_{\lambda_{01}\in\mathcal{I}^1}$ of 
$(G\ltimes X)^{(1)}=G\times X$ and functions $u_{\lambda_0\lambda_1\lambda_{01}}\in 
C (U^1_{\lambda_0\lambda_1\lambda_{01}}, U(\H))$ such that
$\beta^{\alpha,\Phi}|_{U_{\lambda_0\lambda_1\lambda_{01}}^1}=
\operatorname{Ad}u_{\lambda_0\lambda_1\lambda_{01}}$. 
Finally, if $G$ is an abelian group, a simple computation using the definitions shows that 
for any $(g_0,g_1,x)\in U^2_{\lambda_0\lambda_1\lambda_2\lambda_{01}\lambda_{02}\lambda_{12}}:$
\[\operatorname{Ad}u_{\lambda_1\lambda_2\lambda_{12}}(g_1,x)
\operatorname{Ad}u_{\lambda_0\lambda_1\lambda_{01}}(g_0,g_1^{-1}x)=
\operatorname{Ad}u_{\lambda_0\lambda_2\lambda_{02}}(g_0g_1,x).\]
Therefore we may define a Tu-\v{C}ech 2-cocycle 
$\varphi(\alpha)\in \check{Z}(\sigma\U^\bullet,{\mathcal S})$, by
\begin{align}\label{cocycledefn}
&\varphi(\alpha)_{\lambda_0\lambda_1\lambda_2\lambda_{01}\lambda_{02}\lambda_{12}}(g_0,g_1,x)
:=&u_{\lambda_1\lambda_2\lambda_{12}}(g_1,x)u_{\lambda_0\lambda_1\lambda_{01}}
(g_0,g_1^{-1}x)u_{\lambda_0\lambda_2\lambda_{02}}(g_0g_1,x)^*,
\end{align}
and the map $\Upsilon$ is given by
$
\Upsilon:[CT(x,\delta),\alpha]\mapsto [\varphi(\alpha)].
$
We now provide the missing proofs.

\begin{lemma}\label{betawelledfined}
$\beta^{\alpha,\Phi}_{(\lambda_0(g,x)\lambda_1)}(T)$ is independent of the 
choice of $a\in CT(X,\delta)$ satisfying $\Phi_{\lambda_0}(a)(g^{-1}x)=T$.
\end{lemma}
\begin{proof}
Let $b$ be another element such that $\Phi_{\lambda_0}(b)(g^{-1}x)=T$. Then
$\Phi_{\lambda_0}(a-b)(g^{-1}x)=0.$
By \cite[Lemma 2.1]{EchWil98} and its proof, we can identify $CT(X,\delta)|_{X\backslash\{g^{-1}x\}}$ with
\[C_0(X\backslash\{g^{-1}x\})\cdot CT(X,\delta),\]
and there exists $f\in C_0(X)$ and $c\in CT(X,\delta)$ such that $f(g^{-1}x)=0$ and $f\cdot c=a-b$.

 Then the calculation
$$
\Phi_{\lambda_1}(\alpha_g(a-b))(x)=\Phi_{\lambda_1}(\alpha_g(f\cdot c))(x)
=\Phi_{\lambda_1}(\tau_g(f)\alpha_g(c))(x)
=f(g^{-1}x)\Phi_{\lambda_1}(\alpha_g(c))(x)
=0,
$$
shows that $\Phi_{\lambda_1}(\alpha_g(a))(x)=\Phi_{\lambda_1}(\alpha_g(b))(x)$.
\end{proof}

\begin{lemma}\label{u1coverexists}
There exists an open cover $\U^1=\{U^1_{\lambda_{01}}\}_{\lambda_{01}\in\mathcal{I}^1}$ 
of $(G\ltimes X)^{(1)}=G\times X$ and functions $u_{\lambda_0\lambda_1\lambda_{01}}\in 
C(U^1_{\lambda_0\lambda_1\lambda_{01}}, U(\H))$ such that
$\beta^{\alpha,\Phi}|_{U_{\lambda_0\lambda_1\lambda_{01}}^1}=
\operatorname{Ad}u_{\lambda_0\lambda_1\lambda_{01}}.$
\end{lemma}
\begin{proof}
Fix a pair of indices $\lambda_0,\lambda_1\in\mathcal{I}^0$. Then, because 
$U(\H)\to\operatorname{Aut}\K$ has continuous local sections, there is a collection of 
open sets
$U_{\mu_i}^1\subset \{(g,x):g^{-1}x\in U^0_{\lambda_0}, x\in U^0_{\lambda_1}\}$
and continuous maps $u_{\mu_i}:U_{\mu_i}^1\to U(\H)$ such that $\{U_{\mu_i}^1\}$ 
covers $\{(g,x):g^{-1}x\in U^0_{\lambda_0}, x\in U^0_{\lambda_1}\}$ and
$\beta^{\alpha,\Phi}_{(\lambda_0(g,x)\lambda_1)}=\operatorname{Ad}u_{\mu_i}(g,x)$
for all $(g,x)$ such that $(g,x)\in U_{\mu_i}$. Note that, for any $i$, $U_{\mu_i}^1$ is an 
open set in $(G\ltimes X)^{(1)}$, and we define
$u_{\lambda_0\lambda_1\mu_i}(g,x):=u_{\mu_i}(g,x).$
This definition is close to our goal, except the open set $U_{\mu_i}^1$ is \emph{a priori} 
dependent on $\lambda_0$ and $\lambda_1$, so it is not obvious that for any other pair 
of indices, $\lambda_2,\lambda_3\in\mathcal{I}^0$, there is a continuous map 
$u_{\lambda_2\lambda_3\mu_i}$ defined on
$\{(g,x):g^{-1}x\in U^0_{\lambda_2}, x\in U^0_{\lambda_3} \mbox{ and } (g,x)\in U_{\mu_i}^1\}.$
On the other hand, if $\lambda_2$ and $\lambda_3$ are another pair of indices such that
$U_{\mu_i}^1\cap\{(g,x):g^{-1}x\in U^0_{\lambda_2}, x\in U^0_{\lambda_3}\}\neq \emptyset,$
we may define a continuous map $u_{\lambda_2\lambda_3\mu_i}$ on
$\{(g,x):g^{-1}x\in U^0_{\lambda_2}, x\in U^0_{\lambda_3} \mbox{ and } (g,x)\in U_{\mu_1}^1\}$
by
\[u_{\lambda_2\lambda_3\mu_i}(g,x):=u_{\lambda_1\lambda_3}(x)
u_{\lambda_0\lambda_1\mu_i}(g,x)u_{\lambda_0\lambda_2}(g^{-1}x)^*.\]
We then claim that
\begin{equation}\label{beta23claim}
\beta^{\alpha,\Phi}_{(\lambda_2(g,x)\lambda_3)}=\operatorname{Ad}u_{\lambda_2\lambda_3\mu_i}(g,x)
\end{equation}
for all $(g,x)\in \{(g,x):g^{-1}x\in U^0_{\lambda_2},x\in U^0_{\lambda_3}\mbox{ and } 
(g,x)\in U_{\mu_i}^1\}$. Indeed, for any such element $(g,x)$ we have by definition
$\beta^{\alpha,\Phi}_{(\lambda_2(g,x)\lambda_3)}(T):=\Phi_{\lambda_3}(\alpha_g(a))(x),$
where $\Phi_{\lambda_2}(a)(g^{-1}x)=T$. Of course, by inserting $\Phi_{\lambda_0}^{-1}\circ 
\Phi_{\lambda_0}$ we have
\begin{align*}
&\Phi_{\lambda_2}\circ \Phi_{\lambda_0}^{-1}\circ \Phi_{\lambda_0}(a)(g^{-1}x)=T
\implies u_{\lambda_0\lambda_2}(x)\Phi_{\lambda_0}(a)(g^{-1}x)u_{\lambda_0\lambda_2}(g^{-1}x)^*=T
\end{align*}
which implies $ \Phi_{\lambda_0}(a)(g^{-1}x)=u_{\lambda_0\lambda_2}(x)^*
Tu_{\lambda_0\lambda_2}(g^{-1}x)$.
Then, by inserting $\Phi_{\lambda_1}^{-1}\circ \Phi_{\lambda_1}$ we have
\begin{align*}
&\beta^{\alpha,\Phi}_{(\lambda_2(g,x)\lambda_3)}(T)
=\Phi_{\lambda_3}\circ\Phi_{\lambda_1}^{-1}\circ \Phi_{\lambda_1}(\alpha_g(a))(x)
=u_{\lambda_1\lambda_3}(x)\Phi_{\lambda_1}(\alpha_g(a))(x)u_{\lambda_1\lambda_3}(x)^*\\
\quad\quad&=u_{\lambda_1\lambda_3}(x)u_{\lambda_0\lambda_1\mu_i}(g,x)
u_{\lambda_0\lambda_2}(g^{-1}x)^*Tu_{\lambda_0\lambda_2}(g^{-1}x)
u_{\lambda_0\lambda_1\mu_i}(g,x)^*u_{\lambda_1\lambda_3}(x)^*.
\end{align*}
This proves claim (\ref{beta23claim}). Therefore, to define the open cover 
$\U^1=\{U^1_{\lambda_{01}}\}$ of $G\times X$ such that $\beta^{\alpha,\Phi}$ is 
implemented by unitaries defined on sets of the form $U^1_{\lambda_0\lambda_1\lambda_{01}}$, 
it suffices to take $\U^1$ to be the union of all the $\{U_{\mu_i}\}$ generated by all 
pairs of indices $\lambda_0,\lambda_1$.
\end{proof}
\begin{prop}\label{mymaptotucech}
Let $G$ be a second countable locally compact Hausdorff abelian group 
acting on a second countable locally compact space $X$. Then the map $\Upsilon$
given by (\ref{cocycledefn}) induces a homomorphism (denoted by the same symbol)
$\Upsilon:\mathrm{Br}_G(X)\to \check{H}^2(G\ltimes X,{\mathcal S}).$
\end{prop}

\begin{proof}
In constructing $\Upsilon:\mathfrak{Br}_G(X)\to \check{H}^2(G\ltimes X,\mathcal{S})$, 
we had to choose local trivialisations $\{\Phi_\lambda\}$ of $CT(X,\delta)$, and the unitary 
lifts $\{u_{\lambda_0\lambda_1\lambda_{01}}\}$ of $\beta^{\alpha,\Phi}_{(\lambda_0(g,x)\lambda_1)}$, 
and thus it is not clear that $\Upsilon$ is well-defined.  First we assume that it is, and prove it 
descends to a map on $\operatorname{Br}_{\R^n}(X)$. Below, we will omit mentioning that 
we may need to refine the cover at each stage.

Let $(CT(X,\delta),\chi)$ be outer conjugate to $(CT(X,\delta),\alpha)$. Choose local 
trivialisations $\Psi_{\lambda_0}:CT(X,\delta)|_{U_{\lambda_0}}$ $\to C_0(U_{\lambda_0},\K)$, and let
$\beta^{\chi,\Psi}_{(\lambda_0(g,x)\lambda_1)}(T):=\Psi_{\lambda_1}(\chi_g(b))(x),$
where $b\in CT(X,\delta)$ satisfies $\Psi_{\lambda_0}(b)(g^{-1}x)=T$. Then there exists 
$\{\tilde{u}_{\lambda_0\lambda_1\lambda_{01}}\}$ such that
$\beta^{\chi,\Psi}_{(\lambda_0(s,x)\lambda_1)}=\operatorname{Ad}
\tilde{u}_{\lambda_0\lambda_1\lambda_{01}}(s,x)$.  By the definition of outer conjugacy, 
there exists a $C_0(X)$-isomorphism $\Phi:CT(X,\delta)\to CT(X,\delta)$ and a continuous 
map $w:G\to UM(CT(X,\delta))$ such that
\begin{align}
\label{extcondition1ch3}\Phi^{-1}\circ \chi_s\circ \Phi(a)&=w_s\alpha_s(a)w_s^*,\\
\label{extcondition2ch3}w_{s+t}&=w_s\overline{\alpha}_s(w_t).
\end{align}
Finally, we may also suppose there are continuous maps $\nu_{\lambda_0}:U_{\lambda_0}\to 
U(\H)$ such that
\begin{equation}\label{4PointBlah}
\Psi_{\lambda_0}\circ \Phi\circ \Phi_{\lambda_0}^{-1}=\operatorname{Ad} \nu_{\lambda_0}.
\end{equation}

Using all these relations, we re-examine $\beta^{\chi,\Psi}$. Fix $T\in \K$ and let 
$b\in CT(X,\delta)$ satisfy $\Psi_{\lambda_0}(b)(-s\cdot x)=T$. Define 
$a\in CT(X,\delta)$ by $a:=\Phi^{-1}(b)$, which must satisfy
$\Psi_{\lambda_0}\circ\Phi(a)(-s\cdot x)=T.$
Notice that, since
$\Psi_{\lambda_0}\circ\Phi=\Psi_{\lambda_0}\circ\Phi\circ \Phi_{\lambda_0}^{-1}\circ\Phi_{\lambda_0},$
the element $a$ satisfies
\begin{equation}\label{4PointBlah2}
\Phi_{\lambda_0}(a)(-s\cdot x)=\operatorname{Ad}\nu_{\lambda_0}(-s\cdot x)^*(T).
\end{equation}
\begin{align*}\mbox{Then \quad\quad\quad\quad\quad }
\beta^{\chi,\Psi}_{(\lambda_0(s,x)\lambda_1)}(T)&:=\Psi_{\lambda_1}(\chi_g(b))(x)
=\Psi_{\lambda_1}\circ \Phi\circ\Phi^{-1}(\chi_g(b))(x)\\
=&\Psi_{\lambda_1}\circ \Phi\circ \Phi_{\lambda_1}^{-1}
\circ\Phi_{\lambda_1}\circ\Phi^{-1}(\chi_g(\Phi(a)))(x)\\
=&\operatorname{Ad} \nu_{\lambda_1}(x) \Phi_{\lambda_1}(w_{s}\alpha_s(a)
w_s^*)(x)\quad \mbox{by } (\ref{extcondition1ch3}) \mbox{ and } 
(\ref{4PointBlah})\quad\quad\quad\quad\quad\\
=&[\operatorname{Ad} \nu_{\lambda_1}(x)\overline{\Phi_{\lambda_1}}(w_{s})(x)]
\Phi_{\lambda_1}(\alpha_s(a))(x).
\end{align*}
From this calculation, using (\ref{4PointBlah2}) we obtain the relation:
\begin{align*}
\beta^{\chi,\Psi}_{(\lambda_0(s,x)\lambda_1)}(T)=&\operatorname{Ad} 
[\nu_{\lambda_1}(x)\overline{\Phi_{\lambda_1}}(w_{s})(x)u_{\lambda_0
\lambda_1\lambda_{01}}(s,x)\nu_{\lambda_0}(-s\cdot x)^*](T).
\end{align*}
Therefore there exist continuous functions 
$\tau_{\lambda_0\lambda_1\lambda_{01}}:U^1_{\lambda_0\lambda_1\lambda_{01}}\to \T$ and
\[
\tilde{u}_{\lambda_0\lambda_1\lambda_{01}}(s,x)=\tau_{\lambda_0\lambda_1\lambda_{01}}(s,x)
\nu_{\lambda_1}(x)\overline{\Phi_{\lambda_1}}(w_{s})(x)u_{\lambda_0\lambda_1\lambda_{01}}(s,x)
\nu_{\lambda_0}(-s\cdot x)^*.
\]
Finally, let $\overline{\Phi}_{\lambda_1}(w_{g_0})(g_{1}^{-1}x)=T$ so that 
$\overline{\beta}_{\lambda_1(g_1,x)\lambda_2}(T)=\overline{\Phi_{\lambda_2}}
(\overline{\alpha}(w_{g_0}))(x)$, and then using Lemma \ref{cyclicpermute}, the 
fact that $G$ is abelian, and Equation (\ref{extcondition2ch3}) we obtain
\begin{align*}
\varphi(\chi)&_{\lambda_0\lambda_1\lambda_2\lambda_{01}\lambda_{02}
\lambda_{12}}(g_0,g_1,x)=\tau_{\lambda_1\lambda_2\lambda_{12}}(g_1,x)
\nu_{\lambda_2}(x)\overline{\Phi_{\lambda_2}}(w_{g_1})(x)
u_{\lambda_1\lambda_2\lambda_{12}}(g_1,x)\nu_{\lambda_1}(g_1^{-1}x)^*\\
&\times\tau_{\lambda_0\lambda_1\lambda_{01}}(g_0,g_1^{-1}x)
\nu_{\lambda_1}(g_1^{-1}x)\overline{\Phi_{\lambda_1}}(w_{g_0})
(g_1^{-1}x)u_{\lambda_0\lambda_1\lambda_{01}}(g_0,g_1^{-1}x)
\nu_{\lambda_0}(g_0^{-1}g_1^{-1}x)^*\\
&\times\nu_{\lambda_0}(g_1^{-1}g_0^{-1}x)u_{\lambda_0\lambda_2\lambda_{02}}(g_0g_1,x)^*
\overline{\Phi_{\lambda_2}}(w_{g_0g_1})(x)^*\nu_{\lambda_2}(x)^*
\tau_{\lambda_0\lambda_2\lambda_{01}}(g_0g_1,x)^*\\
=&\partial_{Tu}\tau_{\lambda_0\lambda_1\lambda_2\lambda_{01}\lambda_{02
}\lambda_{12}}(g_0,g_1,x)\overline{\Phi_{\lambda_2}}(w_{g_0g_1})(x)^*
\overline{\Phi_{\lambda_2}}(w_{g_1})(x)u_{\lambda_1\lambda_2\lambda_{12}}(g_0,x)
\overline{\Phi_{\lambda_1}}(w_{g_0})(g_1^{-1}x)\\
&\times u_{\lambda_0\lambda_1\lambda_{01}}(g_0,g_1^{-1}x)
u_{\lambda_0\lambda_2\lambda_{02}}(g_0g_1,x)^*\\
=&\varphi(\alpha)_{\lambda_0\lambda_1\lambda_2\lambda_{01}
\lambda_{02}\lambda_{12}}(g_0,g_1,x)\partial_{Tu}
\tau_{\lambda_0\lambda_1\lambda_2\lambda_{01}\lambda_{02}
\lambda_{12}}(g_0,g_1,x)\overline{\Phi_{\lambda_2}}(\overline{\alpha}(w_{g_0}))(x)^*\\
&\times u_{\lambda_1\lambda_2\lambda_{12}}(g_0,x)\overline{\Phi_{\lambda_1}}(w_{g_0})
(g_1^{-1}x)u_{\lambda_1\lambda_2\lambda_{12}}(g_1,x)\\
=&(\varphi(\alpha)\partial_{Tu}
\tau)_{\lambda_0\lambda_1\lambda_2\lambda_{01}\lambda_{02}\lambda_{12}}(g_0,g_1,x)
\overline{\Phi_{\lambda_2}}(\overline{\alpha}(w_{g_0}))(x)^*
\overline{\beta}_{\lambda_1(g_1,x)\lambda_2}(T)\\
=&(\varphi(\alpha)\partial_{Tu}
\tau)_{\lambda_0\lambda_1\lambda_2\lambda_{01}\lambda_{02}\lambda_{12}}(g_0,g_1,x).
\end{align*}
Therefore $\Upsilon(CT(X,\delta),\chi)=\Upsilon(CT(X,\delta),\alpha)$, and 
$\Upsilon$ is constant on outer conjugacy classes.

Finally, to see that $\Upsilon$ is well-defined, we use the same proof as above, 
except we set $\Phi=\operatorname{id}$ and $w_g(x)=1$.
\end{proof}

We claim the map $\Upsilon$ is an isomorphism. 
The proof of injectivity is self-contained while surjectivity depends on results external to this paper.
First we show that if $[CT(X,\delta),\alpha]$ $\mapsto 1\in \check{H}^2(G\ltimes X,{\mathcal S})$, 
then $CT(X,\delta)$ is $C_0(X)$-isomorphic to $C_0(X,\K)$.
\begin{lemma}\label{forgetfullemma}
Suppose $\Upsilon:(CT(X,\delta),\alpha)\mapsto [\varphi(\alpha)]\in 
\check{H}^2(\sigma\U^\bullet,{\mathcal S})$, where $\U^\bullet$ is a 
\emph{simplicial} open cover $\mathcal{U^\bullet}$ of $(G\ltimes X)^\bullet$ 
(from Lemma \ref{simplicialrefinement}). Let $\Delta$ be the isomorphism 
$\check{H}^2(X,{\mathcal S})\stackrel{\cong}{\to} \check{H}^3(X,\underline{\Z})$ 
and define $F(\varphi)\in \check{Z}^2(\U^0,{\mathcal S})$ to be the \v{C}ech 
cohomology 2-cocycle  given by
\[F(\varphi):x\mapsto \varphi(\alpha)_{\lambda_0\lambda_1\lambda_2
\tilde{\eta}(\lambda_1)\tilde{\eta}(\lambda_2)\tilde{\eta}(\lambda_2)}(e,e,x), \quad x\in 
U_{\lambda_0\lambda_1\lambda_2}^0.\]
Then $\Delta([F(\varphi(\alpha))])=\delta$.
\end{lemma}
\begin{proof}
Notice that, for fixed $T\in\K$, if $a\in CT(X,\delta)$ is such that $\Phi_{\lambda_0}(a)(x)=T$, then
\begin{align*}
\beta^{\alpha,\Phi}_{(\lambda_{0}(e,x)\lambda_{1})}(T):=\Phi_{\lambda_1}(\alpha_e(a))(x)
=\Phi_{\lambda_1}\circ\Phi_{\lambda_0}^{-1}(T).
\end{align*}
That is $\beta^{\alpha,\Phi}_{(\lambda_{0}(e,x)\lambda_{1})}(T)=
\operatorname{Ad}v_{\lambda_{0}\lambda_{1}}(x)(T)$
for some transition function $v_{\lambda_{0}\lambda_{1}}\in 
C(U_{\lambda_{0}\lambda_{1}}^0,U(\H))$ corresponding to a vector bundle 
$p:E\to X$ with fibre $\K$, such that $\delta=[\check\partial v]$. 
Therefore, given the maps $u_{\lambda_0\lambda_1\lambda_{01}}\in 
C(U^1_\lambda, U(\H))$ that implement $\beta^{\alpha,\Phi}$, 
there must exist continuous functions $p_{\lambda_0\lambda_1\lambda_{01}}:
U^1_\lambda\cap \iota(\G^{(0)})\to\T$ such that
$u_{\lambda_0\lambda_1\lambda_{01}}(e,x)=p_{\lambda_0\lambda_1\lambda_{01}}
(x)v_{\lambda_0\lambda_1}(x).$
Thus, if $\nu$ denotes a 2-cocycle such that $\Delta([\nu])=\delta$, we must have
\begin{align*}
\varphi(\alpha)_{\lambda_0\lambda_1\lambda_2\lambda_{01}\lambda_{02}\lambda_{12}}(e,e,x)
&=u_{\lambda_1\lambda_2\lambda_{12}}(e,x)u_{\lambda_0\lambda_1\lambda_{01}}
(e,x)u_{\lambda_0\lambda_2\lambda_{02}}(e,x)^* \\
&=p_{\lambda_1\lambda_2\lambda_{12}}(x)v_{\lambda_1\lambda_2}
(x)p_{\lambda_0\lambda_1\lambda_{01}}(x)v_{\lambda_0\lambda_1}
(x)p_{\lambda_0\lambda_2\lambda_{02}}(x)^*v_{\lambda_0\lambda_2}(x)^*\\
&=p_{\lambda_1\lambda_2\lambda_{12}}(x)p_{\lambda_0\lambda_1\lambda_{01}}
(x)p_{\lambda_0\lambda_2\lambda_{02}}(x)^*\nu_{\lambda_0\lambda_1\lambda_2}(x).
\end{align*}
In particular, if we now refine the cover to be simplicial, then the map defined by
$x\mapsto p_{\lambda_1\lambda_2\tilde{\eta}(\lambda_2)}(x)p_{\lambda_0\lambda_1\tilde{\eta}
(\lambda_1)}(x)p_{\lambda_0\lambda_2\tilde{\eta}(\lambda_2)}(x)^*$
is a coboundary in (the group of ordinary \v{C}ech cochains) $\check{C}^2(\U^0,{\mathcal S})$, 
so the map $F:\check{Z}^2(\sigma\U^\bullet,{\mathcal S})\mapsto \check{Z}^2(\U^0,{\mathcal S})$ given
in the statement of the lemma satisfies $\Delta[F(\varphi)]=\delta$.
\end{proof}

\begin{cor}\label{mymapinjective}
The homomorphism $\Upsilon: \operatorname{Br}_G(X)\mapsto 
\check{H}^2(G\ltimes X,{\mathcal S})$ is injective.
\end{cor}
\begin{proof}
Let $(CT(X,\delta),\alpha)\in\mathfrak{Br}_G(X)$ and let $\U^0$  be an open cover
of $X$ such that there are trivialisations $\Phi_{\lambda_0}:CT(X,\delta)
|_{U_{\lambda_0}}\to C_0(U_{\lambda_0},\K)$. Choose an open cover $\U^1$ 
of $G\times X$ and continuous maps 
$u_{\lambda_0\lambda_1\lambda_{01}}:U^1_{\lambda_0\lambda_1\lambda_{01}}\to U(\H)$ so that
$\beta^{\alpha,\Phi}|_{U^1_{\lambda_0\lambda_1\lambda_{01}}}=\operatorname{Ad} 
u_{\lambda_0\lambda_1\lambda_{01}}$, and then define $\varphi(\alpha)$ as in Equation 
(\ref{cocycledefn}). Let $\tau$ denote the $G$-action on $X$ and
assume, perhaps after a refinement, that $\varphi(\alpha)$ is a coboundary.  
We have to show $[CT(X,\delta),\alpha]=[C_0(X,\K),\tau]=0$ in $\mbox{Br}_G(X)$. 
This means finding an isomorphism $\Phi: CT(X,\delta)\to C_0(X,\K)$ such that 
$\Phi\circ \alpha\circ \Phi^{-1}$ is exterior equivalent to $\tau$. For this purpose, 
we may assume $\U$ is simplicial.
Now, as $\varphi(\alpha)$ is a Tu-\v{C}ech coboundary, there exists a 
Tu-\v{C}ech 1-cochain $\phi$ such that
\[\varphi(\alpha)_{\lambda_0\lambda_1\lambda_{2}\lambda_{01}
\lambda_{02}\lambda_{12}}(g,h,x)=\phi_{\lambda_1\lambda_2\lambda_{12}}(h,x)
\phi_{\lambda_0\lambda_1\lambda_{01}}(g,h^{-1}x)\phi_{\lambda_0\lambda_2\lambda_{02}}(gh,x)^*.\]
Observe that $x\mapsto \phi_{\lambda_0\lambda_1\tilde{\eta}(\lambda_1)}(e,x)$ defines a 
\v{C}ech 1-cochain in $\check{C}^1(\U^0,{\mathcal S})$, and moreover that the cocycle 
$F(\varphi(\alpha))$ from Lemma \ref{forgetfullemma} satisfies
\begin{align*}
F(\varphi(\alpha))_{\lambda_0\lambda_1\lambda_2}(x)=&
\varphi_{\lambda_0\lambda_1\lambda_2\tilde{\eta}(\lambda_1)\tilde{\eta}(\lambda_2)\tilde{\eta}
(\lambda_2)}(e,e,x)\\
=&\phi_{\lambda_1\lambda_2\tilde{\eta}(\lambda_2)}(e,x)\phi_{\lambda_0\lambda_1\tilde{\eta}
(\lambda_1)}(e,x)\phi_{\lambda_0\lambda_2\tilde{\eta}(\lambda_2)}(e,x)^*.
\end{align*}
Therefore $F(\varphi(\alpha))$ is a coboundary, and $\delta=0$. Consequently, 
there exists a $C_0(X)$-isomorphism $\Phi:CT(X,\delta)\to C_0(X,\K)$.

That the continuous maps $\phi^*u_{\lambda_0\lambda_1\lambda_{01}}:(g,x)\mapsto 
\phi^*_{\lambda_0\lambda_1\lambda_{01}}(g,x)u_{\lambda_0\lambda_1\lambda_{01}}(g,x)$ 
defined on $U^1_{\lambda_0\lambda_1\lambda_{01}}$ are independent of $\lambda_{01}$  
follows directly from the proof of Lemma \ref{onecocycle}, and we write 
$\{\phi^*u_{\lambda_0\lambda_1\lambda_{01}}\}=:\{v_{\lambda_0\lambda_1}\}$. Moreover, 
$\{v_{\lambda_0\lambda_1}\}$ satisfies
\begin{equation}\label{vequation}
v_{\lambda_1\lambda_2}(g,x)v_{\lambda_0\lambda_1}(h,g^{-1}x)v_{\lambda_0\lambda_2}(hg,x)^*=1.
\end{equation}
Now, since $\beta^{\alpha,\Phi}_{(\lambda_0(e,x)\lambda_1)}=\operatorname{Ad}
v_{\lambda_0\lambda_1}(e,x)$ are the transition functions for some (locally trivial) bundle 
$p:E\to X$ with fibre $\K$
and $CT(X,\delta)\stackrel{\cong}{\to} C_0(X,\K)$, it follows that $p:E\to X$ is isomorphic 
to the trivial bundle $X\times \K\to X$. Therefore, (perhaps after a refinement) there 
exist continuous maps $w:U^0_{\lambda_0}\to U(\H)$ such that for all $h\in C_0(X,\K)$,
$\Phi_{\lambda_0}\circ \Phi^{-1}(h)(x)=\operatorname{Ad} w_{\lambda_0}(x)(h),$
which implies that
\begin{equation}\label{vandw}
v_{\lambda_0\lambda_1}(x)=w_{\lambda_1}(x)w_{\lambda_0}(x)^*.
\end{equation}
Let $(g,x)\in G\times X$, and let $\lambda_0$ and $\lambda_1$ be such 
that $g^{-1}x\in U^0_{\lambda_0}$ and $x\in U^0_{\lambda_{1}}$. We claim that
\begin{equation}\label{unitaryalphacocycle}
(g,x)\mapsto w_{\lambda_1}(x)^* v_{\lambda_0\lambda_1}(g,x) w_{\lambda_0}(g^{-1}x)
\end{equation}
is independent of the indices $\lambda_0$ and $\lambda_1$, and in 
fact is a unitary $\tau$ cocycle implementing an exterior equivalence 
between $\Phi\circ\alpha_g\circ\Phi^{-1}$ and $\tau$. Indeed, let 
$\lambda_0'$ and $\lambda_1'$ be two other indices such that $g^{-1}x\in 
U^0_{\lambda_0'}$ and $x\in U^0_{\lambda_{1}'}$. Two applications of 
Equation (\ref{vequation}) give the identities
$$
v_{\lambda_0\lambda_1}(g,x)v_{\lambda_0'\lambda_0}(e,g^{-1}x)=
v_{\lambda_0'\lambda_1}(g,x),\quad \mbox{and}\ \
v_{\lambda_0'\lambda_1}(g,x)v_{\lambda_0'\lambda_1'}(g,x)^*=v_{\lambda_1\lambda_1'}(g,x)^*.
$$

Applying these as well as Equation (\ref{vandw}) shows
\begin{align*}
&[w_{\lambda_1}(x)^* v_{\lambda_0\lambda_1}(g,x) w_{\lambda_0}(g^{-1}x)][w_{\lambda_0'}
(g^{-1}x)^*v_{\lambda_0'\lambda_1'}(g,x)^* w_{\lambda_1'}(x)]  \\
&=w_{\lambda_1}(x)^* v_{\lambda_0\lambda_1}(g,x) v_{\lambda_0'\lambda_0}
(g^{-1}x)v_{\lambda_0'\lambda_1'}(g,x)^* w_{\lambda_1'}(x)\\
&=w_{\lambda_1}(x)^* v_{\lambda_0'\lambda_1}(g,x)v_{\lambda_0'\lambda_1'}(g,x)^* 
w_{\lambda_1'}(x)
=w_{\lambda_1}(x)^* v_{\lambda_1\lambda_1'}(g,x)^* w_{\lambda_1}(x)=1\,,
\end{align*}
proving that (\ref{unitaryalphacocycle}) is independent of any choices of indices. 
To see that (\ref{unitaryalphacocycle}) is a unitary $\tau$ cocycle, we need to 
check the conditions (\ref{exteriorcondition1}) and (\ref{exteriorcondition2}). 
For  (\ref{exteriorcondition1}) we have for any $f\in C_0(X,\K)$,

\begin{align*}
\Phi\circ\alpha_g\circ\Phi^{-1}(f)(x)=&\Phi\circ\Phi_{\lambda_1}^{-1}\circ 
\Phi_{\lambda_1}\circ\alpha_g\circ\Phi_{\lambda_0}^{-1}\circ \Phi_{\lambda_0}\circ\Phi^{-1}(f)(x)
\end{align*}\begin{align*}
=&\operatorname{Ad}[w_{\lambda_1}(x)^* v_{\lambda_0\lambda_1}(g,x) 
w_{\lambda_0}(g^{-1}x)]f(g^{-1}x)\\
=&\operatorname{Ad}[w_{\lambda_1}(x)^* v_{\lambda_0\lambda_1}(g,x) 
w_{\lambda_0}(g^{-1}x)](\tau_g f)(x).
\end{align*}
For  (\ref{exteriorcondition2}), let $(h,g,x)\in (G\ltimes X)^{(2)}$, and choose indices 
such that $h^{-1}g^{-1}x\in U^0_{\lambda_0}$, $g^{-1}x\in U^0_{\lambda_1}$ and 
$x\in U^0_{\lambda_2}$. Then we may complete the argument as 
Lemma \ref{cyclicpermute} and Equation (\ref{vequation}) imply
\begin{align*}
&[w_{\lambda_2}(x)^* v_{\lambda_1\lambda_2}(g,x) 
w_{\lambda_1}(g^{-1}x)][w_{\lambda_1}(g^{-1}x)^* 
v_{\lambda_0\lambda_1}(h,g^{-1}x) w_{\lambda_0}(h^{-1}g^{-1}x)]\\
&=w_{\lambda_2}(x)^* v_{\lambda_1\lambda_2}(g,x) 
v_{\lambda_0\lambda_1}(h,g^{-1}x) w_{\lambda_0}(h^{-1}g^{-1}x)\\
&=w_{\lambda_2}(x)^* v_{\lambda_0\lambda_2}(hg,x)w_{\lambda_0}(h^{-1}g^{-1}x)
=w_{\lambda_2}(x)^* v_{\lambda_0\lambda_2}(gh,x)w_{\lambda_0}(h^{-1}g^{-1}x).
\end{align*}
\end{proof}

Before we prove surjectivity, we need a digression that will ensure, given a 
Tu-\v{C}ech cocycle $\varphi$, there exist continuous unitary-valued maps satisfying (\ref{cocycledefn}).

\begin{defn}
A \emph{twist $E$ over a groupoid $\G$} is a principal $\T$-bundle $j:E\to \G^{(1)}$ 
equipped with a groupoid structure such that $E$ is a groupoid extension of $\G^{(1)}$ 
by $\G^{(0)}\times\T$:
\[\G^{(0)}\to\G^{(0)}\times\T\stackrel{i}{\hookrightarrow }E\stackrel{j}{\twoheadrightarrow} \G^{(1)}.\]
\end{defn}
Two twists $E\stackrel{j_1}{\to} \G^{(1)}$ and $F\stackrel{j_2}{\to} \G^{(1)}$ over $\G$ are 
\emph{isomorphic} if there is groupoid homomorphism $\phi:E\to F$ such that $\phi$ is a 
$\T$-bundle isomorphism and the following diagram commutes:

\centerline{\xymatrix{
E\ar[rr]^{\phi}\ar[dr]^{j^1}&& F\ar[dl]_{j_2}\\
&\G^{(1)}&
}}

\noindent The ``Baer" sum of twists $E$ and $F$ over $\G$,  is the twist $E\oplus F$ over $\G$
defined by:
\[E\oplus F=\{(\gamma_1,\gamma_2)\in E\times F: j_1(\gamma_1)=j_2(\gamma_2)\},\]
with the obvious projection to $\G^{(1)}$. With this operation, the set of isomorphism 
classes of twists over $\G$ forms a group denoted $Tw(\G)$. To give an example of a twist, we need the:

\begin{defn}
Let $\G$ be a topological groupoid. We define ${\mathcal  R}(\G)$ to be the set 
of continuous groupoid homomorphisms $\pi:\G^{(1)}\to \operatorname{Aut}\K$.
\end{defn}

\begin{example}[{\cite[Sect 8]{KumMuhRenWil98}}]
Let $\G$ be a groupoid and let $\pi\in {\mathcal  R}(\G)$. Then there is a twist $E(\pi)$ over $\G$ given by
$E(\pi):=\{(\gamma,u)\in \G\times U(\H): \pi_{\gamma}=\operatorname{Ad}u\}.$
The $\T$ action is the map $t\cdot(\gamma,u):=(\gamma,tu), t\in\T$, 
and the bundle projection is given by $(\gamma,u)\mapsto \gamma$.
\end{example}
The importance of the above example is captured in the next result.
\begin{prop}[{\cite[Prop 8.7]{KumMuhRenWil98}}]\label{kmrw87}
Let $\G$ be a second countable locally compact groupoid with Haar system. 
Then if $E\in \mbox{Tw}(\G)$ there is a $\pi\in {\mathcal  R}(\G)$ such that $E\cong E(\pi)$.
\end{prop}
The relevance of twists over a groupoid to Tu-\v{C}ech cohomology is the following:
\begin{prop}[{\cite[Prop 5.6]{Tu}}]\label{Tu56}
Let $\G$ be a topological groupoid with Haar system. Then there is a canonical group isomorphism
$\mbox{Tw}_{\U}(\G[\U^0],\T)\cong \check{H}^2(\sigma\U^\bullet,{\mathcal S}),$
for each open cover $\U^{\bullet}$ of $\G^\bullet$
where $\mbox{Tw}_{\U}(\G[\U^0],\T)$ denotes the subgroup of $\mbox{Tw}(\G[\U^0])$ 
consisting of extensions
$(\G[\U^0])^{(0)}\to(\G[\U^0])^{(0)}\times \T\hookrightarrow 
E\stackrel{j}{\twoheadrightarrow} (\G[\U^0])^{(1)}$
such that $j$ admits a continuous lifting over each open set $U^1_{\lambda}$, $\lambda\in \Lambda_1'$.
\end{prop}
To write down this canonical isomorphism, first let $E\in \mbox{Tw}_{\U}(\G[\U^0],\T)$ be a twist. 
By Proposition \ref{kmrw87}, we may assume $E=E(\pi)$ for some $\pi\in {\mathcal  R}(\G)$. 
Then, because $E(\pi)\to (\G[\U^0])^{(1)}$ has continuous liftings over each open set 
$U^1_{\lambda}$, there exist continuous maps 
$u_{\lambda_0\lambda_1\lambda_{01}}:U^1_{\lambda_0\lambda_1\lambda_{01}}\to U(\H)$ 
such that $\pi_{(g,x)}=\operatorname{Ad}u_{\lambda_0\lambda_1\lambda_{01}}(g,x)$. 
One then defines a cocycle $\varphi\in \check{Z}^2(\sigma\U^\bullet,{\mathcal S})$ as in 
Equation (\ref{cocycledefn}).

\begin{cor}\label{UpsilonSurjective}
The homomorphism $\Upsilon:\operatorname{Br}_G(X)\mapsto 
\check{H}^2(G\ltimes X,{\mathcal S})$ is surjective.
\end{cor}
\begin{proof}
Fix a class $c\in \check{H}^2(G\ltimes X,{\mathcal S})$, and  an open cover 
$\U^\bullet$ of $\G^\bullet$ such that there is a cocycle $\varphi$ with 
$c=[\varphi]\in\check{H}^2(\sigma\U^\bullet,{\mathcal S})$. From the discussion 
just prior to this corollary, we know there exist continuous maps 
$u_{\lambda_0\lambda_1\lambda_{01}}:U^1_{\lambda_0\lambda_1\lambda_{01}}\to U(\H)$ such that
\[
\varphi_{\lambda_0\lambda_1\lambda_2\lambda_{01}\lambda_{02}\lambda_{12}}(g_0,g_1,x)
:=u_{\lambda_1\lambda_2\lambda_{12}}(g_1,x)u_{\lambda_0\lambda_1\lambda_{01}}
(g_0,g_1^{-1}x)u_{\lambda_0\lambda_2\lambda_{02}}(g_0g_1,x)^*.
\]
We use this data to construct a (locally trivial) bundle over $X$ with fibre $\K$, 
and then define an action of $G$ on the C*-algebra of sections.
First, note the above identity implies
\begin{align}
\operatorname{Ad}u_{\lambda_1\lambda_2\lambda_{12}}(g_1,x)
\operatorname{Ad}u_{\lambda_0\lambda_1\lambda_{01}}(g_0,g_1^{-1}x)
\operatorname{Ad}u_{\lambda_0\lambda_2\lambda_{02}}(g_0g_1,x)^*=1.
\end{align}
Then, the proof of Lemma \ref{onecocycle} shows that 
$\operatorname{Ad}u_{\lambda_0\lambda_1\lambda_{01}}$ is independent of 
$\lambda_{01}$. Now, we define a bundle $E\stackrel{p}{\to}X$ as the set
$E:= \{(\lambda_0,x,T):\lambda_0\in\mathcal{I}^0,x\in U^0_{\lambda_0},$ and $T\in\K \}/\sim,$
where $\sim$ is the equivalence relation
$(\lambda_0,x,T)\sim (\lambda_1,x,\operatorname{Ad}u_{\lambda_0\lambda_1\bullet}(e,x)T).$
The C*-algebra of sections $\Gamma_0(E,X)$ of $E$ is then a continuous trace 
C*-algebra. To define an action, let $a\in \Gamma_0(E,X)$,  $a(g^{-1}x)=[\lambda_0,g^{-1}x,T]$ 
and set
$(\alpha_ga)(x):=[\lambda_1,x,\operatorname{Ad}u_{\lambda_0\lambda_1\bullet}(g,x)T].$
To show that the definition is independent of the choice of representative 
$(\lambda_0,g^{-1}x,T)$ and index $\lambda_1$, choose another representative 
$(\lambda_0',g^{-1}x,T')$ and index $\lambda_1'$. By the definition of $\sim$ we have:
$(\lambda_0,g^{-1}x,T)\sim (\lambda_0',g^{-1}x,\operatorname{Ad}
u_{\lambda_0\lambda_0'\bullet}(e,g^{-1}x)T),$
which implies $k'=\operatorname{Ad}u_{\lambda_0\lambda_0'\bullet}(e,g^{-1}x)k$. 
This fact and the cocycle identities
$\operatorname{Ad}u_{\lambda_1'\lambda_1\bullet}(e,x)
\operatorname{Ad}u_{\lambda_0\lambda_1'\bullet}(g,x)
\operatorname{Ad}u_{\lambda_0\lambda_1\bullet}(g,x)^*=1,$ and
$\operatorname{Ad}u_{\lambda_0'\lambda_1'\bullet}(g,x)
\operatorname{Ad}u_{\lambda_0\lambda_0'\bullet}(e,g^{-1}x)
\operatorname{Ad}u_{\lambda_0\lambda_1'\bullet}(g,x)^*=1,$ imply
\begin{align*}
\quad\quad
[\lambda_1',x,\operatorname{Ad}u_{\lambda_0'\lambda_1'\bullet}(g,x)T']=[\lambda_1,x,
&\operatorname{Ad}u_{\lambda_1'\lambda_1\bullet}(e,x)u_{\lambda_0'\lambda_1'\bullet}
(g,x)u_{\lambda_0\lambda_0'\bullet}(e,g^{-1}x)T]\quad\quad\quad\quad\quad\quad\quad\quad\\
=[\lambda_1,x,&\operatorname{Ad}u_{\lambda_1'\lambda_1\bullet}(e,x)
u_{\lambda_0\lambda_1'\bullet}(g,x)T]\\
=[\lambda_1,x,&\operatorname{Ad}u_{\lambda_0\lambda_1\bullet}(g,x)T].
\end{align*}
Therefore $\alpha$ is well-defined, and we have an element 
$(\Gamma_0(E,X),\alpha)\in \mathfrak{Br}_G(X)$. By reversing this construction, 
it is easy to see  that $(\Gamma_0(E,X),\alpha)$ maps to the cocycle $\varphi$ 
under the homomorphism from Proposition \ref{mymaptotucech}.
\end{proof}

Thus we have injectivity (Corollary \ref{mymapinjective}) and surjectivity 
(Corollary \ref{UpsilonSurjective}) of $\Upsilon$
proving Theorem \ref{tusbigtheorem}.

% ==============================================================================
\section{Raeburn and Williams' Equivariant Cohomology}\label{inclusion}

Let $G$ be a locally compact abelian group and $N$ a closed subgroup such that 
$G\to G/N$ and $\hat{G}\to\hat{N}$ admit local sections  and let $\pi:X\to Z$ be a 
locally trivial principal $G/N$-bundle over a paracompact space $Z$. Recall the 
definition of $H^k_{G}(X,{\mathcal S})$ from the introduction, which we know from 
Lemma \ref{IsotoKerM} satisfies $H^k_{G}(X,{\mathcal S})\cong \ker \operatorname{M}$.
\begin{theorem}[{\cite[Thm 2.1]{RaeWil93}}]\label{invarianttheorem}
With $G$, $N$ as above, if $(CT(X,\delta),\alpha)$ is an $N$-principal system such that the 
$G$-action $\tau$ on $X$ induces a principal $G/N$-bundle $\pi:X\to Z$, then there is a 
locally finite cover $\{W_{\lambda_0}\}_{\lambda_0\in \mathcal{I}}$ of $Z$ by relatively 
compact open sets such that
\begin{enumerate}
\item[\emph{(1)}] for each $\lambda_0\in \mathcal{I}$ there is a $C_0(\pi^{-1}(W_{\lambda_0}))$-
isomorphism $\Phi_{\lambda_0}$ of $CT(X,\delta)|_{\pi^{-1}(W_{\lambda_0})}$ onto $C_0(\pi^{-1}
(W_{\lambda_0}),\K)$ which carries $\alpha$ to an action exterior equivalent to $\tau$, and
\item[\emph{(2)}] for each pair $\lambda_0,\lambda_1\in \mathcal{I}$, there is a unitary 
$v_{\lambda_0\lambda_1}\in UM(C_0(\pi^{-1}(W_{\lambda_0\lambda_1}),\K))$ such that 
$\Phi_{\lambda_1}\circ \Phi_{\lambda_0}^{-1}=\operatorname{Ad} v_{\lambda_0\lambda_1}$.
\end{enumerate}
\end{theorem}

So for each $\lambda_0\in \mathcal{I}$ the isomorphism $\Phi_{\lambda_0}$ 
maps $\alpha$ to an action exterior equivalent to $\tau$ on $C_0(\pi^{-1}(W_{\lambda_0}),\K)$. 
This means there are (strictly) continuous maps 
$u_{\lambda_0}:G\to C(\pi^{-1}(W_{\lambda_0}),U(\H))$ such that for any 
section $h_{\lambda_0}$ in $C_0(\pi^{-1}(W_{\lambda_0}),\K)$
\begin{align}\label{g1g2uequation}
\Phi_{\lambda_0}\circ\alpha_{g_1}\circ \Phi_{\lambda_0}^{-1}(h_{\lambda_0})= 
\operatorname{Ad}u_{\lambda_0}^{g_1}\circ \tau_{g_1}(h_{\lambda_0}),\quad\quad
u_{\lambda_0}^{g_1g_0}(x)=u_{\lambda_0}^{g_1}(x)u_{\lambda_0}^{g_0}(g_1^{-1}x).
\end{align}
It is also evident that for any $\lambda_0,\lambda_1\in \mathcal{I}$ and appropriate 
sections in $CT(X,\delta)|_{W_{\lambda_0\lambda_1}}$ that
{$\Phi_{\lambda_0}^{-1}\circ \operatorname{Ad}u_{\lambda_0}^{g_1}\circ \tau_{g_1}
\circ \Phi_{\lambda_0}=\Phi_{\lambda_1}^{-1}\circ \operatorname{Ad}u_{\lambda_1}^{g_1}
\circ \tau_{g_1}\circ \Phi_{\lambda_1}$,
which implies that we have the identity
$\Phi_{\lambda_1}\circ\Phi_{\lambda_0}^{-1}\circ \operatorname{Ad}u_{\lambda_0}^{g_1}\circ 
\tau_{g_1}\circ \Phi_{\lambda_0}\circ \Phi_{\lambda_1}^{-1}=\operatorname{Ad}u_{\lambda_1}^{g_1}
\circ \tau_{g_1}.$
This means that, if $h$ is a section in $C_0(\pi^{-1}(W_{\lambda_0\lambda_1}),\K)$ then
\begin{align*}
&\Phi_{\lambda_1}\circ\Phi_{\lambda_0}^{-1}\circ \operatorname{Ad}u_{\lambda_0}^{g_1}\circ 
\tau_{g_1}\circ \Phi_{\lambda_0}\circ \Phi_{\lambda_1}^{-1}(h)(x)
\\
= &v_{\lambda_0\lambda_1}(x)u_{\lambda_0}^{g_1}
(x)v_{\lambda_0\lambda_1}^*(g_1^{-1}x)h(g_1^{-1}x)v_{\lambda_0\lambda_1}
(g_1^{-1}x)u_{\lambda_0}^{g_1}(x)^*v_{\lambda_0\lambda_1}^*=u_{\lambda_1}^{g_1}
(x)h(g_1^{-1}x)u_{\lambda_1}^{g_1}(x)^*.
\end{align*}}
Hence there is a collection of continuous functions $\eta_{\lambda_0\lambda_1}:G\times \pi^{-1}
(W_{\lambda_0\lambda_1})\to \T$ defined by
\begin{equation}\label{rwetaeqn}
\eta_{\lambda_0\lambda_1}(g_1,x):=u_{\lambda_1}^{g_1}(x)v_{\lambda_0\lambda_1}
(g_1^{-1}x)u_{\lambda_0}^{g_1}(x)^*v_{\lambda_0\lambda_1}(x)^*.
\end{equation}
Finally, there is the usual cocycle $\nu_{\lambda_0\lambda_1\lambda_2}:\pi^{-1}
(W_{\lambda_0\lambda_1\lambda_2})\to \T$ given by
\begin{equation}\label{rwnueqn}
\nu_{\lambda_0\lambda_1\lambda_2}(x):=v_{\lambda_1\lambda_2}(x)v_{\lambda_0\lambda_1}
(x)v_{\lambda_0\lambda_2}^*(x),
\end{equation}
which clearly maps to the Dixmier-Douady class under 
$\Delta:\check{H}^2(X,{\mathcal S})\to \check{H}^3(X,\underline{\Z})$.

\begin{prop}
The pair $(\nu,\eta)$ defines a class $[\nu,\eta]$ in 
$H^2_{G}(\{\pi^{-1}(W_{\lambda_0})\},{\mathcal S})$ that is independent of the choice 
of local trivialisations $\Phi_{\lambda_0}$ and continuous maps 
$u_{\lambda_0}:G\to UM(C_0(\pi^{-1}(W_{\lambda_0}),\K))$.
\end{prop}
\begin{proof}
It is easy to see that $(\nu,\eta)$ defines a cochain in
$C_G^2(\{\pi^{-1}(W_{\lambda_0})\}_{\lambda_0},{\mathcal S})$.
Now we check the cocycle identities
\begin{equation}\label{cocycle1}
\check{\partial}\phi^{20}=1;\quad\partial_G\phi^{20}\check{\partial}\phi^{11}= 1; 
\quad{\mbox{and}}\quad
\partial_G\phi^{11}=1.
\end{equation}
That $\check\partial\nu=1$ is an easy computation. For the second identity, in (\ref{cocycle1}),
 we use the fact that certain groupings of terms take values in $ZU(\H)$, and so can be commuted: {
\begin{align*}
&\check\partial \eta(g_1,\cdot)_{\lambda_0\lambda_1\lambda_2}(x)
:=\eta_{\lambda_0\lambda_1}(g_1,x)\eta_{\lambda_1\lambda_2}(g_1,x)
\eta_{\lambda_0\lambda_2}(g_1,x)^*\quad\quad\quad\quad\\
&=u_{\lambda_1}^{g_1}(x)v_{\lambda_0\lambda_1}(g_1^{-1}x)u_{\lambda_0}^{g_1}
(x)^*v_{\lambda_0\lambda_1}(x)^* u_{\lambda_2}^{g_1}(x)v_{\lambda_1\lambda_2}
(g_1^{-1}x)u_{\lambda_1}^{g_1}(x)^*v_{\lambda_1\lambda_2}(x)^*\quad\quad\quad\quad\\
&\quad\times v_{\lambda_0\lambda_2}(x)u_{\lambda_0}^{g_1}(x)v_{\lambda_0\lambda_2}
(g_1^{-1}x)^*u_{\lambda_2}^{g_1}(x)^*\quad\quad\quad\quad\\
&=\nu_{\lambda_0\lambda_1\lambda_2}(x)^*[u_{\lambda_1}^{g_1}(x)v_{\lambda_0\lambda_1}
(g_1^{-1}x)u_{\lambda_0}^{g_1}(x)^*v_{\lambda_0\lambda_1}(x)^*] u_{\lambda_2}^{g_1}
(x)v_{\lambda_1\lambda_2}(g_1^{-1}x)u_{\lambda_1}^{g_1}(x)^*v_{\lambda_0\lambda_1}(x)\\
&\quad\times u_{\lambda_0}^{g_1}(x)v_{\lambda_0\lambda_2}(g_1^{-1}x)^*u_{\lambda_2}^{g_1}(x)^*\\
&=\nu_{\lambda_0\lambda_1\lambda_2}(x)^*u_{\lambda_2}^{g_1}(x)
v_{\lambda_1\lambda_2}(g_1^{-1}x)u_{\lambda_1}^{g_1}(x)^* [u_{\lambda_1}^{g_1}
(x)v_{\lambda_0\lambda_1}(g_1^{-1}x)u_{\lambda_0}^{g_1}(x)^*v_{\lambda_0\lambda_1}
(x)^*]v_{\lambda_0\lambda_1}(x)\\
&\quad\times u_{\lambda_0}^{g_1}(x)v_{\lambda_0\lambda_2}(g_1^{-1}x)^*u_{\lambda_2}^{g_1}(x)^*\\
&=\nu_{\lambda_0\lambda_1\lambda_2}(x)^*
u_{\lambda_2}^{g_1}(x)v_{\lambda_1\lambda_2}(g_1^{-1}x)v_{\lambda_0\lambda_1}(g_1^{-1}x) 
v_{\lambda_0\lambda_2}(g_1^{-1}x)^*u_{\lambda_2}^{g_1}(x)^*
\quad\quad\quad\quad\quad\quad\quad\quad\quad\\
&=\nu_{\lambda_0\lambda_1\lambda_2}(x)^*\nu_{\lambda_0\lambda_1\lambda_2}
(g_1^{-1}x)=(\partial_G\nu_{\lambda_0\lambda_1\lambda_2}(g_1,x))^{-1}.
\quad\quad\quad\quad\quad\quad\quad\quad\quad\quad\quad\quad\quad\quad
\end{align*}}
For the last identity, in (\ref{cocycle1}), we use the same technique as above, 
Lemma \ref{cyclicpermute} and the fact that $G$ is abelian:
\begin{align*}
&\partial_G\eta_{\lambda_0\lambda_1}(g_0,g_1,x)
:=u_{\lambda_1}^{g_0}(g_1^{-1}x)v_{\lambda_0\lambda_1}(g_0^{-1}g_1^{-1}x)
u_{\lambda_0}^{g_0}(g_1^{-1}x)^*v_{\lambda_0\lambda_1}(g_1^{-1}x)^*
u_{\lambda_1}^{g_1}(x)\\&\quad\times v_{\lambda_0\lambda_1}(g_1^{-1}x)
u_{\lambda_0}^{g_1}(x)^*v_{\lambda_0\lambda_1}(x)^*
v_{\lambda_0\lambda_1}(x)u_{\lambda_0}^{g_0g_1}(x)
v_{\lambda_0\lambda_1}(g_0^{-1}g_1^{-1}x)^*u_{\lambda_1}^{g_0g_1}(x)^*\\
&=[u_{\lambda_1}^{g_0}(g_1^{-1}x)v_{\lambda_0\lambda_1}(g_0^{-1}g_1^{-1}x)
u_{\lambda_0}^{g_0}(g_1^{-1}x)^*v_{\lambda_0\lambda_1}(g_1^{-1}x)^*]
u_{\lambda_1}^{g_1}(x)v_{\lambda_0\lambda_1}(g_1^{-1}x)u_{\lambda_0}^{g_1}(x)^*\\
&\quad\times u_{\lambda_0}^{g_0g_1}(x)v_{\lambda_0\lambda_1}
(g_0^{-1}g_1^{-1}x)^*u_{\lambda_1}^{g_0g_1}(x)^*\\
&=u_{\lambda_1}^{g_1}(x)[u_{\lambda_1}^{g_0}(g_1^{-1}x)v_{\lambda_0\lambda_1}
(g_0^{-1}g_1^{-1}x)u_{\lambda_0}^{g_0}(g_1^{-1}x)^*v_{\lambda_0\lambda_1}
(g_1^{-1}x)^*]v_{\lambda_0\lambda_1}(g_1^{-1}x)u_{\lambda_0}^{g_1}(x)^*\\
&\quad\times u_{\lambda_0}^{g_0g_1}(x)v_{\lambda_0\lambda_1}
(g_0^{-1}g_1^{-1}x)^*u_{\lambda_1}^{g_0g_1}(x)^*\\
&=u_{\lambda_1}^{g_1}(x)u_{\lambda_1}^{g_0}(g_1^{-1}x)v_{\lambda_0\lambda_1}
(g_0^{-1}g_1^{-1}x)u_{\lambda_0}^{g_0}(g_1^{-1}x)^*u_{\lambda_0}^{g_1}(x)^* 
u_{\lambda_0}^{g_0g_1}(x)v_{\lambda_0\lambda_1}(g_0^{-1}g_1^{-1}x)^*
u_{\lambda_1}^{g_0g_1}(x)^*\\
&= 1\quad\quad\mbox{by } (\ref{g1g2uequation}) \mbox{ and Lemma } \ref{cyclicpermute}.
\end{align*}
To see that the class $[\nu,\eta]$ is independent of the choice of $\{\Phi_{\lambda_0}\}$ and 
$\{u_{\lambda_0}\}$, take local trivialisations 
$\Psi_{\lambda_0}:CT(X,\delta)|_{\pi^{-1}(W_{\lambda_0})}\to C_0(\pi^{-1}(W_{\lambda_0}),\K)$,  
maps $\tilde{v}_{\lambda_0\lambda_1}\in C_0(\pi^{-1}(W_{\lambda_0}),\K)$, 
$w_{\lambda_0}\in C(\pi^{-1}(W_{\lambda_0}),U(\H))$ and strictly continuous maps 
$\tilde{u}_{\lambda_0}:G\to C_0(\pi^{-1}(W_{\lambda_0}),\K)$ with
\begin{align}
\label{Difference1}&\Psi_{\lambda_1}\circ \Psi_{\lambda_0}^{-1}=\operatorname{Ad} 
\tilde{v}_{\lambda_0\lambda_1},\\
\label{Difference2}&\Psi_{\lambda_0}\circ \Phi_{\lambda_0}^{-1}=\operatorname{Ad}w_{\lambda_0},\\
\label{Difference3}&\Psi_{\lambda_0}\circ\alpha_{g_1}\circ \Psi_{\lambda_0}^{-1}
(h_{\lambda_0})= \operatorname{Ad}\tilde{u}_{\lambda_0}^{g_1}\circ \tau_{g_1}(h_{\lambda_0}),\\
\label{Difference4}&\tilde{u}_{\lambda_0}^{g_1g_0}(x)=\tilde{u}_{\lambda_0}^{g_1}(x)
\tilde{u}_{\lambda_0}^{g_0}(g_1^{-1}x).
\end{align}
Let $(\tilde{\nu},\tilde{\eta})$ be the cocycle associated to the choices $\{\Psi_{\lambda_0}\}$ 
and $\{\tilde{u}_{\lambda_0}\}$ then using Equations (\ref{Difference1}) and (\ref{Difference2}) 
we see that there exist continuous functions $\gamma_{\lambda_0\lambda_1}\in C(\pi^{-1}
(W_{\lambda_0\lambda_1}),\T)$ such that
$\tilde{v}_{\lambda_0\lambda_1}=
\gamma_{\lambda_0\lambda_1}w_{\lambda_1}v_{\lambda_0\lambda_1}w_{\lambda_0}^*$.
Moreover, Equations (\ref{Difference3}) and (\ref{Difference2}) imply

\begin{align*}
\operatorname{Ad}\tilde{u}_{\lambda_0}^{g_1}\circ \tau_{g_1}(h_{\lambda_0})=
&\Psi_{\lambda_0}\circ\alpha_{g_1}\circ \Psi_{\lambda_0}^{-1}(h_{\lambda_0})
=\Psi_{\lambda_0}\circ\Phi_{\lambda_0}^{-1}\circ \operatorname{Ad}u_{\lambda_0}^{g_1}
\circ \tau_{g_1}\circ \Phi_{\lambda_0}^{-1}\circ \Psi_{\lambda_0}^{-1}(h_{\lambda_0})\\
=&\operatorname{Ad}w_{\lambda_0}\circ \operatorname{Ad}u_{\lambda_0}^{g_1}\circ 
\tau_{g_1}\circ \operatorname{Ad}w_{\lambda_0}^*(h_{\lambda_0}).
\end{align*}
It follows that there exist continuous functions $\sigma_{\lambda_0}:G\to 
C(\pi^{-1}(W_{\lambda_0}),\T)$ such that
\begin{equation}\label{sigmaeqn}
\tilde{u}_{\lambda_0}^{g_1}(x)=\sigma_{\lambda_0}(g_1,x)
w_{\lambda_0}(x)u_{\lambda_0}^{g_1}(x)w_{\lambda_0}(g_1^{-1}x)^*.
\end{equation}

We now claim that $(\gamma,\sigma)$ is a cochain and $(\tilde{\nu},\tilde{\eta})$ 
differs from $(\nu,\eta)$ by the differential of $(\gamma,\sigma)$.  That $\partial_G\sigma=1$ 
follows from a calculation identical to the proof that $\partial_G\eta=1$, using Equation (\ref{sigmaeqn}).
As it is straightforward to check $\tilde{\nu}=(\check\partial\gamma) \nu$, to 
prove our claim consider the other component
$\tilde{\eta}_{\lambda_0\lambda_1}(g_1,x):=\tilde{u}_{\lambda_1}^{g_1}(x)
\tilde{v}_{\lambda_0\lambda_1}(g_1^{-1}x)\tilde{u}_{\lambda_0}^{g_1}(x)^*
\tilde{v}_{\lambda_0\lambda_1}(x)^*$, which can be written
\begin{align*}
&\sigma_{\lambda_1}(g_1,x)w_{\lambda_1}(x)u_{\lambda_1}^{g_1}(x)
w_{\lambda_1}^*(g_1^{-1}x)\gamma_{\lambda_0\lambda_1}(g_1^{-1}x)
w_{\lambda_1}(g_1^{-1}x)v_{\lambda_0\lambda_1}(g_1^{-1}x)w_{\lambda_0}^*(g_1^{-1}x)\\
&\quad\times w_{\lambda_0}(g_1^{-1}x)u_{\lambda_0}^{g_1}(x)^*w_{\lambda_0}(x)^*
\sigma_{\lambda_0}(g_1,x)^* w_{\lambda_0}(x)v_{\lambda_0\lambda_1}(x)^*
w_{\lambda_1}(x)^*\gamma_{\lambda_0\lambda_1}(x)^*\\
&=\sigma_{\lambda_1}(g_1,x)\gamma_{\lambda_0\lambda_1}(g_1^{-1}x)
\sigma_{\lambda_0}(g_1,x)^*\gamma_{\lambda_0\lambda_1}(x)^* 
u_{\lambda_1}^{g_1}(x) v_{\lambda_0\lambda_1}(g_1^{-1}x)u_{\lambda_0}^{g_1}(x)^* 
v_{\lambda_0\lambda_1}(x)^* \\
&=\sigma_{\lambda_1}(g_1,x)\gamma_{\lambda_0\lambda_1}(g_1^{-1}x)
\sigma_{\lambda_0}(g_1,x)^*\gamma_{\lambda_0\lambda_1}(x)^*\eta_{\lambda_0\lambda_1}(g_1,x)\\
&=\check\partial\sigma (g_1,\cdot)_{\lambda_0\lambda_1}(x)(\partial_G
\gamma_{\lambda_0\lambda_1}(g_1,x))^{-1}\eta_{\lambda_0\lambda_1}(g_1,x).
\end{align*}
\end{proof}

We have established that the image of $(CT(X,\delta),\alpha)\in\mathfrak{Br}_G(X)$ in 
$H^2_{G}(X,{\mathcal S})$ is $[\nu,\eta]$. One can then check, using the techniques in 
the proof above and those in the proof of Proposition \ref{mymaptotucech}, that this induces 
a map from $\ker \operatorname{M}\subset \operatorname{Br}_G(X)$ to $H^2_{G}(X,{\mathcal S})$ 
(that is, the image is constant on outer conjugacy classes).

\begin{theorem}\label{PRWagreement}
Let $\pi:X\to Z$ be a principal $\R^n/\Z^n$-bundle over a $C^\infty$ 
manifold $Z$, and let $F\in\check{Z}^2(\W,\underline{\Z}^n)$ be a representative of the 
Euler vector defined over a good open cover $\W$ of $Z$. Then there is a commutative 
diagram with exact rows:

\centerline{\xymatrix{
\ar[r]&\check{H}^k(\W,{\mathcal S})\ar[r]^{\pi^*_{\R^n}}\ar[d]^{\operatorname{id}}& 
H^k_{\R^n}(\pi^{-1}(\W),{\mathcal S})\ar[r]^{\pi_*}\ar[d]&\check{H}^{k-1}(\W,\hat{{\mathcal N}})
\ar[r]^{\cup F}\ar[d]&\check{H}^{k+1}(\W,{\mathcal S})\ar[r]\ar[d]^{\operatorname{id}}&\\
\ar[r]&\check{H}^{k}(\W,{\mathcal S})\ar[r]^{\pi^*}& {\mathbb H}^{k}_F(\W,{\mathcal S})\ar[r]^{\pi_*}
& \overline{{\mathbb H}}^{k-1}_{F}(\W,{\mathcal S})\ar[r]^{\cup F}& \check{H}^{k+1}(\W,{\mathcal S})
\ar[r]&}}

\end{theorem}

\begin{proof}
We need only define the downward arrows $H^k_{\R^n}(\pi^{-1}(\W),{\mathcal S})\to 
{\mathbb H}^{k}_F(\W,{\mathcal S})$ and $\check{H}^{k-1}(\W,\hat{{\mathcal N}})\to 
\overline{{\mathbb H}}^k_F(\W,{\mathcal S})$, and the commutativity will follow 
immediately from the definitions. Indeed, if $(\nu,\eta)$ is a cochain in 
$C^k_{\R^n}(\pi^{-1}(\W),{\mathcal S})$, its image in $C^{k}_F(\W,{\mathcal S})$ 
is defined to be the triple $(\phi^{k0}(\nu,\eta),\phi^{(k-1)1}(\nu,\eta),\phi^{(k-2)2}(\nu,\eta))$ given by
\begin{align*}
\phi(\nu,\eta)^{k0}_{\mu_0\dots\mu_k}(z):=&\nu_{\mu_0\dots\mu_k}(\sigma_{\mu_k}(z))
\eta_{\mu_0\dots\mu_{k-1}}(-s_{\mu_{k-1}\mu_k}(z),\sigma_{\mu_k}(z)),\\
\phi(\nu,\eta)^{(k-1)1}_{\mu_0\dots\mu_{k-1}}(m,z):=&\eta_{\mu_0\dots\mu_{k-1}}(m,
\sigma_{\mu_{k-1}}(z)),\quad\mbox{and}\quad
\phi(\nu,\eta)^{(k-2)2}:=1.
\end{align*}
In order for this map to be well-defined on cohomology, it needs to commute with the 
respective differentials. To see that it does, we compute:

\begin{align*}
&[\check\partial\phi(\nu,\eta)^{k0}+(-1)^{k+1}\phi(\nu,\eta)^{(k-1)1}\cup_1 F + (-1)^{k+1}
\phi(\nu,\eta)^{(k-2)2}\cup_2 C(F)]_{\mu_0\dots\mu_{k+1}}(z)  \\
&=\oplus_{i=0}^k (-1)^i\nu_{\mu_0\dots\hat{\mu_i}\dots\mu_{k+1}}
(\sigma_{\mu_{k+1}}(z)) +(-1)^{k+1}\nu_{\mu_0\dots\mu_k}(\sigma_{\mu_k}(z))\\
&\quad+\oplus_{i=0}^{k-1}(-1)^i\eta_{\mu_0\dots\hat{\mu_i}\dots\mu_{k}}(-s_{\mu_{k}\mu_{k+1}}(z),
\sigma_{\mu_{k+1}}(z))\\
&\quad +(-1)^k\eta_{\mu_0\dots\mu_{k-1}}(-s_{\mu_{k-1}\mu_{k+1}},
\sigma_{\mu_{k+1}}(z)) +(-1)^{k+1}\eta_{\mu_0\dots\mu_{k-1}}(-s_{\mu_{k-1}\mu_{k}}(z),
\sigma_{\mu_{k}}(z))\\
&\quad+(-1)^{k+1}\eta_{\mu_0\dots\mu_{k-1}}(F_{\mu_{k-1}\mu_{k}\mu_{k+1}}(z),\sigma_{\mu_{k-1}}(z))\\
&=\check\partial\nu _{\mu_0\dots\mu_{k+1}}(\sigma_{\mu_{k+1}}(z))+(-1)^k
\partial_{\R^n}(\nu_{\mu_0\dots\mu_{k}}) (-s_{\mu_{k}\mu_{k+1}}(z),\sigma_{\mu_{k+1}}(z)) \\
&\quad+\check\partial\eta_{\mu_0\dots\mu_{k}}(-s_{\mu_{k}\mu_{k+1}},
\sigma_{\mu_{k+1}}(z))+(-1)^{k+1}\eta_{\mu_0\dots\mu_{k-1}}(-s_{\mu_{k}\mu_{k+1}},
\sigma_{\mu_{k+1}}(z))\quad\quad\quad\\
&\quad  (-1)^{k}\eta_{\mu_0\dots\mu_{k-1}}(-s_{\mu_{k-1}\mu_{k}}(z)-
s_{\mu_{k}\mu_{k+1}}(z)-F_{\mu_{k-1}\mu_k\mu_{k+1}}(z),\sigma_{\mu_{k+1}}(z)) \\
&\quad+(-1)^{k+1}\eta_{\mu_0\dots\mu_{k-1}}(-s_{\mu_{k-1}\mu_{k}}(z),
s_{\mu_{k}\mu_{k+1}}(z)\cdot\sigma_{\mu_{k+1}}(z))\\
&\quad+(-1)^{k}\eta_{\mu_0\dots\mu_{k-1}}(-F_{\mu_{k-1}\mu_{k}\mu_{k+1}}(z),\sigma_{\mu_{k+1}}(z)).
\end{align*}
In the last line above we have used  Equation (\ref{etainvariantonorbits}) 
(i.e. $\eta|_{\Z^n}$ is constant on orbits). Now we twice use the fact that 
$\partial_{\R^n}\eta=1$ so the previous expression is equal to
\begin{align*}
&\check\partial\nu _{\mu_0\dots\mu_{k+1}}(\sigma_{\mu_{k+1}}(z))+(-1)^{k}
\partial_{\R^n}(\nu_{\mu_0\dots\mu_{k}}) (-s_{\mu_{k}\mu_{k+1}}(z),\sigma_{\mu_{k+1}}(z)) \\
&\quad+\check\partial\eta_{\mu_0\dots\mu_{k}}(-s_{\mu_{k}\mu_{k+1}}(z),\sigma_{\mu_{k+1}}(z))
=\phi^{(k+1)0}(\check\partial\nu,(-1)^{k}\partial_{\R^n}\nu+\check\partial\eta)_{\mu_0\dots\mu_{k+1}}(z).
\end{align*}
On the other hand, we also have
\begin{align*}
[\check\partial\phi(\nu,\eta)^{(k-1)1}+&(-1)^{k}\phi(\nu,\eta)^{(k-2)2}\cup_1 F]_{\mu_0\dots\mu_{k}}(m,z)
=\oplus_{i=0}^{k}(-1)^i\phi(\nu,\eta)^{(k-1)1}_{\mu_0\dots\hat{\mu_i}\dots\mu_{k}}(m,z)\\
=&\oplus_{i=0}^{k-1}(-1)^i\eta_{\mu_0\dots\hat{\mu_i}\dots\mu_{k}}(m,
\sigma_{\mu_{k}}(z))-\eta_{\mu_0\dots\mu_{k-1}}(m,\sigma_{\mu_{k-1}}(z))\quad\\
=&\oplus_{i=0}^{k-1}(-1)^i\eta_{\mu_0\dots\hat{\mu_i}\dots\mu_{k}}(m,\sigma_{\mu_{k}}(z))-
\eta_{\mu_0\dots\mu_{k-1}}(m,\sigma_{\mu_{k}}(z))\quad\quad\\
=&\check\partial_{\mu_0\dots\mu_{k}}\eta(m,\sigma_{\mu_{k}}(z))
=\phi^{k1}(\check\partial\nu,(-1)^{k}\partial_{\R^n}\nu+\check\partial\eta)_{\mu_0\dots\mu_{k}}(m,z)
\end{align*}
giving the required property of
$(\nu,\eta)\mapsto(\phi(\nu,\eta)^{k0},\phi(\nu,\eta)^{(k-1)1},\phi(\nu,\eta)^{(k-2)2})$.

For the other vertical arrow $\check{H}^{k-1}(\W,\hat{{\mathcal N}})\to 
\overline{{\mathbb H}}^k_F(\W,{\mathcal S})$, the image of 
$\gamma\in \check{Z}^{k-1}(\W,\hat{{\mathcal N}})$ is
$\varphi(\gamma)^{(k-1)1}_{\mu_0\dots\mu_{k-1}}(m,z):=\gamma_{\mu_0\dots\mu_{k-1}}(m,z)$ and $
\varphi(\gamma)^{(k-2)2}:=1$.
This clearly induces a homomorphism on cohomology groups, and commutativity 
of the diagram is immediate.
\end{proof}

% ==============================================================================
\section{Good Covers of Principal $\T^n$-Bundles}\label{goodcoverssection}

For this Section we work only with Riemannian manifolds. So, let $\pi:X\to Z$ 
be a $C^\infty$ principal $\R^n/\Z^n$-bundle over a Riemannian manifold $Z$. 
We will show there is a good cover of $X$ that pushes down to a good cover of $Z$.
\newpage %%%%

\begin{defn}[\cite{Spi79}] $\quad$
\begin{itemize}
\item[(i)] Let $(M,g)$ be a Riemannian manifold, and $U\subset M$. Then $U$ is 
said to be \emph{geodesically convex} if every two points $x,y\in U$ have a unique 
geodesic of minimum length between them, and that geodesic lies entirely within $U$.
\item[(ii)] Let $\U$ be an open cover of a Riemannian manifold $(M,g)$. Then 
we call $\U$ \emph{geodesically convex} if every open set in $\U$ is geodesically convex.
\end{itemize}
\end{defn}

Note that every geodesically convex open set is contractible, and every 
geodesically convex open cover is good \cite{Spi79}.

\begin{lemma}\label{geoconvex}
Let $(Z,dz^2)$ be a Riemannian manifold and $\pi:X\to Z$ a $C^\infty$ 
principal $\R^n/\Z^n$-bundle. Then $X$ has a geodesically convex open cover 
$\U=\{U_{\lambda}\}_{\lambda\in\mathcal{I}}$ such that $\pi(\U):=\{\pi(U_{\lambda})\}_{\lambda\in
\mathcal{I}}$ is a geodesically convex open cover of  $Z$. Moreover, there exist $C^\infty$ local 
sections $\sigma_{\lambda}:\pi(U_\lambda)\to U_\lambda$.
\end{lemma}

\begin{proof}
Fix a connection 1-form $\omega$ on $\pi:X\to Z$, and let $<,>$ denote a bi-invariant 
metric on the Lie algebra $t$ of $\R^n/\Z^n$. Therefore we have a $\R^n$-bi-invariant 
Riemannian metric on $X$ defined by $g:=\pi^*dz^2+<\omega,\omega>$. Moreover, 
with respect to the metrics $g$ and $dz^2$, the projection $\pi: X\to Z$ is a Riemannian submersion.

Now let $\U=\{U_{\lambda}\}$ be an open cover of $X$ consisting of geodesically 
convex sets defined as follows. For each $x\in X$ choose $\epsilon_0>0$ so that
$B_x(\epsilon_0)=\exp\{v\in TX_x:||v||<\epsilon_0\}$
is geodesically convex (this is possible by \cite[Vol. 1, Ex. 32, p.491]{Spi79}). 
We claim, perhaps after choosing a smaller $\epsilon_0$, 
that $\pi(\U)=\{\pi(U_\lambda)\}$ is a geodesically convex open cover of $Z$.

To see that this is the case, observe that \cite[Vol. 2, Ch. 8, Prop 7]{Spi79} and 
\cite[Cor 1.1]{FalStePas04} show that a curve $\gamma$ in $Z$ is a geodesic in $Z$ 
if and only if its unique lift to a horizontal curve in $X$ is a geodesic. Moreover, 
\cite[Prop 1.10]{FalStePas04} says that, if $\gamma:I\to X$ is a geodesic such that 
$\dot{\gamma}(t_0)$ is horizontal at $x=\gamma(t_0)$, then $\gamma$ is horizontal. 
Finally, since $\pi$ is a \emph{Riemannian} submersion, the map $d\pi:TX_x\to TZ_{\pi(x)}$ 
is a surjection that preserves the length of horizontal vectors. Therefore, if $\epsilon_1>0$ 
is such that $B_{\epsilon_1}(\pi(x))$ is geodesically convex in $Z$, let $\epsilon=\min
\{\epsilon_0,\epsilon_1\}$, and then $\pi(B_{\epsilon}(x))=B_{\epsilon}(\pi(x))$.

For the last claim, if $z_0\in B_{\epsilon}(\pi(x))$, then there exists a unique geodesic $\xi$ 
contained in $B_{\epsilon}(\pi(x))$ for which there exists $t\in [0,\epsilon)$ with 
$\xi(0)=\pi(x), \xi(t)=z_0$. From the above, there is a unique (horizontal) geodesic 
$\gamma$ in $B_{\epsilon}(x)$ such that $\pi(\gamma)=\xi$. Then we define 
$\sigma_{\lambda}(z_0):=\gamma(t)$.  This is a $C^\infty$ section by 
\cite[Vol. 1, Ch. 9, Thm 14(2)]{Spi79}.
\end{proof}

% ==============================================================================
\section[Isomorphism With The Equivariant Brauer Group]{Isomorphism 
With The Equivariant Brauer Group}

In this Section we prove our main result, Theorem \ref{MainSquare}.
Recall the projection $\pi^{0,3}$ of $\check{H}^3(X,\underline{\Z})$ onto the 
$E^{0,3}_\infty$ term of the Leray-Serre spectral sequence of $\pi:X\to Z$. Since
\[E^{0,3}_\infty=\{f\in C(Z,\check{H}^3(\R^n/\Z^n,\underline{\Z})):d_2f=d_3f=d_4f=0\},\]
we can view $\pi^{0,3}$ as a map $\pi^{0,3}:\check{H}^3(X,\underline{\Z})\to 
C(Z,\check{H}^3(\R^n/\Z^n,\underline{\Z}))$, that can be calculated as follows. 
For each $x\in X$, let $\iota_x:\R^n/\Z^n\to X$ be the fibre inclusion $\iota_x:[t]\mapsto [-t]\cdot x$. 
Then, for $\delta\in\check{H}^3(X,\underline{\Z})$ we have
$(\pi^{0,3}\delta)(\pi(x))=\iota^*_x\delta.$
We will denote the kernel of $\pi^{0,3}$ by $\check{H}^3(X,\underline{\Z})|_{\pi^{0,3}=0}$.  
We begin our argument by introducing some notation.

\begin{defn}
Let $\pi:X\to Z$ be a $C^\infty$ principal $\R^n/\Z^n$-bundle over a Riemannian 
manifold $Z$, and let $\U=\{U_{\lambda_0}\}_{\lambda_0\in\mathcal{I}}$ be a fixed 
good open cover of $X$ such that $\pi(\U)$ is a good open cover of $Z$. Also fix $C^\infty$ 
local sections $\sigma_{\lambda_0}:\pi(U_{\lambda_0})\to U_{\lambda_0}$ and a 
cochain $s\in \check{C}^1(\pi(\U),\mathcal{\R}^n)$ such that, for all indices 
$\lambda_0,\lambda_1\in\mathcal{I}$, and all $z\in \pi(U_{\lambda_0})\cap \pi(U_{\lambda_1})$, 
we have
$s_{\lambda_0\lambda_1}(z)\cdot\sigma_{\lambda_1}(z)=\sigma_{\lambda_0}(z),$
and
$s_{\lambda_1\lambda_1}(z)=0.$
Then we say the $(X,\U,s)$ is \emph{in the standard setup} (with the space $Z$, projection 
$\pi:X\to Z$ and local sections $\{\sigma_{\lambda_0}\}$ implicit).
\end{defn}

Note that Lemma \ref{geoconvex} implies that every $C^\infty$ principal 
$\R^n/\Z^n$-bundle over a Riemannian manifold can be put in the standard setup.

Our major task is to define a homomorphism $\Xi_{\U,s}:\mbox{Br}_{\R^n}(X)
\to {\mathbb H}^2_{\check{\partial}s}(\pi(\U),{\mathcal S})$
that will eventually provide our required isomorphism. We first recall:
\begin{lemma}[{\cite[Prop 4.27]{RaeWill98}}]\label{zetavanishremark}
Let $U$ be a contractible paracompact locally compact space, and 
$\beta:U\to\operatorname{Aut}\K$ a continuous map. 
Then there exists a continuous map $u:U\to U(\H)$ such that $\beta=\operatorname{Ad}u$.
\end{lemma}

Now let $(CT(X,\delta),\alpha)$ be an element of $\mathfrak{Br}_{\R^n}(X)$. 
In order to define $\Xi_{\U,s}([CT(X,\delta),\alpha])$ we need to revisit the construction 
of the map $\beta^{\alpha,\Phi}:(\R^n\ltimes X[\U])^{(1)}\to\operatorname{Aut}\K$ that was 
used as an intermediate stage in proving the isomorphism of $\mbox{Br}_{\R^n}(X)$ with 
$\check{H}^2(\R^n\ltimes X,{\mathcal S})$.
\begin{defn}\label{prelims}
(i) As $CT(X,\delta)$ is trivialised over $\U$, we fix local trivialisations 
$\Phi_{\lambda}:CT(X,\delta)|_{U_{\lambda}}\to C_0(U_{\lambda},\K)$ so that 
$\beta^{\alpha,\Phi}$ may be defined in this setting by
$\beta^{\alpha,\Phi}_{(\lambda_0(s,x)\lambda_1)}(T):=\Phi_{\lambda_1}(\alpha_s(a))(x),$
where $a\in CT(X,\delta)$ is any element such that $\Phi_{\lambda_0}(a)(-s\cdot x)=T$. \\
(ii) Introduce unitary valued lifts of $\beta^{\alpha,\Phi}$ by letting $\{e_i\}$ be 
generators of $\Z^n$ and recalling from Section \ref{Mackeysectiononeonetwo} that, 
for the  trivialisations $\{\Phi_{\lambda_0}\}$, the maps 
$\Phi_{\lambda_0}\circ\alpha|_{\Z^n}\circ\Phi_{\lambda_0}^{-1}$ are locally inner,
so by {Lemma} \ref{zetavanishremark} there exist continuous maps 
$v_{\lambda_0}^i: U_{\lambda_0}\to U(\H)$ such that 
$\Phi_{\lambda_0}\circ\alpha_{e_i}\circ\Phi_{\lambda_0}^{-1}=
\operatorname{Ad}v_{\lambda_0}^i$. Then we may define for any $m\in \Z^n$
\begin{equation}\label{vsplit}
v_{\lambda_0}^m(z):=(v_{\lambda_0}^1(\sigma_{\lambda_0}(z)))^{m_1}
(v_{\lambda_0}^2(\sigma_{\lambda_0}(z)))^{m_2}\dots 
(v_{\lambda_0}^n(\sigma_{\lambda_0}(z)))^{m_n},
\end{equation}
which gives
\begin{equation}\label{betav01}
\beta^{\alpha,\Phi}_{(\lambda_0(m,\sigma_{\lambda_0}(z))\lambda_0)}(T)=
\operatorname{Ad}v_{\lambda_0}^m(z)(T).
\end{equation}
\end{defn}
There is a map from $\pi(U_{\lambda_0})\cap \pi(U_{\lambda_1})\to \operatorname{Aut}\K$ given by
$z\mapsto \beta^{\alpha,\Phi}_{(\lambda_0(-s_{\lambda_0\lambda_1}(z),
\sigma_{\lambda_1}(z))\lambda_1)}.$
Since $\pi(\U)$ is a good cover of $Z$, the open set $\pi(U_{\lambda_0})\cap \pi(U_{\lambda_1})$ 
is contractible, and there exists a continuous map $v_{\lambda_0\lambda_1}:\pi(U_{\lambda_0})\cap
\pi(U_{\lambda_1})\to U(\H)$ that satisfies
\begin{equation}\label{betav10}
\beta^{\alpha,\Phi}_{(\lambda_0(-s_{\lambda_0\lambda_1}(z),\sigma_{\lambda_1}(z))\lambda_1)}=
\operatorname{Ad}v_{\lambda_0\lambda_1}(z).
\end{equation}

Next, $\Xi_{\U,s}[CT(X,\delta),\alpha]$ will be a triple $[\phi(\alpha)^{20},\phi(\alpha)^{11},
\phi(\alpha)^{02}]$.
Define the last component $\phi(\alpha)^{02}\in \check{C}^0(\pi(\U),\mathcal{M})$ by:
\begin{equation}\phi(\alpha)^{02}_{\lambda_0}(z)_{ij}:=v_{\lambda_0}^j(\sigma_{\lambda_0}
(z))v_{\lambda_0}^i(\sigma_{\lambda_0}(z))v_{\lambda_0}^{e_i+e_j}(\sigma_{\lambda_0}(z))^*,
\end{equation}
(where $v_{\lambda_0}^{e_i+e_j}(\sigma_{\lambda_0}(z))^*$ is defined using Equation (\ref{vsplit})).  
For the middle component, note that since $\alpha$ is a homomorphism and $\R^n$ is abelian
\begin{align*}
\operatorname{Ad} v_{\lambda_0\lambda_1}(z)v_{\lambda_0}^m(z)=&
\beta^{\alpha,\Phi}_{(\lambda_0(-s_{\lambda_0\lambda_1}(z),
\sigma_{\lambda_1}(z))\lambda_1)}\beta^{\alpha,\Phi}_{(\lambda_0(m,\sigma_{\lambda_0}(z))
\lambda_0)}
=\beta^{\alpha,\Phi}_{(\lambda_0(-s_{\lambda_0\lambda_1}(z)+m,\sigma_{\lambda_1}(z))\lambda_1)}\\
=&\beta^{\alpha,\Phi}_{(\lambda_0(m-s_{\lambda_0\lambda_1}(z),\sigma_{\lambda_1}(z))\lambda_1)}
=\beta^{\alpha,\Phi}_{(\lambda_1(m,\sigma_{\lambda_1}(z))\lambda_1)}
\beta^{\alpha,\Phi}_{(\lambda_0(-s_{\lambda_0\lambda_1}(z),\sigma_{\lambda_1}(z))\lambda_1)}\\
=&\operatorname{Ad} v_{\lambda_1}^m(z)v_{\lambda_0\lambda_1}(z).
\end{align*}
We prove in  Lemma \ref{triplecocycle} that the map
$m\mapsto v_{\lambda_1}^m(z)v_{\lambda_0\lambda_1}
(z)v_{\lambda_0}^m(z)^*v_{\lambda_0\lambda_1}(z)^*$
is a homomorphism from $\Z^n$ to $\T$. This being the case, we 
may define $\phi(\alpha)^{11}\in\check{C}^1(\pi(\U),\hat{{\mathcal N}})$ by

\begin{equation}\label{phi11defn}
\phi(\alpha)^{11}_{\lambda_0\lambda_1}(m,z):=v_{\lambda_1}^m(z)
v_{\lambda_0\lambda_1}(z)v_{\lambda_0}^m(z)^*v_{\lambda_0\lambda_1}(z)^*.
\end{equation}
Finally,  to define the  third 
component, we use a similar calculation to the previous one:
\begin{align*}
&\operatorname{Ad}  v_{\lambda_1\lambda_2}(z)v_{\lambda_0\lambda_1}(z)=
\beta^{\alpha,\Phi}_{(\lambda_1(-s_{\lambda_1\lambda_2}(z),
\sigma_{\lambda_2}(z))\lambda_2)}\beta^{\alpha,\Phi}_{(\lambda_0(-s_{\lambda_0\lambda_1}(z),
\sigma_{\lambda_1}(z))\lambda_1)}\\
&\quad=\beta^{\alpha,\Phi}_{(\lambda_0(-s_{\lambda_0\lambda_1}(z)-
s_{\lambda_1\lambda_2}(z),\sigma_{\lambda_2}(z))\lambda_2)}
=\beta^{\alpha,\Phi}_{(\lambda_0(-s_{\lambda_0\lambda_2}(z)-\check\partial 
s_{\lambda_0\lambda_1\lambda_2}(z),\sigma_{\lambda_2}(z))\lambda_2)}\\
&\quad=\beta^{\alpha,\Phi}_{(\lambda_0(-s_{\lambda_0\lambda_2}(z),\sigma_{\lambda_2}(z))
\lambda_2)}\beta^{\alpha,\Phi}_{(\lambda_0(-\check\partial s_{\lambda_0\lambda_1\lambda_2}(z),
\sigma_{\lambda_2}(z))\lambda_0)}
=\operatorname{Ad} v_{\lambda_0\lambda_2}(z)v_{\lambda_0}^{-\check\partial
 s_{\lambda_0\lambda_1\lambda_2}(z)}(z) .
\end{align*}
Therefore we may define the cochain $\phi(\alpha)^{20}\in\check{C}^2(\pi(\U),{\mathcal S})$ by
\begin{equation}\label{phi20defn}
\phi(\alpha)^{20}_{\lambda_0\lambda_1\lambda_2}(z):=v_{\lambda_1\lambda_2}
(z)v_{\lambda_0\lambda_1}(z)v_{\lambda_0}^{-\check\partial s_{\lambda_0\lambda_1
\lambda_2}(z)}(z)^*v_{\lambda_0\lambda_2}(z)^*.
\end{equation}
It is then immediate that the triple $(\phi(\alpha)^{20},\phi(\alpha)^{11},\phi(\alpha)^{02})$ 
is a cochain in $C^2_{\check\partial s}(\pi(\U),{\mathcal S})$. Later
we will see that  it is $D_{\check\partial s}$-closed and hence gives our required map
\begin{equation}\label{Xi}\Xi_{\U,s}:(CT(X,\delta),\alpha)\mapsto[\phi(\alpha)^{20},
\phi(\alpha)^{11},\phi(\alpha)^{02}].\end{equation}
Presently though,  we state a lemma relating $\phi(\alpha)^{02}$ to the Mackey obstruction.
\begin{lemma}\label{phi02identity}
Let $f\in C(Z,M_n^u(\T))$ be the Mackey obstruction of $(CT(X,\delta),\alpha)$. 
Then $\phi(\alpha)^{02}_{\lambda_0}=f|_{\pi(U_{\lambda_0})}$\ \  and
$\quad\prod_{1\leq i<j\leq n}\phi(\alpha)^{02}_{\lambda_0}
(z)_{ij}^{m_il_j}=v_{\lambda_0}^l(\sigma_{\lambda_0}(z))
v_{\lambda_0}^m(\sigma_{\lambda_0}(z))v_{\lambda_0}^{m+l}(\sigma_{\lambda_0}(z))^*.$
\end{lemma}
\begin{proof}
Apply the definition of the Mackey obstruction and (\ref{mfequation2}) 
and (\ref{fullmackey}) in Section \ref{Mackeysectiononeonetwo}.
\end{proof}
Next we need to see how $(\phi(\alpha)^{20},\phi(\alpha)^{11},\phi(\alpha)^{02})$ 
compares with the Tu-\v{C}ech cocycle $\varphi(\alpha)$ that is, the image of 
$(CT(X,\delta),\alpha)$ under
 $\Upsilon:\mbox{Br}_{\R^n}(X)\to \check{H}^2(\R^n\ltimes X,{\mathcal S})$.
Recalling the construction of $\Upsilon$ 
from Theorem \ref{tusbigtheorem}, we know there is an open cover 
$\U^1=\{U_{\lambda_{01}}^1\}$ of $\R^n\ltimes X$ and continuous 
maps $u_{\lambda_0\lambda_1\lambda_{01}}:U^1_{\lambda_0\lambda_1
\lambda_{01}}\to U(\H)$ such that
$\beta^{\alpha,\Phi}|_{U^1_{\lambda_0\lambda_1\lambda_{01}}}=
\operatorname{Ad}u_{\lambda_0\lambda_1\lambda_{01}},$
and so $\Upsilon(CT(X,\delta),\alpha)\mapsto[\varphi(\alpha)]$ is given by Equation (\ref{cocycledefn}).

\begin{lemma}\label{cocyclerelationships}
Let $(X,\U,s)$ be in the standard setup, fix $(CT(X,\delta),\alpha)$ and local trivialisations 
$\{\Phi_{\lambda_0}\}$ defined over  $\U$. With these choices define
$\Upsilon:(CT(X,\delta),\alpha)\mapsto [\varphi(\alpha)]$.   Assume
$(CT(X,\delta),\alpha)$ generates $(\phi(\alpha)^{20},\phi(\alpha)^{11},\phi(\alpha)^{02})$
as preceding (\ref{Xi}). Then, for any open sets $U_{\lambda_{01}}^1, U_{\lambda_{00}}^1\in \U^1$ 
such that
$(-s_{\lambda_0\lambda_1}(z),\sigma_{\lambda_1}(z))\in U^1_{\lambda_0\lambda_1\lambda_{01}}$ 
and $(m,\sigma_{\lambda_0})\in U^1_{\lambda_0\lambda_0\lambda_{00}}$,
there are continuous maps
$\tau^{10}_{\lambda_0\lambda_1\lambda_{01}}: \{z\in 
\pi(U_{\lambda_0})\cap \pi(U_{\lambda_1}): (-s_{\lambda_0\lambda_1}(z),
\sigma_{\lambda_1}(z))\in U^1_{\lambda_{01}}\}\to \T$
and
$\tau^{01}_{\lambda_0\lambda_{00}}(m,\cdot): \{z\in \pi(U_{\lambda_0}): 
(m,\sigma_{\lambda_0}(z))\in U_{\lambda_{00}}^1\}\to \T$
such that for arbitrary indices $\lambda_{01}',\lambda_{02}'$ satisfying
$(-s_{\lambda_0\lambda_1}(z)+m,\sigma_{\lambda_1}(z))\in U^1_{\lambda_{01}'}$, 
$ (-s_{\lambda_0\lambda_1}(z)-s_{\lambda_1\lambda_2}(z),\sigma_{\lambda_2}(z))\in
 U^1_{\lambda_{02}'}$
the following three identities hold:
\begin{align*}
\prod_{1\leq i<j\leq n}\phi(\alpha)^{02}_{\lambda_0}(z)_{ij}^{m_il_j}=
&\tau^{01}_{\lambda_0\lambda_{12}}(l,z)\tau^{01}_{\lambda_0\lambda_{01}}(m,z)
\tau^{01}_{\lambda_0\lambda_{02}}(m+l,z)^*
\varphi(\alpha)_{\lambda_0\lambda_0\lambda_0\lambda_{01}\lambda_{02}\lambda_{12}}(m,l,
\sigma_{\lambda_0}(z)),\\
\phi(\alpha)^{11}_{\lambda_0\lambda_1}(m,z)=&\tau^{01}_{\lambda_1\lambda_{11}}(m,z)
\tau^{01}_{\lambda_0\lambda_{00}}(m,z)^*
\varphi(\alpha)_{\lambda_0\lambda_1\lambda_1\lambda_{01}\lambda_{01}'\lambda_{11}}
(-s_{\lambda_{0}\lambda_{1}}(z),m,\sigma_{\lambda_1}(z))\\
&\quad\times \varphi(\alpha)_{\lambda_0\lambda_0\lambda_1\lambda_{00}\lambda_{01}'
\lambda_{01}}(m,-s_{\lambda_{0}\lambda_{1}}(z),\sigma_{\lambda_1}(z))^*,
\end{align*}\begin{align*}
\phi(\alpha)^{20}_{\lambda_0\lambda_1\lambda_2}(z)=&
\tau^{10}_{\lambda_1\lambda_2\lambda_{12}}(z)
\tau^{10}_{\lambda_0\lambda_1\lambda_{01}}(z)
\tau^{01}_{\lambda_0\lambda_{00}}(-\check\partial 
s_{\lambda_0\lambda_1\lambda_2}(z),z)^*\tau^{10}_{\lambda_0\lambda_2\lambda_{02}}(z)^*\\
&\quad\times \varphi(\alpha)_{\lambda_0\lambda_1\lambda_2
\lambda_{01}\lambda_{02}'\lambda_{12}}(-s_{\lambda_0\lambda_1}(z),
-s_{\lambda_1\lambda_2}(z),\sigma_{\lambda_2}(z))\\
&\quad\times \varphi(\alpha)_{\lambda_0\lambda_0\lambda_2\lambda_{00}
\lambda_{02}'\lambda_{02}}(-\check\partial s_{\lambda_0\lambda_1\lambda_2}(z),
-s_{\lambda_0\lambda_2}(z),\sigma_{\lambda_2}(z))^*.
\end{align*}
\end{lemma}

\begin{proof}
By the definition of $\beta^{\alpha,\Phi}$ and the local lifts 
$\{u_{\lambda_0\lambda_1\lambda_{01}}\}$ we know
\begin{align*}
\operatorname{Ad}v_{\lambda_0\lambda_1}(z)= &\beta^{\alpha,\Phi}_{(\lambda_0
(-s_{\lambda_0\lambda_1}(z),\sigma_{\lambda_1}(z))\lambda_0)}
=\operatorname{Ad}u_{\lambda_0\lambda_1\lambda_{01}}(-s_{\lambda_0\lambda_1}(z),
\sigma_{\lambda_1}(z)).
\end{align*}
Therefore there exists a continuous function $\tau^{10}_{\lambda_0\lambda_1\lambda_{01}}$, 
as required, defined by
\[v_{\lambda_0\lambda_1}(z)=\tau^{10}_{\lambda_0\lambda_1\lambda_{01}}
(z)u_{\lambda_0\lambda_1\lambda_{01}}(-s_{\lambda_0\lambda_1}(z),\sigma_{\lambda_1}(z)).\]
Similarly, define $\tau^{01}_{\lambda_0\lambda_{00}}$ by the equation
$v_{\lambda_0}^m(z)=\tau^{01}_{\lambda_0\lambda_{00}}(m,z)
u_{\lambda_0\lambda_0\lambda_{00}}(m,\sigma_{\lambda_0}(z)).$
Then, using Lemma \ref{phi02identity}, we compute:
\begin{align*}
&\prod_{1\leq i<j\leq n}\phi(\alpha)^{02}_{\lambda_0}(z)_{ij}^{m_il_j}
=\prod_{1\leq i<j\leq n}f(z)_{ij}^{m_il_j}
=v_{\lambda_0}^l(\sigma_{\lambda_0}(z))v_{\lambda_0}^m
(\sigma_{\lambda_0}(z))v_{\lambda_0}^{m+l}(\sigma_{\lambda_0}(z))^*\\
&=\tau^{01}_{\lambda_0\lambda_{12}}(l,z)u_{\lambda_0\lambda_0\lambda_{12}}
(l,\sigma_{\lambda_0}(z))\tau^{01}_{\lambda_0\lambda_{01}}(m,z)
u_{\lambda_0\lambda_0\lambda_{01}}(m,\sigma_{\lambda_0}(z)) 
u_{\lambda_0\lambda_0\lambda_{02}}(m+l,\sigma_{\lambda_0}(z))^*\\ 
&\quad\quad\quad\quad\times\tau^{01}_{\lambda_0\lambda_{02}}(m+l,z)^*\\
&=\tau^{01}_{\lambda_0\lambda_{12}}(l,z)\tau^{01}_{\lambda_0\lambda_{01}}(m,z)
\tau^{01}_{\lambda_0\lambda_{02}}(m+l,z)^*\varphi(\alpha)_{\lambda_0
\lambda_0\lambda_0\lambda_{01}
\lambda_{02}\lambda_{12}}(m,l,\sigma_{\lambda_0}(z)).
\end{align*}
The second identity is a similar computation, but we need to choose an index 
$\lambda_{01}'$ such that $(-s_{\lambda_0\lambda_1}(z)+m,\sigma_{\lambda_1}(z))\in 
U^1_{\lambda_{01}'}$ and introduce the terms \[u_{\lambda_0\lambda_1\lambda_{01}'}
(m-s_{\lambda_0\lambda_1}(z),\sigma_{\lambda_1}(z))^*u_{\lambda_0\lambda_1\lambda_{01}'}
(m-s_{\lambda_0\lambda_1}(z),\sigma_{\lambda_1}(z)).\]
Doing so after the third equality sign below shows
\begin{align*}
&\phi(\alpha)^{11}_{\lambda_0\lambda_1}(m,z)
=v_{\lambda_1}^m(z)v_{\lambda_0\lambda_1}(z)v_{\lambda_0}^m(z)^*
v_{\lambda_0\lambda_1}(z)^*\quad\quad\quad\quad\quad\quad\quad\quad \\
&=\tau^{01}_{\lambda_1\lambda_{11}}(m,z)u_{\lambda_1\lambda_1\lambda_{11}}
(m,\sigma_{\lambda_1}(z))\tau^{10}_{\lambda_0\lambda_1\lambda_{01}}
(z)u_{\lambda_0\lambda_1\lambda_{01}}(-s_{\lambda_0\lambda_1}(z),\sigma_{\lambda_1}(z))\\
&\quad\times u_{\lambda_0\lambda_0\lambda_{00}}(m,\sigma_{\lambda_0}(z))^*
\tau^{01}_{\lambda_0\lambda_{00}}(m,z)^* u_{\lambda_0\lambda_1\lambda_{01}}
(-s_{\lambda_0\lambda_1}(z),\sigma_{\lambda_1}(z))^*
\tau^{10}_{\lambda_0\lambda_1\lambda_{01}}(z)^* \\
&=\tau^{01}_{\lambda_1\lambda_{11}}(m,z)\tau^{01}_{\lambda_0\lambda_{00}}(m,z)^*\\
\quad &\quad\times u_{\lambda_1\lambda_1\lambda_{11}}(m,\sigma_{\lambda_1}
(z))u_{\lambda_0\lambda_1\lambda_{01}}(-s_{\lambda_0\lambda_1}(z),\sigma_{\lambda_1}
(z))u_{\lambda_0\lambda_1\lambda_{01}'}(m-s_{\lambda_0\lambda_1}(z),\sigma_{\lambda_1}(z))^*\\
\quad &\quad\times u_{\lambda_0\lambda_1\lambda_{01}'}(m-s_{\lambda_0\lambda_1}(z),
\sigma_{\lambda_1}(z))u_{\lambda_0\lambda_0\lambda_{00}}(m,\sigma_{\lambda_0}
(z))^*u_{\lambda_0\lambda_1\lambda_{01}}(-s_{\lambda_0\lambda_1}(z),\sigma_{\lambda_1}(z))^*\\
&=\tau^{01}_{\lambda_1\lambda_{11}}(m,z)\tau^{01}_{\lambda_0\lambda_{00}}(m,z)^* 
\varphi(\alpha)_{\lambda_0\lambda_1\lambda_1\lambda_{00}\lambda_{01}'\lambda_{01}}
(-s_{\lambda_{0}\lambda_{1}}(z),m,\sigma_{\lambda_1}(z))\\
&\quad\times \varphi(\alpha)_{\lambda_0\lambda_0\lambda_1\lambda_{01}
\lambda_{01}'\lambda_{11}}(m,-s_{\lambda_{0}\lambda_{1}}(z),\sigma_{\lambda_1}(z))^*.
\end{align*}

Finally, for the last identity we choose an index $\lambda_{02}'$ such that 
$(-s_{\lambda_0\lambda_1}(z)-s_{\lambda_1\lambda_2}(z),\sigma_{\lambda_2}(z))\in 
U^1_{\lambda_{02}'}$, and introduce the term:
\[u_{\lambda_0\lambda_1\lambda_{02}'}(-s_{\lambda_0\lambda_1}(z)-
s_{\lambda_1\lambda_2}(z),\sigma_{\lambda_2}(z))^*
u_{\lambda_0\lambda_1\lambda_{02}'}(-s_{\lambda_0\lambda_2}(z)-\check\partial 
s_{\lambda_0\lambda_1\lambda_2}(z),\sigma_{\lambda_2}(z)).\]
Then we see that

\begin{align*}
&\phi(\alpha)^{20}_{\lambda_0\lambda_1\lambda_2}(z)
=v_{\lambda_1\lambda_2}(z)v_{\lambda_0\lambda_1}(z)v_{\lambda_0}^{-\check\partial 
s_{\lambda_0\lambda_1\lambda_2}(z)}(z)^*v_{\lambda_0\lambda_2}(z)^*
\quad\quad\quad\quad\quad\quad\quad\quad\quad\quad\quad\\
&=\tau^{10}_{\lambda_1\lambda_2\lambda_{12}}(z)u_{\lambda_1\lambda_2\lambda_{12}}
(-s_{\lambda_1\lambda_2}(z),\sigma_{\lambda_2}(z))\tau^{10}_{\lambda_0\lambda_1\lambda_{01}}
(z)u_{\lambda_0\lambda_1\lambda_{01}}(-s_{\lambda_0\lambda_1}(z),\sigma_{\lambda_1}(z))
\quad\quad\quad\\
&\quad\times u_{\lambda_0\lambda_0\lambda_{00}}(-\check\partial s_{\lambda_0
\lambda_1\lambda_2}(z),\sigma_{\lambda_0}(z))^*\tau^{01}_{\lambda_0
\lambda_{00}}(-\check\partial s_{\lambda_0\lambda_1\lambda_2}(z),z)^*\\
&\quad\times u_{\lambda_0\lambda_2\lambda_{02}}(-s_{\lambda_0\lambda_2}(z),
\sigma_{\lambda_2}(z))^*\tau^{10}_{\lambda_0\lambda_2\lambda_{02}}(z)^*\\
&=\tau^{10}_{\lambda_1\lambda_2\lambda_{12}}(z)\tau^{10}_{\lambda_0\lambda_1\lambda_{01}}(z)
\tau^{01}_{\lambda_0\lambda_{00}}(-\check\partial s_{\lambda_0\lambda_1\lambda_2}(z),z)^*
\tau^{10}_{\lambda_0\lambda_2\lambda_{02}}(z)^*\\
&\quad\times u_{\lambda_1\lambda_2\lambda_{12}}(-s_{\lambda_1\lambda_2}(z),
\sigma_{\lambda_2}(z))u_{\lambda_0\lambda_1\lambda_{01}}(-s_{\lambda_0\lambda_1}(z),
\sigma_{\lambda_1}(z))\\
&\quad\times u_{\lambda_0\lambda_1\lambda_{02}'}(-s_{\lambda_0\lambda_1}(z)-
s_{\lambda_1\lambda_2}(z),\sigma_{\lambda_2}(z))^*u_{\lambda_0\lambda_1\lambda_{02}'}(-
s_{\lambda_0\lambda_2}(z)-\check\partial s_{\lambda_0\lambda_1\lambda_2}(z),
\sigma_{\lambda_2}(z))\\
&\quad\times u_{\lambda_0\lambda_0\lambda_{00}}(-\check\partial s_{\lambda_0
\lambda_1\lambda_2}(z),\sigma_{\lambda_0}(z))^*u_{\lambda_0\lambda_2\lambda_{02}}(-
s_{\lambda_0\lambda_2}(z),\sigma_{\lambda_2}(z))^* \\
&=\tau^{10}_{\lambda_1\lambda_2\lambda_{12}}(z)\tau^{10}_{\lambda_0\lambda_1\lambda_{01}}(z)
\tau^{01}_{\lambda_0\lambda_{00}}(-\check\partial s_{\lambda_0\lambda_1\lambda_2}(z),z)^*
\tau^{10}_{\lambda_0\lambda_2\lambda_{02}}(z)^*\\
&\quad\times \varphi(\alpha)_{\lambda_0\lambda_1\lambda_2\lambda_{01}\lambda_{02}'
\lambda_{12}}(-s_{\lambda_0\lambda_1}(z),-s_{\lambda_1\lambda_2}(z),\sigma_{\lambda_2}(z))\\
&\quad\times \varphi(\alpha)_{\lambda_0\lambda_0\lambda_2\lambda_{00}\lambda_{02}'
\lambda_{02}}(-\check\partial s_{\lambda_0\lambda_1\lambda_2}(z),-s_{\lambda_0\lambda_2}(z),
\sigma_{\lambda_2}(z))^*.\quad\quad\quad\quad\quad\quad\quad\quad\quad\quad
\end{align*}
\end{proof}

\begin{lemma}\label{triplecocycle}
The triple $(\phi(\alpha)^{20},\phi(\alpha)^{11},\phi(\alpha)^{02})$ defines a cocycle 
in $Z^2_{\check\partial s}(\pi(\U),{\mathcal S})$.
\end{lemma}

\begin{proof}
First we check that $(\phi(\alpha)^{20},\phi(\alpha)^{11},\phi(\alpha)^{02})$ is a cochain. 
This is clear except for whether $\phi(\alpha)^{11}$ homomorphism from $\Z^n$ to $\T$. 
To see this apply Lemma \ref{phi02identity} to the definition of 
$\phi(\alpha)^{11}_{\lambda_0\lambda_1}$. Using that $\T$ is central we have the 
required relation:
\begin{align*}
\phi(\alpha)^{11}_{\lambda_0\lambda_1}(m+l,z):=&v_{\lambda_1}^{m+l}(z)
v_{\lambda_0\lambda_1}(z)v_{\lambda_0}^{m+l}(z)^*v_{\lambda_0\lambda_1}(z)^* \\
=&v_{\lambda_1}^l(z)v_{\lambda_1}^m(z)v_{\lambda_0\lambda_1}
(z)v_{\lambda_0}^m(z)^*v_{\lambda_0}^l(z)^*v_{\lambda_0\lambda_1}(z)^*\\
=&v_{\lambda_1}^l(z)v_{\lambda_1}^m(z)v_{\lambda_0\lambda_1}
(z)v_{\lambda_0}^m(z)^*v_{\lambda_0\lambda_1}(z)^*v_{\lambda_0\lambda_1}
(z)v_{\lambda_0}^l(z)^*v_{\lambda_0\lambda_1}(z)^* \\
=&\phi(\alpha)^{11}_{\lambda_0\lambda_1}(m,z)v_{\lambda_1}^l(z)v_{\lambda_0\lambda_1}
(z)v_{\lambda_0}^l(z)^*v_{\lambda_0\lambda_1}(z)^* \\
=&\phi(\alpha)^{11}_{\lambda_0\lambda_1}(m,z)\phi(\alpha)^{11}_{\lambda_0\lambda_1}(l,z).
\end{align*}

To check that $(\phi(\alpha)^{20},\phi(\alpha)^{11},\phi(\alpha)^{02})$ is a cocycle requires
first that $\check\partial \phi(\alpha)^{02}_{\lambda_0\lambda_1}(z)_{ij}=1$, 
which is immediate from  $\phi(\alpha)^{02}_{\lambda_0}=f|_{U_{\lambda_0}}$ 
(Lemma \ref{phi02identity}).  Then we need
\begin{align}
\label{last1}
\check\partial \phi(\alpha)^{11}_{\lambda_0\lambda_1\lambda_2}(m,z)^*=&\prod_{1\leq i<j\leq n}
\phi(\alpha)^{02}_{\lambda_0}(z)_{ij}^{\check\partial s_{\lambda_{0}\lambda_1
\lambda_{2}}(z)_im_j-m_i\check\partial s_{\lambda_{0}\lambda_1\lambda_{2}}(z)_j}, 
\quad \mbox{ and}
\end{align}\begin{align}
\label{last2}\check\partial \phi(\alpha)^{20}_{\lambda_0\lambda_1\lambda_2\lambda_3}(z)=&
\phi(\alpha)^{11}_{\lambda_0\lambda_1}(\check\partial s_{\lambda_1\lambda_2\lambda_3}(z),z)
\prod_{1\leq i<j\leq n}\phi(\alpha)^{02}(z)_{ij}^{C_{\lambda_{0}\lambda_1\lambda_{2}
\lambda_3}(z)_{ij}}.
\end{align}
For identity (\ref{last1}), we use
\begin{align*}
\beta^{\alpha,\Phi}_{(\lambda_0(-\check\partial s_{\lambda_{0}\lambda_1\lambda_{2}}(z),
\sigma_{\lambda_2}(z))\lambda_0)}
=(\beta^{\alpha,\Phi}_{(\lambda_0(-s_{\lambda_{0}\lambda_2}(z),\sigma_{\lambda_2}(z))
\lambda_2)})^{-1}\beta^{\alpha,\Phi}_{(\lambda_1(-s_{\lambda_{1}\lambda_2}(z),\sigma_{\lambda_2}(z))
\lambda_2)}\beta^{\alpha,\Phi}_{(\lambda_0(-s_{\lambda_{0}\lambda_1}(z),\sigma_{\lambda_1}(z))
\lambda_1)}
\end{align*}
implying that there is a continuous map $\gamma_{\lambda_{0}\lambda_1\lambda_{2}}: 
\pi(U_{\lambda_0})\cap \pi(U_{\lambda_1})\cap\pi(U_{\lambda_2})\to \T$ such that

\begin{align*}
v_{\lambda_0}^{-\check\partial s_{\lambda_{0}\lambda_1\lambda_{2}}(z)}(z)=
\gamma_{\lambda_0\lambda_1\lambda_2}(z)v_{\lambda_0\lambda_2}(z)^*v_{\lambda_1\lambda_2}
(z)v_{\lambda_0\lambda_1}(z).
\end{align*}

Using this fact, we deduce
\begin{align*}
&\prod_{1\leq i<j\leq n}\phi(\alpha)^{02}_{\lambda_0}(z)_{ij}^{\check\partial 
s_{\lambda_{0}\lambda_1\lambda_{2}}(z)_im_j-m_i\check\partial 
s_{\lambda_{0}\lambda_1\lambda_{2}}(z)_j}
=\left[\prod_{1\leq i<j\leq n}\phi(\alpha)^{02}_{\lambda_0}(z)_{ij}^{-\check\partial 
s_{\lambda_{0}\lambda_1\lambda_{2}}(z)_im_j+m_i\check\partial 
s_{\lambda_{0}\lambda_1\lambda_{2}}(z)_j}\right]^*\\
&\quad=\left[v_{\lambda_0}^m(z)v_{\lambda_0}^{-\check\partial s_{\lambda_{0}\lambda_1
\lambda_{2}}(z)}(z)v_{\lambda_0}^{-\check\partial s_{\lambda_{0}\lambda_1
\lambda_{2}}(z)+m}(z)^*\right.
\left. v_{\lambda_0}^{-\check\partial s_{\lambda_{0}\lambda_1\lambda_{2}}(z)+m}
(z)v_{\lambda_0}^m(z)^*v_{\lambda_0}^{-\check\partial s_{\lambda_{0}\lambda_1
\lambda_{2}}(z)}(z)^*\right]^* \\
&\quad=\left[v_{\lambda_0}^m(z)v_{\lambda_0}^{-\check\partial 
s_{\lambda_{0}\lambda_1\lambda_{2}}(z)}(z)v_{\lambda_2}^m(z)^*
v_{\lambda_0}^{-\check\partial s_{\lambda_{0}\lambda_1
\lambda_{2}}(z)}(z)^*\right]^*\quad\quad\ \quad\quad\\
&\quad=v_{\lambda_0\lambda_2}(z)^*v_{\lambda_1\lambda_2}(z)
v_{\lambda_0\lambda_1}(z)v_{\lambda_0}^m(z) v_{\lambda_0\lambda_1}(z)^*
v_{\lambda_1\lambda_2}(z)^*v_{\lambda_0\lambda_2}(z)
v_{\lambda_0}^m(z)^*.\quad\quad\ \quad\quad\quad\quad\ \quad\quad
\end{align*}
At this point, we introduce the terms $v_{\lambda_1}^m(z)^*v_{\lambda_1}^m(z)$ 
and $v_{\lambda_2}^m(z)^*v_{\lambda_2}^m(z)$ and use Lemma \ref{cyclicpermute}, 
so that the previous calculation now  produces  (cf. (\ref{last1})).
\begin{align*}
&v_{\lambda_0\lambda_2}(z)v_{\lambda_0}^m(z)^*v_{\lambda_0\lambda_2}(z)^*
v_{\lambda_1\lambda_2}(z)
[v_{\lambda_0\lambda_1}(z)v_{\lambda_0}^m(z)v_{\lambda_0\lambda_1}(z)^*v_{\lambda_1}^m(z)^*]
v_{\lambda_1}^m(z)v_{\lambda_1\lambda_2}(z)^*v_{\lambda_2}^m(z)^*v_{\lambda_2}^m(z)\\
&\quad= \phi(\alpha)^{11}_{\lambda_0\lambda_1}(m,z)^*
\phi(\alpha)^{11}_{\lambda_0\lambda_2}(m,z)\phi(\alpha)^{11}_{\lambda_1\lambda_2}(m,z)^*.
\end{align*}

As (\ref{last2}) involves noncommuting unitary operators we appeal to the relationship of 
$(\phi(\alpha)^{20},\phi(\alpha)^{11},\phi(\alpha)^{02})$ with the Tu-\v{C}ech cocycle 
$\varphi(\alpha)$ from Lemma \ref{cocyclerelationships}.
We use six identities that come directly from the Tu-\v{C}ech cocycle identity. 
To simplify the notation,  denote indices $\lambda_i$ indexing sets in the 
cover $\pi(\U)$ of $Z$ by just $i$. We will also write $F_{012},s_{01}$ and $z_0$ 
instead of $\check\partial s_{\lambda_0\lambda_1\lambda_2}(z)$, 
$s_{\lambda_0\lambda_1}(z)$ and $\sigma_{\lambda_0}(z)$ respectively. 
The six identities are as follows:
\begin{align}\label{lastidentity0}
&\varphi(\alpha)_{023\lambda_{12}\lambda_{13}'\lambda_{23}}(-s_{12},-
s_{23},z_3)\varphi(\alpha)_{003\lambda_{01}\lambda_{02}'\lambda_{12}}(-s_{01},-s_{12},z_2)^*\\
&=\varphi(\alpha)_{023\lambda_{02}'\lambda_{0123}\lambda_{23}}(-F_{012}-s_{02},-
s_{23},z_3)\varphi(\alpha)_{013\lambda_{01}\lambda_{0123}\lambda_{13}'}(-s_{01},-F_{123}-
s_{13},z_3)^*,\notag
\end{align}\begin{align}
\label{lastidentity1}
&\varphi(\alpha)_{023\lambda_{02}'\lambda_{0123}\lambda_{23}}(-F_{012}-s_{02},-
s_{23},z_3)\varphi(\alpha)_{023\lambda_{02}\lambda_{03}''\lambda_{23}}(-s_{02},-s_{23},z_3)^*\\
&=\varphi(\alpha)_{003\lambda_{012}\lambda_{0123}\lambda_{03}''}(-F_{012},-s_{02}-
s_{23},z_3)\varphi(\alpha)_{002\lambda_{012}\lambda_{02}'
\lambda_{02}}(-F_{012},-s_{02},z_2)^*,\notag
\end{align}\begin{align}
\label{lastidentity1a}
&\varphi(\alpha)_{013\lambda_{01}\lambda_{0123}\lambda_{13}'}(-s_{01},-F_{123}-s_{13},z_3)^*
= \varphi(\alpha)_{113\lambda_{123}\lambda_{13}'\lambda_{13}}(-F_{123},-s_{13},z_3)\\
&\quad\times\varphi(\alpha)_{013\lambda_{0123}'\lambda_{0123}\lambda_{13}}(-s_{01}-F_{123},-
s_{13},z_3)^*\times\varphi(\alpha)_{011\lambda_{01}\lambda_{0123}'\lambda_{123}'}(-s_{01},-
F_{123},z_1)^*\notag
\end{align}\begin{align}
\label{lastidentity2}
&\varphi(\alpha)_{013\lambda_{0123}'\lambda_{13}'\lambda_{13}}(-F_{123}-s_{01},-
s_{13},z_3)\varphi(\alpha)_{013\lambda_{01}\lambda_{03}'\lambda_{13}}(-s_{01},-s_{13},z_3)^*\\
&\quad=\varphi(\alpha)_{003\lambda_{123}\lambda_{0123}\lambda_{03}'}(-F_{123},-s_{01}-
s_{13},z_3)\varphi(\alpha)_{001\lambda_{123}\lambda_{0123}'\lambda_{01}}(-F_{123},-s_{01},z_1)^*,
\notag
\end{align}\begin{align}
\label{lastidentity3}
&\varphi(\alpha)_{003\lambda_{123}\lambda_{0123}\lambda_{03}'}(-F_{123},-F_{013}-s_{03},z_3)^*\\
&\quad=\varphi(\alpha)_{003\lambda_{0123}''\lambda_{0123}\lambda_{03}}(-F_{123}-F_{013},-
s_{03},z_3)^*\notag\\
&\quad\times\varphi(\alpha)_{000\lambda_{123}\lambda_{0123}''\lambda_{013}}(-F_{123},-
F_{013},z_0)^*\varphi(\alpha)_{003\lambda_{013}\lambda_{03}'\lambda_{03}}(-F_{013},-
s_{03},z_3),\notag
\end{align}\begin{align}
\label{lastidentity4}
&\varphi(\alpha)_{003\lambda_{0123}''\lambda_{0123}\lambda_{03}}(-F_{012}-F_{023},
-s_{03},z_3)^*=\varphi(\alpha)_{003\lambda_{023}\lambda_{03}''\lambda_{03}}(-F_{023},-s_{03},z_3)^*\\
&\quad\times\varphi(\alpha)_{003\lambda_{012}\lambda_{0123}\lambda_{03}''}(-F_{012},-F_{023}-
s_{03},z_3)^*
\varphi(\alpha)_{000\lambda_{012}\lambda_{01}\lambda_{023}}(-F_{012},-F_{023},z_0).\notag
\end{align}

We use these in the order they are presented to perform the computation below.
\begin{align*}
&\varphi(\alpha)_{123\lambda_{12}\lambda_{13}'\lambda_{23}}(-s_{12},-
s_{23},z_3)\varphi(\alpha)_{023\lambda_{02}\lambda_{03}''\lambda_{23}}(-s_{02},-s_{23},z_3)^*\\
&\quad\times\varphi(\alpha)_{013\lambda_{01}\lambda_{03}'\lambda_{13}}(-s_{01},-
s_{13},z_3)\varphi(\alpha)_{012\lambda_{01}\lambda_{02}'\lambda_{12}}(-s_{01},-s_{12},z_2)^*\\
=&\varphi(\alpha)_{023\lambda_{02}'\lambda_{0123}\lambda_{23'}}(-F_{012}-s_{02},-
s_{23},z_3)\varphi(\alpha)_{023\lambda_{02}\lambda_{03}''\lambda_{23}}(-s_{02},-s_{23},z_3)^*\\
&\quad\times\varphi(\alpha)_{013\lambda_{01}\lambda_{03}'\lambda_{13}}(-s_{01},-
s_{13},z_3)\varphi(\alpha)_{013\lambda_{01}\lambda_{0123}\lambda_{13}'}(-s_{01},-F_{123}-
s_{13},z_3)^*\quad \mbox{by } (\ref{lastidentity0})\\
=&\varphi(\alpha)_{003\lambda_{012}\lambda_{0123}\lambda_{03}''}(-F_{012},-s_{02}-
s_{23},z_3)\varphi(\alpha)_{002\lambda_{012}\lambda_{02}'\lambda_{02}}(-F_{012},-s_{02},z_2)^*\\
&\quad\times\varphi(\alpha)_{013\lambda_{01}\lambda_{03}'\lambda_{13}}(-s_{01},-
s_{13},z_3)\varphi(\alpha)_{013\lambda_{01}\lambda_{0123}\lambda_{13}'}(-s_{01},-F_{123}-
s_{13},z_3)^*\quad \mbox{by } (\ref{lastidentity1})\\
=&\varphi(\alpha)_{003\lambda_{012}\lambda_{0123}\lambda_{03}''}(-F_{012},-s_{02}-
s_{23},z_3)\varphi(\alpha)_{002\lambda_{012}\lambda_{02}'\lambda_{02}}(-F_{012},-s_{02},z_2)^*\\
&\quad\times\varphi(\alpha)_{013\lambda_{01}\lambda_{03}'\lambda_{13}}(-s_{01},-
s_{13},z_3)\varphi(\alpha)_{113\lambda_{123}\lambda_{13}'\lambda_{13}}(-F_{123},-s_{13},z_3)\\
&\quad\times\varphi(\alpha)_{013\lambda_{0123}'\lambda_{0123}\lambda_{13}}(-s_{01}-F_{123},-
s_{13},z_3)^*\varphi(\alpha)_{011\lambda_{01}\lambda_{0123}'\lambda_{123}'}(-s_{01},-
F_{123},z_1)^* \mbox{ by } (\ref{lastidentity1a})\\
=&\varphi(\alpha)_{003\lambda_{012}\lambda_{0123}\lambda_{03}''}(-F_{012},-s_{02}-
s_{23},z_3)\varphi(\alpha)_{002\lambda_{012}\lambda_{02}'\lambda_{02}}(-F_{012},-s_{02},z_2)^*\\
&\quad\times\varphi(\alpha)_{113\lambda_{123}\lambda_{13}'\lambda_{13}}(-F_{123},-
s_{13},z_3)\varphi(\alpha)_{011\lambda_{01}\lambda_{0123}'\lambda_{123}'}(-s_{01},-F_{123},z_1)^*\\
&\quad\times\varphi(\alpha)_{003\lambda_{123}\lambda_{0123}\lambda_{03}'}(-F_{123},-s_{01}-
s_{13},z_3)^*\varphi(\alpha)_{001\lambda_{123}\lambda_{0123}'\lambda_{01}}(-F_{123},-
s_{01},z_1)\quad \mbox{by } (\ref{lastidentity2})\\
=&\varphi(\alpha)_{003\lambda_{012}\lambda_{0123}\lambda_{03}''}(-F_{012},-F_{023}-
s_{03},z_3)\varphi(\alpha)_{002\lambda_{012}\lambda_{02}'\lambda_{02}}(-F_{012},-s_{02},z_2)^*\\
&\quad\times\varphi(\alpha)_{113\lambda_{123}\lambda_{13'}\lambda_{13}}(-F_{123},-
s_{13},z_3)\varphi(\alpha)_{011\lambda_{01}\lambda_{0123}'\lambda_{123}'}(-s_{01},-F_{123},z_1)^*\\
&\quad\times\varphi(\alpha)_{003\lambda_{123}\lambda_{0123}\lambda_{03}'}(-F_{123},-F_{013}-
s_{03},z_3)^*\varphi(\alpha)_{001\lambda_{123}\lambda_{0123}'\lambda_{01}}(-F_{123},-s_{01},z_1)
\ (\check\partial s=F)\\
=&\varphi(\alpha)_{003\lambda_{012}\lambda_{0123}\lambda_{03}''}(-F_{012},-F_{023}-
s_{03},z_3)\varphi(\alpha)_{002\lambda_{012}\lambda_{02}'\lambda_{02}}(-F_{012},-s_{02},z_2)^*\\
&\quad\times\varphi(\alpha)_{113\lambda_{123}\lambda_{13}'\lambda_{13}}(-F_{123},-
s_{13},z_3)\varphi(\alpha)_{011\lambda_{01}\lambda_{0123}'\lambda_{123}'}(-s_{01},-F_{123},z_1)^*\\
&\quad\times\varphi(\alpha)_{003\lambda_{013}\lambda_{03}'\lambda_{03}}(-F_{013},-
s_{03},z_3)\varphi(\alpha)_{003\lambda_{0123}''\lambda_{0123}\lambda_{03}}(-F_{123}-F_{013},-
s_{03},z_3)^*\\
&\quad\times\varphi(\alpha)_{000\lambda_{123}\lambda_{0123}''\lambda_{013}}(-F_{123},-
F_{013},z_0)^*\varphi(\alpha)_{001\lambda_{123}\lambda_{0123}'\lambda_{01}}(-F_{123},-
s_{01},z_1)\quad \mbox{by } (\ref{lastidentity3})\\
=&\varphi(\alpha)_{002\lambda_{012}\lambda_{02}'\lambda_{02}}(-F_{012},-s_{02},z_2)^*
\varphi(\alpha)_{113\lambda_{123}\lambda_{13}'\lambda_{13}}(-F_{123},-s_{13},z_3)\\
&\quad\times\varphi(\alpha)_{003\lambda_{013}\lambda_{03}'\lambda_{03}}(-F_{013},-
s_{03},z_3)\varphi(\alpha)_{003\lambda_{023}\lambda_{03}''\lambda_{03}}(-F_{023},-s_{03},z_3)^*\\
&\quad\times\varphi(\alpha)_{000\lambda_{123}\lambda_{0123}''\lambda_{013}}(-F_{123},-F_{013},z_0)^*
\varphi(\alpha)_{000\lambda_{012}\lambda_{0123}''\lambda_{023}}(-F_{012},-F_{023},z_0)\\
&\quad\times\varphi(\alpha)_{001\lambda_{123}\lambda_{0123}'\lambda_{01}}(-F_{123},-
s_{01},z_1)\varphi(\alpha)_{011\lambda_{01}\lambda_{0123}'\lambda_{123}'}(-s_{01},-F_{123},z_1)^*
\quad \mbox{by } (\ref{lastidentity4}).
\end{align*}
Summarizing the above computation we have
\begin{align*}
&\varphi(\alpha)_{123\lambda_{12}\lambda_{13}'\lambda_{23}}(-s_{12},-
s_{23},z_3)\varphi(\alpha)_{113\lambda_{123}\lambda_{13}'\lambda_{13}}(-F_{123},-s_{13},z_3)^*\\
&\quad\times\varphi(\alpha)_{023\lambda_{02}\lambda_{03}''\lambda_{23}}(-s_{02},-s_{23},z_3)^*
\varphi(\alpha)_{003\lambda_{023}\lambda_{03}''\lambda_{03}}(-F_{023},-s_{03},z_3)\\
&\quad\times\varphi(\alpha)_{013\lambda_{01}\lambda_{03}'\lambda_{13}}(-s_{01},-
s_{13},z_3)\varphi(\alpha)_{003\lambda_{013}\lambda_{03}'\lambda_{03}}(-F_{013},-s_{03},z_3)^*\\
&\quad\times\varphi(\alpha)_{012\lambda_{01}\lambda_{02}'\lambda_{12}}(-s_{01},-s_{12},z_2)^*
\varphi(\alpha)_{002\lambda_{012}\lambda_{02}'\lambda_{02}}(-F_{012},-s_{02},z_2)\\
\quad\quad\quad=&\varphi(\alpha)_{000\lambda_{123}\lambda_{0123}''\lambda_{013}}(-F_{123},-
F_{013},z_0)^*\varphi(\alpha)_{000\lambda_{012}\lambda_{0123}''\lambda_{023}}
(-F_{012},-F_{023},z_0)\\
&\quad\times\varphi(\alpha)_{011\lambda_{01}\lambda_{0123}'\lambda_{123}'}(-s_{01},-
F_{123},z_1)^*\varphi(\alpha)_{001\lambda_{123}\lambda_{0123}'\lambda_{01}}(-F_{123},-s_{01},z_1).
\end{align*}
To see what this has to do with identity (\ref{last2}), we invoke Lemma \ref{cocyclerelationships}. 
Using those identities in the above equation, and canceling out the $\tau$ terms gives

\begin{align*}
\check\partial\phi(\alpha)^{20}_{\lambda_0\lambda_1\lambda_2\lambda_3}&
=\phi(\alpha)^{11}_{\lambda_0\lambda_1}(F_{\lambda_1\lambda_2\lambda_3}(z),z)
\prod_{1\leq i<j\leq n}\phi(\alpha)^{02}_{\lambda_0}(z)_{ij}^{F_{\lambda_0\lambda_1\lambda_2}
(z)_iF_{\lambda_0\lambda_2\lambda_3}(z)_j-F_{\lambda_1\lambda_2\lambda_3}
(z)_iF_{\lambda_0\lambda_1\lambda_3}(z)_j}\\
&=\phi(\alpha)^{11}_{\lambda_0\lambda_1}(F_{\lambda_1\lambda_2\lambda_3}(z),z)
\prod_{1\leq i<j\leq n}\phi(\alpha)^{02}_{\lambda_0}
(z)_{ij}^{C(F)_{\lambda_0\lambda_1\lambda_2\lambda_3}(z)_{ij}},
\end{align*}
which is what we wanted to show.
\end{proof}

\begin{lemma}
Let $(X,\U,s)$ be in the standard setup. Then the map $\Xi_{\U,s}$
 in (\ref{Xi}) 
is well-defined  and constant on outer conjugacy classes, and therefore defines a 
map (denoted by the same symbol) $\Xi_{\U,s}:\operatorname{Br}_{\R^n}(X)\to 
\mathbb{H}^2_{\check\partial s}(\pi(\U),\mathcal{S})$.
\end{lemma}
\begin{proof}
In constructing $\Xi_{\U,s}$
we had to choose local trivialisations $\{\Phi_\lambda\}$ of $CT(X,\delta)$, and the unitary lifts 
$\{v_{\lambda_0}\}$ and $\{v_{\lambda_0\lambda_1}\}$ of $\beta^{\alpha,\Phi}_{(\lambda_0
(m,\sigma_{\lambda_0}(z))\lambda_0)}$ and $\beta^{\alpha,\Phi}_{(\lambda_0
(-s_{\lambda_0\lambda_1}(z),\sigma_{\lambda_1}(z))\lambda_1)}$ respectively. 
Let us first assume that $\Xi_{\U,s}$ is independent of these choices and prove
it descends to $\operatorname{Br}_{\R^n}(X)$. 

Let $(CT(X,\delta),\chi)$ be outer conjugate to $(CT(X,\delta),\alpha)$. 
Choose local trivialisations $\Psi_{\lambda_0}:CT(X,\delta)|_{U_{\lambda_0}}$ 
$\to C_0(U_{\lambda_0},\K)$ (since $\U$ is good), so that there exists 
$\tilde{v}^m_{\lambda_0}$ and $\tilde{v}_{\lambda_0\lambda_1}$ such that
$\beta^{\chi,\Psi}_{\lambda_0(m,\sigma_{\lambda_0}(z))\lambda_0}=\operatorname{Ad}
\tilde{v}^m_{\lambda_0}(z)$ and
$\beta^{\chi,\Psi}_{\lambda_0(-s_{\lambda_0\lambda_1}(z),\sigma_{\lambda_1}(z))\lambda_1}=
\operatorname{Ad}\tilde{v}_{\lambda_0\lambda_1}(z).$
Then $\Xi_{\U,s}((CT(X,\delta),\chi))=[\phi(\chi)^{20},\phi(\chi)^{11},\phi(\chi)^{02}]$ is defined by
\begin{align*}
\phi(\chi)^{02}_{\lambda_0}(z)_{ij}:=&\tilde{v}_{\lambda_0}^j(\sigma_{\lambda_0}(z))
\tilde{v}_{\lambda_0}^i(\sigma_{\lambda_0}(z))\tilde{v}_{\lambda_0}^{e_i+e_j}
(\sigma_{\lambda_0}(z))^*,\\
\phi(\chi)^{11}(m,z):=&\tilde{v}_{\lambda_1}^m(z)\tilde{v}_{\lambda_0\lambda_1}(z)
\tilde{v}_{\lambda_0}^m(z)^*\tilde{v}_{\lambda_0\lambda_1}(z)^*,\\
\phi(\chi)^{20}(z):=&\tilde{v}_{\lambda_1\lambda_2}(z)\tilde{v}_{\lambda_0\lambda_1}(z)
\tilde{v}_{\lambda_0}^{-\check\partial s_{\lambda_0\lambda_1\lambda_2}(z)}(z)^*
\tilde{v}_{\lambda_0\lambda_2}(z)^*.
\end{align*}

Observe that the definitions of $\beta^{\alpha,\Phi}$, $\{v_{\lambda_0}\}$ and 
$\{v_{\lambda_0\lambda_1}\}$ in (\ref{betav01}) and (\ref{betav10}) imply for any 
$a\in CT(X,\delta)$ and any indices $\lambda_0,\lambda_1$ that
\begin{align}
\label{phialpham}(\Phi_{\lambda_0}\circ\alpha_{m}(a))(\sigma_{\lambda_0}(z))=&
v_{\lambda_0}^m(z)\Phi_{\lambda_0}(a)(\sigma_{\lambda_0}(z))v_{\lambda_0}^m(z)^*, 
\quad  m\in \Z^n,\\
\label{phialphas}(\Phi_{\lambda_1}\circ\alpha_{-{s_{\lambda_0\lambda_1}(z)}}(a))
(\sigma_{\lambda_1}(z))=&v_{\lambda_0\lambda_1}(z)\Phi_{\lambda_0}(a)
(\sigma_{\lambda_0}(z))v_{\lambda_0\lambda_1}(z)^*.
\end{align}
As in the proof of Proposition \ref{mymaptotucech}, we find there exists 
$\nu_{\lambda_0}:U_{\lambda_0}\to U(\H)$ such that $\Psi_{\lambda_0}\circ \Phi\circ 
\Phi_{\lambda_0}^{-1}=\operatorname{Ad} \nu_{\lambda_0}$, and
\begin{align*}
\beta^{\chi,\Psi}_{(\lambda_0(m,\sigma_{\lambda_0}(z))\lambda_0)}(T)=&
\operatorname{Ad}[ \nu_{\lambda_0}(\sigma_{\lambda_0}(z))
\overline{\Phi_{\lambda_0}}(w_{m})(\sigma_{\lambda_0}(z))
v^m_{\lambda_0}(z)\nu_{\lambda_0}(\sigma_{\lambda_0}(z))^*](T),\\
\beta^{\chi,\Psi}_{(\lambda_0(-s_{\lambda_0\lambda_1}(z),\sigma_{\lambda_1}(z))\lambda_1)}(T)=&
\operatorname{Ad} [\nu_{\lambda_1}(\sigma_{\lambda_1}(z))\overline{\Phi_{\lambda_1}}
(w_{-s_{\lambda_0\lambda_1}(z)})(\sigma_{\lambda_1}(z))v_{\lambda_0\lambda_1}(z)
\nu_{\lambda_0}(\sigma_{\lambda_0}(z))^*](T).
\end{align*}
Therefore there exist continuous functions $\tau^{01}_{\lambda_0}(m,\cdot):
\pi(U_{\lambda_0})\to \T$ and $\tau_{\lambda_0\lambda_1}^{10}:\pi(U_{\lambda_0})\cap 
\pi(U_{\lambda_0})\to \T$ such that
\begin{align*}
\tilde{v}^m_{\lambda_0}(z)=&\tau^{01}_{\lambda_0}(m,z)\nu_{\lambda_0}(\sigma_{\lambda_0}(z))
\overline{\Phi_{\lambda_0}}(w_{m})(\sigma_{\lambda_0}(z))v^m_{\lambda_0}(z)\nu_{\lambda_0}
(\sigma_{\lambda_0}(z))^*,\\
\tilde{v}_{\lambda_0\lambda_1}(z)=&\tau_{\lambda_0\lambda_1}^{10}(z)\nu_{\lambda_1}
(\sigma_{\lambda_1}(z))\overline{\Phi_{\lambda_1}}(w_{-s_{\lambda_0\lambda_1}(z)})
(\sigma_{\lambda_1}(z))v_{\lambda_0\lambda_1}(z)\nu_{\lambda_0}(\sigma_{\lambda_0}(z))^*.
\end{align*}
Now we recalculate $(\phi(\chi)^{20},\phi(\chi)^{11},\phi(\chi)^{02})$ using these equations, 
with the purpose being to show $(\tau^{10},\tau^{01})$ is in fact a 1-cochain in 
$C^1_{\check\partial s}(\pi(\U),{\mathcal S})$ such that

\begin{equation}
(\phi(\chi)^{20},\phi(\chi)^{11},\phi(\chi)^{02})=(\phi(\alpha)^{20},\phi(\alpha)^{11},
\phi(\alpha)^{02})D_{\check\partial s}(\tau^{10},\tau^{01}).
\end{equation}
First we note that Lemma \ref{phi02identity} says $\phi(\chi)^{20}$ is \emph{exactly} the 
Mackey obstruction of $(CT(X,\delta),\chi)$, which we recall is constant on outer conjugacy 
classes. Therefore
\begin{equation}\label{diff02}
\phi(\chi)^{20}=\phi(\alpha)^{20}.
\end{equation}
This suffices to show that $m\mapsto\tau^m_{\lambda_0}(z)$ is a homomorphism from 
$\Z^n$ to $\T$, which will imply $(\tau^{10},\tau^{01})\in C^1_{\check\partial s}(\pi(\U),{\mathcal S})$.
\begin{align*}
&\prod_{1\leq i<j\leq n} \phi(\chi)^{02}(z)_{ij}^{m_il_j}
:=\tilde{v}^l_{\lambda_0}(z)\tilde{v}^m_{\lambda_0}(z)\tilde{v}^{m+l}_{\lambda_0}(z)^* \quad 
\mbox{by Lemma } \ref{phi02identity}\quad\quad\quad\quad\quad \\
=&\tau^{01}_{\lambda_0}(l,z)\nu_{\lambda_0}(\sigma_{\lambda_0}(z))
\overline{\Phi_{\lambda_0}}(w_{l})(\sigma_{\lambda_0}(z))v^l_{\lambda_0}(z)\nu_{\lambda_0}
(\sigma_{\lambda_0}(z))^*\\
&\quad\times \tau^{01}_{\lambda_0}(m,z)\nu_{\lambda_0}(\sigma_{\lambda_0}(z))
\overline{\Phi_{\lambda_0}}(w_{m})(\sigma_{\lambda_0}(z))v^m_{\lambda_0}(z)\nu_{\lambda_0}
(\sigma_{\lambda_0}(z))^*\\
&\quad\times\nu_{\lambda_0}(\sigma_{\lambda_0}(z))v^{m+l}_{\lambda_0}(z)^*
\overline{\Phi_{\lambda_0}}(w_{m+l})(\sigma_{\lambda_0}(z))^*\nu_{\lambda_0}
(\sigma_{\lambda_0}(z))^*\tau^{01}_{\lambda_0}(m+l,z)^*\quad\\
=&\tau^{01}_{\lambda_0}(l,z)\tau^{01}_{\lambda_0}(m,z)\tau^{01}_{\lambda_0}(m+l,z)^*\ 
\overline{\Phi_{\lambda_0}}(w_{m+l})(\sigma_{\lambda_0}(z))^*\overline{\Phi_{\lambda_0}}(w_{l})
(\sigma_{\lambda_0}(z))v^l_{\lambda_0}(z)\\
&\quad\times \overline{\Phi_{\lambda_0}}(w_{m})(\sigma_{\lambda_0}(z))v^m_{\lambda_0}(z)
v^{m+l}_{\lambda_0}(z)^* \quad \mbox{by Lemma } \ref{cyclicpermute} \\
=&\tau^{01}_{\lambda_0}(l,z)\tau^{01}_{\lambda_0}(m,z)\tau^{01}_{\lambda_0}(m+l,z)^*
 \overline{\Phi_{\lambda_0}}(w_{m+l})(\sigma_{\lambda_0}(z))^*
 \overline{\Phi_{\lambda_0}}(w_{l})(\sigma_{\lambda_0}(z))\\
&\quad\times v^l_{\lambda_0}(z)\overline{\Phi_{\lambda_0}}(w_{m})
(\sigma_{\lambda_0}(z))v^l_{\lambda_0}(z)^* \prod_{1\leq i<j\leq n} \phi^{02}(z)_{ij}^{m_il_j} 
\quad \mbox{by Lemma } \ref{phi02identity}\\
=&\tau^{01}_{\lambda_0}(l,z)\tau^{01}_{\lambda_0}(m,z)\tau^{01}_{\lambda_0}(m+l,z)^*
\prod_{1\leq i<j\leq n} \phi^{02}(z)_{ij}^{m_il_j} \quad \mbox{by }(\ref{phialpham}) \mbox{ and } 
(\ref{extcondition2ch3}).
\end{align*}
As $\phi(\chi)^{02}=\phi(\alpha)^{02}$, this implies that 
$\tau^{01}_{\lambda_0}(l,z)\tau^{01}_{\lambda_0}(m,z)\tau^{01}_{\lambda_0}(m+l,z)^*=1$.
Therefore  $(\tau^{10},\tau^{01})\in C^1_{\check\partial s}(\pi(\U),{\mathcal S})$.
Now, for $\phi(\chi)^{11}_{\lambda_0\lambda_1}(m,z):=\tilde{v}_{\lambda_1}^m(z)
\tilde{v}_{\lambda_0\lambda_1}(z)\tilde{v}_{\lambda_0}^m(z)^*\tilde{v}_{\lambda_0\lambda_1}(z)^*.
$
\begin{align*}
&\phi(\chi)^{11}_{\lambda_0\lambda_1}(m,z)=\tau^{01}_{\lambda_1}(m,z)\nu_{\lambda_1}
(\sigma_{\lambda_1}(z))\overline{\Phi_{\lambda_1}}(w_{m})(\sigma_{\lambda_1}(z))
v^m_{\lambda_1}(z)\nu_{\lambda_1}(\sigma_{\lambda_1}(z))^*\notag\\
&\quad\times\tau^{10}_{\lambda_0\lambda_1}(z)\nu_{\lambda_1}(\sigma_{\lambda_1}(z))
\overline{\Phi_{\lambda_1}}(w_{-s_{\lambda_0\lambda_1}(z)})(\sigma_{\lambda_1}
(z))v_{\lambda_0\lambda_1}(z)\nu_{\lambda_0}(\sigma_{\lambda_0}(z))^*\notag\\
&\quad\times\nu_{\lambda_0}(\sigma_{\lambda_0}(z))v^m_{\lambda_0}(z)^*
\overline{\Phi_{\lambda_0}}(w_{m})(\sigma_{\lambda_0}(z))^*
\nu_{\lambda_0}(\sigma_{\lambda_0}(z))^*\tau^{01}_{\lambda_0}(m,z)^*\notag\\
&\quad\times\nu_{\lambda_0}(\sigma_{\lambda_0}(z))v_{\lambda_0\lambda_1}(z)^*
\overline{\Phi_{\lambda_1}}(w_{-s_{\lambda_0\lambda_1}(z)})(\sigma_{\lambda_1}(z))^*
\nu_{\lambda_1}(\sigma_{\lambda_1}(z))^*\tau^{10}_{\lambda_0\lambda_1}(z)^*\\
&=\tau^{01}_{\lambda_1}(m,z)\tau^{01}_{\lambda_0}(m,z)^*\overline{\Phi_{\lambda_1}}(w_{m})
(\sigma_{\lambda_1}(z))v^m_{\lambda_1}(z)\overline{\Phi_{\lambda_1}}
(w_{-s_{\lambda_0\lambda_1}(z)})(\sigma_{\lambda_1}(z))\notag\\
&\quad\times v_{\lambda_0\lambda_1}(z)v^m_{\lambda_0}(z)^*
\overline{\Phi_{\lambda_0}}(w_{m})(\sigma_{\lambda_0}(z))^* 
v_{\lambda_0\lambda_1}(z)^*\overline{\Phi_{\lambda_1}}(w_{-s_{\lambda_0\lambda_1}(z)})
(\sigma_{\lambda_1}(z))^*\quad \mbox{by Lemma } \ref{cyclicpermute}\notag\\
&=\phi^{11}_{\lambda_0\lambda_1}(m,z)\tau^{01}_{\lambda_1}(m,z)\tau^{01}_{\lambda_0}(m,z)^*
\overline{\Phi_{\lambda_1}}(w_{m})(\sigma_{\lambda_1}(z))v^m_{\lambda_1}(z)
\overline{\Phi_{\lambda_1}}(w_{-s_{\lambda_0\lambda_1}(z)})(\sigma_{\lambda_1}(z))\notag\\
&\quad\times v^m_{\lambda_1}(z)^*v_{\lambda_0\lambda_1}(z)\overline{\Phi_{\lambda_0}}(w_{m})
(\sigma_{\lambda_0}(z))^*v_{\lambda_0\lambda_1}(z)^*\overline{\Phi_{\lambda_1}}(w_{-
s_{\lambda_0\lambda_1}(z)})(\sigma_{\lambda_1}(z))^*\quad \mbox{by }(\ref{phi11defn})
\quad\quad
\notag\\
&=\phi(\alpha)^{11}_{\lambda_0\lambda_1}(m,z)\tau^{01}_{\lambda_1}
(m,z)\tau^{01}_{\lambda_0}(m,z)^*\overline{\Phi_{\lambda_1}}(w_{m})(\sigma_{\lambda_1}(z))
\overline{\Phi_{\lambda_1}}(\alpha_m(w_{-s_{\lambda_0\lambda_1}(z)}))(\sigma_{\lambda_1}(z))\notag\\
&\quad\times\overline{\Phi_{\lambda_1}}(\alpha_{-s_{\lambda_0\lambda_1}(z)}(w_{m}))
(\sigma_{\lambda_1}(z))^*\overline{\Phi_{\lambda_1}}(w_{-s_{\lambda_0\lambda_1}(z)})
(\sigma_{\lambda_1}(z))^*\quad \mbox{by }(\ref{phialpham}) \mbox{ and } (\ref{phialphas})
\quad\quad\quad\quad\notag\\
&=\phi(\alpha)^{11}_{\lambda_0\lambda_1}(m,z)\tau^{01}_{\lambda_1}(m,z)
\tau^{01}_{\lambda_0}(m,z)^*
\overline{\Phi_{\lambda_1}}(w_{m}\alpha_m(w_{-s_{\lambda_0\lambda_1}(z)})
\alpha_{-s_{\lambda_0\lambda_1}(z)}(w_{m})^*
w_{-s_{\lambda_0\lambda_1}(z)})(\sigma_{\lambda_1}(z))\notag
\end{align*}\begin{align*}
&=\phi(\alpha)^{11}_{\lambda_0\lambda_1}(m,z)\tau^{01}_{\lambda_1}(m,z)
\tau^{01}_{\lambda_0}(m,z)^*\quad \mbox{by }(\ref{extcondition2ch3}).
\end{align*}

That is, we have shown
\begin{equation}\label{diff11}
\phi(\chi)^{11}_{\lambda_0\lambda_1}(m,z)=\phi(\alpha)^{11}_{\lambda_0\lambda_1}
(m,z)\tau^{01}_{\lambda_1}(m,z)\tau^{01}_{\lambda_0}(m,z)^*.
\end{equation}
Finally we work on the last component
\begin{align*}
\phi(\chi)&^{20}_{\lambda_0\lambda_1\lambda_2}(z)
:=\tilde{v}_{\lambda_1\lambda_2}(z)\tilde{v}_{\lambda_0\lambda_1}(z)
\tilde{v}_{\lambda_0}^{-\check\partial s_{\lambda_0\lambda_1\lambda_2}(z)}(z)^*
\tilde{v}_{\lambda_0\lambda_2}(z)^*\quad\quad\quad\quad\\
=\tau_{\lambda_1\lambda_2}^{10}&(z)\nu_{\lambda_2}(\sigma_{\lambda_2}(z))
\overline{\Phi_{\lambda_2}}(w_{-s_{\lambda_1\lambda_2}(z)})(\sigma_{\lambda_2}
(z))v_{\lambda_1\lambda_2}(z)\nu_{\lambda_1}(\sigma_{\lambda_1}(z))^*
\tau_{\lambda_0\lambda_1}^{10}(z)\nu_{\lambda_1}(\sigma_{\lambda_1}(z))\\\times&
\overline{\Phi_{\lambda_1}}(w_{-s_{\lambda_0\lambda_1}(z)})(\sigma_{\lambda_1}
(z))v_{\lambda_0\lambda_1}(z)\nu_{\lambda_0}(\sigma_{\lambda_0}(z))^*
\nu_{\lambda_0}(\sigma_{\lambda_0}(z))v^{-\check\partial s_{\lambda_0\lambda_1\lambda_2}
(z)}_{\lambda_0}(z)^*\\\times&\overline{\Phi_{\lambda_0}}(w_{- \check\partial 
s_{\lambda_0\lambda_1\lambda_2}(z)})(\sigma_{\lambda_0}(z))^*\nu_{\lambda_0}(\sigma_{\lambda_0}
(z))^*\tau^{01}_{\lambda_0}(-\check\partial s_{\lambda_0\lambda_1\lambda_2}(z),z)^*\nu_{\lambda_0}
(\sigma_{\lambda_0}(z))v_{\lambda_0\lambda_2}(z)^*\\
\times&\overline{\Phi_{\lambda_2}}(w_{-s_{\lambda_0\lambda_2}(z)})(\sigma_{\lambda_2}(z))^*
 \nu_{\lambda_2}(\sigma_{\lambda_2}(z))^*\tau_{\lambda_0\lambda_2}^{10}(z)^*  \\
=\tau_{\lambda_1\lambda_2}^{10}&(z)\tau_{\lambda_0\lambda_1}^{10}(z)\tau^{01}_{\lambda_0}
(-\check\partial s_{\lambda_0\lambda_1\lambda_2}(z),z)^*\tau_{\lambda_0\lambda_2}^{10}(z)^*
\overline{\Phi_{\lambda_2}}(w_{-s_{\lambda_0\lambda_2}(z)})(\sigma_{\lambda_2}(z))^*\\\times&
\overline{\Phi_{\lambda_2}}(w_{-s_{\lambda_1\lambda_2}(z)})(\sigma_{\lambda_2}
(z))v_{\lambda_1\lambda_2}(z)\overline{\Phi_{\lambda_1}}(w_{-s_{\lambda_0\lambda_1}(z)})
(\sigma_{\lambda_1}(z))v_{\lambda_0\lambda_1}(z)v^{-\check\partial 
s_{\lambda_0\lambda_1\lambda_2}(z)}_{\lambda_0}(z)^*\\
\times&\overline{\Phi_{\lambda_0}}(w_{-\check\partial s_{\lambda_0\lambda_1\lambda_2}(z)})
(\sigma_{\lambda_0}(z))^*v_{\lambda_0\lambda_2}(z)^* \quad \mbox{by Lemma } \ref{cyclicpermute}\\
=\phi(\alpha)&^{20}_{\lambda_0\lambda_1\lambda_2}(z)\tau_{\lambda_1\lambda_2}^{10}(z)
\tau_{\lambda_0\lambda_1}^{10}(z)\tau^{01}_{\lambda_0}(-\check\partial 
s_{\lambda_0\lambda_1\lambda_2}(z),z)^*\tau_{\lambda_0\lambda_2}^{10}(z)^*\\
\times&\overline{\Phi_{\lambda_2}}(w_{-s_{\lambda_0\lambda_2}(z)})(\sigma_{\lambda_2}(z))^*
\overline{\Phi_{\lambda_2}}(w_{-s_{\lambda_1\lambda_2}(z)})(\sigma_{\lambda_2}(z)) 
v_{\lambda_1\lambda_2}(z)\\
\times&\overline{\Phi_{\lambda_1}}(w_{-s_{\lambda_0\lambda_1}(z)})(\sigma_{\lambda_1}
(z))v_{\lambda_1\lambda_2}(z)^* v_{\lambda_0\lambda_2}(z)\overline{\Phi_{\lambda_0}}(w_{-\check
\partial s}^*)(\sigma_{\lambda_0}(z))v_{\lambda_0\lambda_2}(z)^* \quad \mbox{by }(\ref{phi20defn})\\
=\phi(\alpha)&^{20}_{\lambda_0\lambda_1\lambda_2}(z)\tau_{\lambda_1\lambda_2}^{10}(z)
\tau_{\lambda_0\lambda_1}^{10}(z)\tau^{01}_{\lambda_0}(-\check\partial 
s_{\lambda_0\lambda_1\lambda_2}(z),z)^*\tau_{\lambda_0\lambda_2}^{10}(z)^*
\overline{\Phi_{\lambda_2}}(w_{-s_{\lambda_0\lambda_2}(z)})(\sigma_{\lambda_2}(z))^*\\\times&
\overline{\Phi_{\lambda_2}}(w_{-s_{\lambda_1\lambda_2}(z)})(\sigma_{\lambda_2}(z))
\overline{\Phi_{\lambda_2}}(\alpha_{-s_{\lambda_1\lambda_2}(z)}(w_{-s_{\lambda_0\lambda_1}(z)}))
(\sigma_{\lambda_2}(z))\\
\times&\overline{\Phi_{\lambda_2}}(\alpha_{-s_{\lambda_0\lambda_2}(z)}(w_{-\check\partial 
s_{\lambda_0\lambda_1\lambda_2}(z)}^*))(\sigma_{\lambda_2}(z))^* \quad \mbox{by }(\ref{phialphas})
\times 2  \\
=\phi(\alpha)&^{20}_{\lambda_0\lambda_1\lambda_2}(z)\tau_{\lambda_1\lambda_2}^{10}(z)
\tau_{\lambda_0\lambda_1}^{10}(z)\tau^{01}_{\lambda_0}(-\check\partial 
s_{\lambda_0\lambda_1\lambda_2}(z),z)^*\tau_{\lambda_0\lambda_2}^{10}(z)^*\\
\times&\overline{\Phi_{\lambda_2}}(w_{-s_{\lambda_0\lambda_2}(z)}^*
w_{-s_{\lambda_1\lambda_2}(z)})(\sigma_{\lambda_2}(z))\\
\times&\overline{\Phi_{\lambda_2}}(\alpha_{-s_{\lambda_1\lambda_2}(z)}
(w_{-s_{\lambda_0\lambda_1}(z)})\alpha_{-s_{\lambda_0\lambda_2}(z)}(w_{-\check\partial 
s_{\lambda_0\lambda_1\lambda_2}(z)}))(\sigma_{\lambda_2}(z))  \\
=\phi(\alpha)&^{20}_{\lambda_0\lambda_1\lambda_2}(z)\tau_{\lambda_1\lambda_2}^{10}(z)
\tau_{\lambda_0\lambda_1}^{10}(z)\tau^{01}_{\lambda_0}(-\check\partial 
s_{\lambda_0\lambda_1\lambda_2}(z),z)^*\tau_{\lambda_0\lambda_2}^{10}(z)^*\\
\times&\overline{\Phi_{\lambda_2}}(w_{-s_{\lambda_0\lambda_2}(z)}^*
w_{-s_{\lambda_1\lambda_2}(z)})(\sigma_{\lambda_2}(z))\\
\times&\overline{\Phi_{\lambda_2}}(w_{-s_{\lambda_1\lambda_2}(z)}^*
w_{-s_{\lambda_1\lambda_2}(z)-s_{\lambda_0\lambda_1}(z)}
w_{-s_{\lambda_1\lambda_2}(z)-s_{\lambda_0\lambda_1}(z)}^*
w_{-s_{\lambda_0\lambda_2}(z)})\mbox{ by }(\ref{extcondition2ch3})\\
=\phi(\alpha)&^{20}_{\lambda_0\lambda_1\lambda_2}(z)\tau_{\lambda_1\lambda_2}^{10}(z)
\tau_{\lambda_0\lambda_1}^{10}(z)\tau^{01}_{\lambda_0}(-\check\partial 
s_{\lambda_0\lambda_1\lambda_2}(z),z)^*\tau_{\lambda_0\lambda_2}^{10}(z)^*\\
=\phi(\alpha)&^{20}_{\lambda_0\lambda_1\lambda_2}(z)\tau_{\lambda_1\lambda_2}^{10}(z)
\tau_{\lambda_0\lambda_1}^{10}(z)\tau^{01}_{\lambda_0}(\check\partial 
s_{\lambda_0\lambda_1\lambda_2}(z),z)\tau_{\lambda_0\lambda_2}^{10}(z)^*.\quad
\end{align*}
Therefore we have shown
\begin{equation}\label{diff20}
\phi(\chi)^{20}_{\lambda_0\lambda_1\lambda_2}(z)=
\phi(\alpha)^{20}_{\lambda_0\lambda_1\lambda_2}(z)
\tau_{\lambda_1\lambda_2}^{10}(z)\tau_{\lambda_0\lambda_1}^{10}(z)
\tau^{01}_{\lambda_0}(\check\partial s_{\lambda_0\lambda_1\lambda_2}(z),z)
\tau_{\lambda_0\lambda_2}^{10}(z)^*.
\end{equation}
Putting together the three equations (\ref{diff02}), (\ref{diff11}) and (\ref{diff20}), we have proved

\begin{equation}
(\phi(\chi)^{20},\phi(\chi)^{11},\phi(\chi)^{02})=(\phi(\alpha)^{20},\phi(\alpha)^{11},
\phi(\alpha)^{02})D_{\check\partial s}(\tau^{10},\tau^{01}).
\end{equation}
Finally, we need to prove that $\Xi_{\U,s}:(CT(X,\delta),\alpha))\mapsto 
[\phi(\alpha)^{20},\phi(\alpha)^{11},\phi(\alpha)^{02}]$ is well-defined. 
We use the same proof as above, except we set $\Phi=\operatorname{id}$ and $w_g(x)=1$.
\end{proof}

\begin{cor}\label{XiHomomorphism}
Let $(X,\U,s)$ be in the standard setup. Then the map $[CT(X,\delta),\alpha]$ 
$\mapsto [\phi(\alpha)^{20},\phi(\alpha)^{11},\phi(\alpha)^{02}]$ defines a homomorphism of groups
$\Xi_{\U,s}:\mathrm{Br}_{\R^n}(X)\mapsto {\mathbb H}^2_{\check\partial s}(\pi(\U),{\mathcal S}).$
\end{cor}
\begin{proof}
Let $[CT(X,\delta),\alpha],[CT(X,\delta'),\chi]\in \mathrm{Br}_{\R^n}(X)$. We have to show
\[\Xi_{\U,s}[CT(X,\delta),\alpha]+ \Xi_{\U,s}[CT(X,\delta'),\chi]=
\Xi_{\U,s}[CT(X,\delta)\otimes_{C_0(X)}CT(X,\delta'),\alpha\otimes\chi],\] 
where $\otimes_{C_0(X)}$ denotes the balanced tensor product from 
Section \ref{ctstracesection}. Suppose $\{\Phi_{\lambda_0}\}$ and 
$\{\Psi_{\lambda_0}\}$ are local trivialisations of $CT(X,\delta)$ and 
$CT(X,\delta')$ respectively. Suppose also $(\phi(\alpha)^{20},\phi(\alpha)^{11},\phi(\alpha)^{02})$ 
and $(\phi(\chi)^{20},\phi(\chi)^{11},\phi(\chi)^{02})$ are such that
\begin{align*}
(CT(X,\delta),\alpha)\mapsto [\phi(\alpha)^{20},\phi(\alpha)^{11},\phi(\alpha)^{02}] \mbox{ and}\quad
(CT(X,\delta'),\chi)\mapsto [\phi(\chi)^{20},\phi(\chi)^{11},\phi(\chi)^{02}].
\end{align*}
We will assume that $(\phi(\alpha)^{20},\phi(\alpha)^{11},\phi(\alpha)^{02})$ and 
$(\phi(\chi)^{20},\phi(\chi)^{11},\phi(\chi)^{02})$ are defined using unitary valued lifts 
$\{v_{\lambda_0}\}, \{v_{\lambda_0\lambda_1}\}$ and $\{\tilde{v}_{\lambda_0}\}, 
\{\tilde{v}_{\lambda_0\lambda_1}\}$ respectively.

Now, since $\K(\H)\otimes \K(\H)\cong \K(\H\otimes \H)$, it is clear that \[\Phi_{\lambda_0}
\otimes_{C_0(X)} \Psi_{\lambda_0}(a\otimes a')(x):=\Phi_{\lambda_0}(a)(x)\otimes
\Psi_{\lambda_0}(a')(x)\]
defines a local trivialisation of $CT(X,\delta)\otimes_{C_0(X)}CT(X,\delta')$. 
Moreover, it is straightforward to see that, if we define $\beta^{\alpha\otimes\chi,\Phi\otimes\Psi}$ by
\[\beta^{\alpha\otimes\chi,\Phi\otimes\Psi}_{(\lambda_0(s,x)\lambda_1)}(T)=
\Phi_{\lambda_1}\otimes_{C_0(X)} \Psi_{\lambda_1}(a\otimes a')(x)\quad T\in \K(\H\otimes \H),\]
for $\Phi_{\lambda_0}\otimes_{C_0(X)} \Psi_{\lambda_0}(a\otimes a')(-s\cdot x)=T$, then
\begin{align*}
\beta^{\alpha\otimes\chi,\Phi\otimes\Psi}_{(\lambda_0(-s_{\lambda_0\lambda_1}(z),
\sigma_{\lambda_1}(z))\lambda_1)}=\operatorname{Ad}v_{\lambda_0\lambda_1}(z)\otimes
 \tilde{v}_{\lambda_0\lambda_1}(z),\ \mbox{and}\
\beta^{\alpha\otimes\chi,\Phi\otimes\Psi}_{(\lambda_0(m,\sigma_{\lambda_1}(z))\lambda_0)}=
&\operatorname{Ad}v_{\lambda_0}^m(z)\otimes \tilde{v}_{\lambda_0}^m(z).
\end{align*}
With these facts one observes
\[[CT(X,\delta)\otimes_{C_0(X)}CT(X,\delta'),\alpha\otimes \chi]\mapsto 
[\phi(\alpha)^{20}\phi(\chi)^{20},\phi(\alpha)^{11}\phi(\chi)^{11},\phi(\alpha)^{02}\phi(\chi)^{02}].\]
\end{proof}

\begin{theorem}\label{MainIsomorphism}
Let $(X,\U,s)$ be in the standard setup. Then 
$\Xi_{\U,s}$ is a group isomorphism.
\end{theorem}

To prove injectivity we use an obvious argument.

\begin{prop}\label{BrauerCommuting}
Let $(X,\U,s)$ be in the standard setup, and recall $\operatorname{M}$ denotes the 
Mackey obstruction map
$\operatorname{M}:\mathrm{Br}_{\R^n}(X)\to C(Z,M_n^u(\T)).$
Then there is a commutative diagram

\centerline{\xymatrix{
\mathrm{Br}_{\R^n}(X)|_{\operatorname{M}=0}\ar[rr]^{\cong}\ar[drr] && 
H_{\R^n}^2(\pi^{-1}(\pi(\U)),{\mathcal S})\ar[d]\\
&&{\mathbb H}^2_{\check\partial s}(\pi(\U),{\mathcal S}).
}}

\end{prop}
\begin{proof}
Recall from Theorem \ref{PRWagreement} that the map $H_{\R^n}^2(\pi^{-1}(\pi(\U)),
{\mathcal S})\mapsto {\mathbb H}^2_{\check\partial s}(\pi(\U),{\mathcal S})$ is 
given by $(\nu,\eta)\mapsto(\phi^{20}(\nu,\eta),\phi^{11}(\nu,\eta),\phi^{02}(\nu,\eta))$ with:
\begin{align*}
&\phi(\nu,\eta)^{20}_{\lambda_0\lambda_1\lambda_2}(z):=\nu_{\lambda_0\lambda_1\lambda_2}
(\sigma_{\lambda_2}(z))\eta_{\lambda_0\lambda_{1}}(-s_{\lambda_{1}\lambda_2}(z),
\sigma_{\lambda_2}(z)), \\
&\phi(\nu,\eta)^{11}_{\lambda_0\lambda_1}(m,z):=
\eta_{\lambda_0\lambda_{1}}(m,\sigma_{\lambda_{1}}(z)),\quad \mbox {and} \quad
\phi(\nu,\eta)^{02}:= 1.
\end{align*}
Suppose that $[CT(X,\delta),\alpha]$ has trivial Mackey obstruction and 
$[CT(X,\delta),\alpha]\mapsto [\nu,\eta]\in H^2_{\R^n}(\pi^{-1}(\pi(\U)),{\mathcal S})$ 
under the map from Section \ref{inclusion}. Suppose also that $\Xi_{\U,s}[CT(X,\delta),
\alpha]=[\phi(\alpha)^{20},\phi(\alpha)^{11},\phi(\alpha)^{02}]$. Then we need to show
\[[\phi(\alpha)^{20},\phi(\alpha)^{11},\phi(\alpha)^{02}]=[\phi(\nu,\eta)^{20},\phi(\nu,\eta)^{11},
\phi(\nu,\eta)^{02}].\]
First, we revisit the construction of $(\nu,\eta)$. Recall from Theorem \ref{invarianttheorem} 
that there is an open cover $\W=\{W_{\lambda_0}\}$ of $Z$, local trivialisations 
$\Phi_{\lambda_0}:CT(X,\delta)|_{\pi^{-1}(W_{\lambda_0})}\to C_0(\pi^{-1}(W_{\lambda_0}),\K)$ 
and continuous maps $u_{\lambda_0}:G\to C(\pi^{-1}(W_{\lambda_0}),U(\H))$ such that for any 
section $h$ in $C_0(\pi^{-1}(W_{\lambda_0}),\K)$
\begin{align}
\Phi_{\lambda_0}\circ\alpha_{g_1}\circ \Phi_{\lambda_0}^{-1}(h)= 
\operatorname{Ad}u_{\lambda_0}^{g_1}\circ \tau_{g_1}(h),\\
\label{g1g2uequationch3}u_{\lambda_0}^{g_1g_0}(x)=u_{\lambda_0}^{g_1}(x)
u_{\lambda_0}^{g_0}(g_1^{-1}x).
\end{align}
In this case, we may as well assume $\W=\pi(\U)$, and that there are continuous maps 
$\tilde{v}_{\lambda_0\lambda_1}:\pi^{-1}(W_{\lambda_0\lambda_1})\to U(\H)$ such that 
$\Phi_{\lambda_1}\circ \Phi_{\lambda_0}^{-1}=\operatorname{Ad} \tilde{v}_{\lambda_0\lambda_1}$. 
Then we defined
\begin{align*}
\nu_{\lambda_0\lambda_1\lambda_2}(x):=&\tilde{v}_{\lambda_1\lambda_2}(x)
\tilde{v}_{\lambda_0\lambda_1}(x)\tilde{v}_{\lambda_0\lambda_2}(x)^*,\\
\eta_{\lambda_0\lambda_1}(s,x):=&u_{\lambda_1}^s(x)\tilde{v}_{\lambda_0\lambda_1}
(-s\cdot x)u_{\lambda_0}^s(x)^*\tilde{v}_{\lambda_0\lambda_1}(x)^*.
\end{align*}
On the other hand, to define $(\phi(\alpha)^{20},\phi(\alpha)^{11},\phi(\alpha)^{02})$ 
we need to find continuous unitary valued maps that implement $z\mapsto \beta_{(\lambda_0
(-s_{\lambda_0\lambda_1}(z),\sigma_{\lambda_1}(z))\lambda_1)}$ and $(m,z)\mapsto 
\beta_{(\lambda_0(m,\sigma_{\lambda_0}(z))\lambda_0)}$. Fix $T\in \K$, $z\in \pi(U_{\lambda_0})$ 
and let $h\in C(U_{\lambda_0},\K)$ be such that $h(\sigma_{\lambda_0}(z))=T$. Then
$$
\beta_{(\lambda_0(m,\sigma_{\lambda_0}(z))\lambda_0)}(T)=\Phi_{\lambda_0}\circ\alpha_m
\circ\Phi_{\lambda_0}^{-1}(h)(\sigma_{\lambda_0}(z))
=\operatorname{Ad}u_{\lambda_0}^{m}\circ \tau_{m}(h)(\sigma_{\lambda_0}(z))
=\operatorname{Ad}u_{\lambda_0}^{m}(\sigma_{\lambda_0}(z))(T).
$$
Note that (\ref{g1g2uequationch3}) implies it is also the case that (for $m\in\Z^n$)
\[u^m_{\lambda_0}(\sigma_{\lambda_0}(z))=(u^{e_1}_{\lambda_0}
(\sigma_{\lambda_0}(z))^{m_1}(u^{e_2}_{\lambda_0}(\sigma_{\lambda_0}(z))^{m_2}
\dots (u^{e_m}_{\lambda_0}(\sigma_{\lambda_0}(z))^{m_1}.\]
(This requirement was necessary: cf. Equation (\ref{vsplit})). Therefore we may suppose 
$\beta_{(\lambda_0(m,\sigma_{\lambda_0}(z))\lambda_0)}$ is implemented by 
$(m,z)\mapsto u_{\lambda_0}^{m}(\sigma_{\lambda_0}(z))$.  Let $T\in \K$, 
$z\in \pi(U_{\lambda_0})\cap \pi(U_{\lambda_1})$ and $h\in C(U_{\lambda_0},\K)$ 
satisfy $h(\sigma_{\lambda_0}(z))=T$. Then
\begin{align*}
\beta_{(\lambda_0(-s_{\lambda_0\lambda_1}(z),\sigma_{\lambda_1}(z))\lambda_1)}(T)=
&\Phi_{\lambda_1}\circ\alpha_{-s_{\lambda_0\lambda_1}(z)}\circ \Phi_{\lambda_0}^{-1}(h)
(\sigma_{\lambda_1}(z))\\
=&\Phi_{\lambda_1}\circ \Phi_{\lambda_0}^{-1}\circ\Phi_{\lambda_0}\circ
\alpha_{-s_{\lambda_0\lambda_1}(z)}\circ \Phi_{\lambda_0}^{-1}(h)(\sigma_{\lambda_1}(z))\\
=&\operatorname{Ad}\tilde{v}_{\lambda_0\lambda_1}(\sigma_{\lambda_1}(z))
\operatorname{Ad}u_{\lambda_0}^{-s_{\lambda_0\lambda_1}(z)}(\sigma_{\lambda_1}(z))
\tau_{-s_{\lambda_0\lambda_1}(z)}(h)(\sigma_{\lambda_1}(z))\\
=&\operatorname{Ad}\tilde{v}_{\lambda_0\lambda_1}(\sigma_{\lambda_1}(z))
u_{\lambda_0}^{-s_{\lambda_0\lambda_1}(z)}(\sigma_{\lambda_1}(z))(T).
\end{align*}
Therefore we may suppose $\beta_{(\lambda_0(-s_{\lambda_0\lambda_1}(z),
\sigma_{\lambda_1}(z))\lambda_1)}$ is implemented by
\[z\mapsto \tilde{v}_{\lambda_0\lambda_1}(\sigma_{\lambda_1}(z))u_{\lambda_0}^{-
s_{\lambda_0\lambda_1}(z)}(\sigma_{\lambda_1}(z)).\]
With this data we can calculate $(\phi(\alpha)^{20},\phi(\alpha)^{11},\phi(\alpha)^{02})$. 
The last entry is easy:
\begin{align*}
\phi(\alpha)^{02}_{\lambda_0}(z):=&u_{\lambda_0}^{e_j}(\sigma_{\lambda_0}(z))
u_{\lambda_0}^{e_i}(\sigma_{\lambda_0}(z))u_{\lambda_0}^{e_i+e_j}(\sigma_{\lambda_0}(z))^*
=1 \quad\mbox{by } (\ref{g1g2uequationch3})\\
=&\phi^{02}(\nu,\eta)_{\lambda_0}(z).
\end{align*}
For the middle entry we need to apply (\ref{g1g2uequationch3}):
\begin{align*}
\phi(\alpha)^{11}_{\lambda_0\lambda_1}(m,z):=& u_{\lambda_1}^{m}
(\sigma_{\lambda_1}(z))\tilde{v}_{\lambda_0\lambda_1}(\sigma_{\lambda_1}(z))
u_{\lambda_0}^{-s_{\lambda_0\lambda_1}(z)}(\sigma_{\lambda_1}(z))\\
&\quad\times u_{\lambda_0}^{m}(\sigma_{\lambda_0}(z))^*u_{\lambda_0}^{-
s_{\lambda_0\lambda_1}(z)}(\sigma_{\lambda_1}(z))^*\tilde{v}_{\lambda_0\lambda_1}
(\sigma_{\lambda_1}(z))^*\\
=&\phi(\nu,\eta)^{11}_{\lambda_0\lambda_1}(m,z)\tilde{v}_{\lambda_0\lambda_1}
(\sigma_{\lambda_1}(z))u_{\lambda_0}^{m}(\sigma_{\lambda_1}(z))u_{\lambda_0}^{-
s_{\lambda_0\lambda_1}(z)}(\sigma_{\lambda_1}(z))\\
&\quad\times u_{\lambda_0}^{m}(\sigma_{\lambda_0}(z))^*u_{\lambda_0}^{-
s_{\lambda_0\lambda_1}(z)}(\sigma_{\lambda_1}(z))^*\tilde{v}_{\lambda_0\lambda_1}
(\sigma_{\lambda_1}(z))^*\\
=&\phi(\nu,\eta)^{11}_{\lambda_0\lambda_1}(m,z)\tilde{v}_{\lambda_0\lambda_1}
(\sigma_{\lambda_1}(z))u_{\lambda_0}^{m-s_{\lambda_0\lambda_1}(z)}(\sigma_{\lambda_1}(z))\\
&\quad\times u_{\lambda_0}^{m-s_{\lambda_0\lambda_1}(z)}(\sigma_{\lambda_1}
(z))^*\tilde{v}_{\lambda_0\lambda_1}(\sigma_{\lambda_1}(z))^*
=\phi(\nu,\eta)^{11}_{\lambda_0\lambda_1}(m,z).
\end{align*}
For the remaining entry, we  apply (\ref{g1g2uequationch3}) again to obtain an identity:
\begin{align*}
&u^{-\check\partial s_{\lambda_0\lambda_1\lambda_2}(z)}_{\lambda_0}
(\sigma_{\lambda_0}(z))^*u^{-s_{\lambda_0\lambda_2}(z)}(\sigma_{\lambda_2}(z))^*\\
&=u^{-s_{\lambda_0\lambda_1}(z)}_{\lambda_0}(\sigma_{\lambda_1}(z))^*
u^{s_{\lambda_0\lambda_2}(z)-s_{\lambda_1\lambda_2}(z)}_{\lambda_0}
(\sigma_{\lambda_0}(z))^*u^{-s_{\lambda_0\lambda_2}(z)}_{\lambda_0}(\sigma_{\lambda_2}(z))^*\\
&=u^{-s_{\lambda_0\lambda_1}(z)}_{\lambda_0}(\sigma_{\lambda_1}(z))^*
u^{s_{\lambda_0\lambda_2}(z)-s_{\lambda_1\lambda_2}(z)}_{\lambda_0}(s_{\lambda_0\lambda_2}(z)
\cdot\sigma_{\lambda_2}(z))^*u^{-s_{\lambda_0\lambda_2}(z)}_{\lambda_0}(\sigma_{\lambda_2}(z))^*\\
&=u^{-s_{\lambda_0\lambda_1}(z)}_{\lambda_0}(\sigma_{\lambda_1}(z))^*
u^{-s_{\lambda_1\lambda_2}(z)}_{\lambda_0}(\sigma_{\lambda_2}(z))^*.
\end{align*}
With this information we can compute
\begin{align*}
\phi(\alpha&^{20}_{\lambda_0\lambda_1\lambda_2}(z):=\tilde{v}_{\lambda_1\lambda_2}
(\sigma_{\lambda_2}(z))u_{\lambda_1}^{-s_{\lambda_1\lambda_2}(z)}(\sigma_{\lambda_2}(z))
\tilde{v}_{\lambda_0\lambda_1}(\sigma_{\lambda_1}(z))u_{\lambda_0}^{-s_{\lambda_0\lambda_1}(z)}
(\sigma_{\lambda_1}(z))\\
&\quad\times u_{\lambda_0}^{-\check\partial s_{\lambda_0\lambda_1\lambda_2}(z)}(\sigma_{\lambda_0}
(z))u_{\lambda_0}^{-s_{\lambda_0\lambda_2}(z)}(\sigma_{\lambda_2}(z))^*
\tilde{v}_{\lambda_0\lambda_2}(\sigma_{\lambda_2}(z))^*\\
=&\tilde{v}_{\lambda_1\lambda_2}(\sigma_{\lambda_2}(z))u_{\lambda_1}^{-s_{\lambda_1\lambda_2}(z)}
(\sigma_{\lambda_2}(z))\tilde{v}_{\lambda_0\lambda_1}(\sigma_{\lambda_1}(z))u_{\lambda_0}^{-
s_{\lambda_0\lambda_1}(z)}(\sigma_{\lambda_1}(z))\\
&\quad\times u^{-s_{\lambda_0\lambda_1}(z)}_{\lambda_0}(\sigma_{\lambda_1}(z))^*u^{-
s_{\lambda_1\lambda_2}(z)}_{\lambda_0}(\sigma_{\lambda_2}(z))^*\tilde{v}_{\lambda_0\lambda_2}
(\sigma_{\lambda_2}(z))^*\\
=&\tilde{v}_{\lambda_1\lambda_2}(\sigma_{\lambda_2}(z))u_{\lambda_1}^{-s_{\lambda_1\lambda_2}(z)}
(\sigma_{\lambda_2}(z))\tilde{v}_{\lambda_0\lambda_1}(\sigma_{\lambda_1}(z)) u^{-
s_{\lambda_1\lambda_2}(z)}_{\lambda_0}(\sigma_{\lambda_2}(z))^*\tilde{v}_{\lambda_0\lambda_2}
(\sigma_{\lambda_2}(z))^*.\\
=&\tilde{v}_{\lambda_1\lambda_2}(\sigma_{\lambda_2}(z))
u_{\lambda_1}^{-s_{\lambda_{1}\lambda_2}(z)}(\sigma_{\lambda_2}(z))
\tilde{v}_{\lambda_0\lambda_1}(\sigma_{\lambda_1}(z))
u_{\lambda_0}^{-s_{\lambda_{1}\lambda_2}(z)}(\sigma_{\lambda_2}(z))^* \\
&\quad\times[\tilde{v}_{\lambda_0\lambda_1}(\sigma_{\lambda_2}(z))^*
\tilde{v}_{\lambda_0\lambda_1}(\sigma_{\lambda_2}(z))]\tilde{v}_{\lambda_0\lambda_2}
(\sigma_{\lambda_2}(z))^*\\
=&\tilde{v}_{\lambda_1\lambda_2}(\sigma_{\lambda_2}(z))\eta_{\lambda_0\lambda_{1}}
(-s_{\lambda_{1}\lambda_2}(z),\sigma_{\lambda_2}(z))\tilde{v}_{\lambda_0\lambda_1}
(\sigma_{\lambda_2}(z))\tilde{v}_{\lambda_0\lambda_2}(\sigma_{\lambda_2}(z))^* \\
=&\tilde{v}_{\lambda_1\lambda_2}(\sigma_{\lambda_2}(z))\tilde{v}_{\lambda_0\lambda_1}
(\sigma_{\lambda_2}(z))\tilde{v}_{\lambda_0\lambda_2}(\sigma_{\lambda_2}(z))^*
\eta_{\lambda_0\lambda_{1}}(-s_{\lambda_{1}\lambda_2}(z),\sigma_{\lambda_2}(z)) \\
=&\nu_{\lambda_0\lambda_1\lambda_2}(\sigma_{\lambda_2}(z))\eta_{\lambda_0\lambda_{1}}
(-s_{\lambda_{1}\lambda_2}(z),\sigma_{\lambda_2}(z))=\phi(\nu,
\eta)^{20}_{\lambda_0\lambda_1\lambda_2}(z).\quad\quad\quad
\end{align*}
\end{proof}

\begin{cor}\label{EquivariantInjective}
The map $H_{\R^n}^2(\pi^{-1}(\pi(\U)),{\mathcal S})\to 
{\mathbb H}^2_{\check\partial s}(\pi(\U),{\mathcal S})$ is injective.
\end{cor}
\begin{proof}
Fix $[\nu,\eta]$, and assume that the image $[\phi(\nu,\eta)^{20},\phi(\nu,\eta)^{11},1]$ in 
${\mathbb H}^2_{\check\partial s}(\pi(\U),{\mathcal S})$ is the identity element $[1,1,1]$. 
Thus, there is a 1-cochain $(\varphi^{10},\varphi^{01})\in C^1_{\check\partial s}(\pi(\U),{\mathcal S})$ 
such that \[[\phi(\nu,\eta)^{20},\phi(\nu,\eta)^{11},1]=
[\check\partial \varphi^{10}+\varphi^{01}\cup \check\partial s,\check\partial \varphi^{01},1]\in 
{\mathbb H}^2_{\check\partial s}(\pi(\U),{\mathcal S}).\]
Recalling theorems \ref{PRWGysin} and \ref{PRWagreement}, we observe that 
$\check\partial\phi(\nu,\eta)^{11}=1$ and therefore $[\phi(\nu,\eta)^{11}]$ is a 
well defined class in $\check{H}^1(\pi(\U),{\mathcal S})$. Moreover, 
$[\phi(\nu,\eta)^{11}]\in\check{H}^1(\pi(\U),{\mathcal S})$  is precisely the image of 
$[\nu,\eta]$ under the ``integration" map $\pi_*:H_{\R^n}^2(\pi^{-1}\pi(\U),{\mathcal S})\to 
\check{H}^1(\pi(\U),{\mathcal S})$. By exactness of the Gysin sequence from Theorem 
\ref{PRWGysin}, the fact that $[\phi(\nu,\eta)^{11}]=[\check\partial \varphi^{01}]=1$ 
implies there exists a $\gamma\in \check{Z}^2(\pi(\U),{\mathcal S})$ such that 
$[\nu,\eta]=\pi^*_{\R^n}[\gamma]:=[\pi^*\gamma,1]$. Thus, we find
\[[1,1,1]=[\phi^{20}(\nu,\eta),\phi^{11}(\nu,\eta),1]=[\phi^{20}(\pi^*\gamma,1),1,1].\]
Using this, and the fact that $\phi(\pi^*\gamma,1)^{20}=\gamma$, shows there must exist a 1-cochain 
$(\tilde{\varphi}^{10},\tilde{\varphi}^{01})\in C^1_{\check\partial s}(\pi(\U),{\mathcal S})$ such that
$\gamma=\check\partial \tilde{\varphi}^{10}+ \tilde{\varphi}^{01}\cup \check\partial s$ 
and $\check\partial \tilde{\varphi}^{01}=1.$
Now, using exactness of the Gysin sequence from Theorem \ref{PRWGysin} 
we see that $\pi^*_{\R^n}\circ (\cup \check\partial s)=1$. Therefore
\[[\nu,\eta]=\pi_{\R^n}^*[\gamma]=\pi_{\R^n}^*[\gamma-\check\partial \tilde{\varphi}^{10}]=
\pi_{\R^n}^*([\tilde{\varphi}^{01}]\cup \check\partial s)=1,\]
which implies $H_{\R^n}^2(\pi^{-1}(\pi(\U)),{\mathcal S})\to 
{\mathbb H}^2_{\check\partial s}(\pi(\U),{\mathcal S})$ is injective.
\end{proof}

\begin{cor}
Let $(X,\U,s)$ be in the standard setup. Then the map $\Xi_{\U,s}$  is injective.
\end{cor}

\begin{proof}
This follows from the diagram of Proposition \ref{BrauerCommuting} and 
Corollary \ref{EquivariantInjective}.
\end{proof}

Now we prove surjectivity of $\Xi_{\U,s}$. Unfortunately, there seems to be no way to explicitly 
construct an element of $\mathrm{Br}_{\R^n}(X)$ corresponding to a given element of 
${\mathbb H}^2_{\check\partial s}(\pi(\U),{\mathcal S})$. Therefore, in order to prove 
surjectivity, we go via $\check{H}^2(\R^n\ltimes X,{\mathcal S})$. First, we make the 
relationship between $\check{H}^2(\sigma\U^\bullet,{\mathcal S})$ and 
${\mathbb H}^2_{\check\partial s}(\pi(\U),{\mathcal S})$ as suggested by 
Lemma \ref{cocyclerelationships} more concrete.

\begin{prop}\label{TuDimRed}
Let $(X,\U,s)$ be in the standard setup, and fix a cocycle $(\phi^{20},\phi^{11},\phi^{02})\in 
Z^2_{\check\partial s}(\pi(\U),{\mathcal S})$. Suppose there exists an open cover $\U^1$ 
of $(\R^n\ltimes X)^{(1)}$ and a Tu-\v{C}ech cocycle $\varphi\in 
\check{Z}^2(\sigma\U^\bullet,{\mathcal S})$ such that for any open sets 
$U_{\lambda_{01}}^1, U_{\lambda_{00}}^1\in \U^1$ with
$(-s_{\lambda_0\lambda_1}(z),\sigma_{\lambda_1})\in U^1_{\lambda_0\lambda_1\lambda_{01}}$ and
$(m,\sigma_{\lambda_0})\in U^1_{\lambda_0\lambda_0\lambda_{00}},$
there exist continuous maps
$$\tau^{10}_{\lambda_0\lambda_1\lambda_{01}}: \{z\in \pi(U_{\lambda_0})\cap 
\pi(U_{\lambda_1}): (-s_{\lambda_0\lambda_1}(z),\sigma_{\lambda_1}(z))\in 
U^1_{\lambda_{01}}\}\to \T\quad \mbox{and}$$
$$\tau^{01}_{\lambda_0\lambda_{00}}(m,\cdot): \{z\in \pi(U_{\lambda_0}): 
(m,\sigma_{\lambda_0}(z))\in U_{\lambda_{00}}^1\}\to \T,$$
such that for arbitrary indices $\lambda_{01}',\lambda_{02}'$ satisfying
$$(-s_{\lambda_0\lambda_1}(z)+m,\sigma_{\lambda_1}(z))\in 
U^1_{\lambda_{01}'} \mbox{ and }
(-s_{\lambda_0\lambda_1}(z)-s_{\lambda_1\lambda_2}(z),\sigma_{\lambda_2}(z))\in 
U^1_{\lambda_{02}'}$$
the following three identities hold:

\begin{align}
\label{TuDimRedId1}\prod_{1\leq i<j\leq n}\phi^{02}_{\lambda_0}(z)_{ij}^{m_il_j}=
&\tau^{01}_{\lambda_0\lambda_{12}}(l,z)\tau^{01}_{\lambda_0\lambda_{01}}(m,z)
\tau^{01}_{\lambda_0\lambda_{02}}(m+l,z)^*
\varphi(\alpha)_{\lambda_0\lambda_0\lambda_0\lambda_{01}\lambda_{02}
\lambda_{12}}(m,l,\sigma_{\lambda_0}(z)),\\
\label{TuDimRedId2}\phi^{11}_{\lambda_0\lambda_1}(m,z)=&\tau^{01}_{\lambda_1
\lambda_{11}}(m,z)\tau^{01}_{\lambda_0\lambda_{00}}(m,z)^*
\varphi(\alpha)_{\lambda_0\lambda_1\lambda_1\lambda_{01}\lambda_{01}'\lambda_{11}}
(-s_{\lambda_{0}\lambda_{1}}(z),m,\sigma_{\lambda_1}(z))\\
&\quad\times \varphi(\alpha)_{\lambda_0\lambda_0\lambda_1\lambda_{00}\lambda_{01}'
\lambda_{01}}(m,-s_{\lambda_{0}\lambda_{1}}(z),\sigma_{\lambda_1}(z))^*,\notag
\end{align}\begin{align}
\label{TuDimRedId3}\phi^{20}_{\lambda_0\lambda_1\lambda_2}(z)=
&\tau^{10}_{\lambda_1\lambda_2\lambda_{12}}(z)\tau^{10}_{\lambda_0\lambda_1\lambda_{01}}(z)
\tau^{01}_{\lambda_0\lambda_{00}}(-\check\partial s_{\lambda_0\lambda_1\lambda_2}(z),z)^*
\tau^{10}_{\lambda_0\lambda_2\lambda_{02}}(z)^* \\
&\quad\times \varphi(\alpha)_{\lambda_0\lambda_1\lambda_2\lambda_{01}\lambda_{02}'
\lambda_{12}}(-s_{\lambda_0\lambda_1}(z),-s_{\lambda_1\lambda_2}(z),
\sigma_{\lambda_2}(z))\notag\\
&\quad\times \varphi(\alpha)_{\lambda_0\lambda_0\lambda_2\lambda_{00}\lambda_{02}'
\lambda_{02}}(-\check\partial s_{\lambda_0\lambda_1\lambda_2}(z),-s_{\lambda_0\lambda_2}(z),
\sigma_{\lambda_2}(z))^*\notag.
\end{align}
Then there is a class $[CT(X,\delta),\alpha]\in \operatorname{Br}_{\R^n}(X)$ with 
$\Xi_{\U,s}[CT(X,\delta),\alpha]=[\phi^{20},\phi^{11},\phi^{02}]$.
\end{prop}

\begin{proof}
We observe first that the proof of Corollary \ref{UpsilonSurjective} carries over verbatim to 
show that there exists $(CT(X,\delta),\alpha)\in \mathfrak{Br}_{\R^n}(X)$ and choices of 
local trivialisations and unitary lifts such that $\Upsilon(CT(X,\delta),\alpha)=\varphi$. 
By Lemma \ref{cocyclerelationships}, there exists an open cover $\mathcal{V}^1$ of 
$\R^n\ltimes X^{(1)}$, a cocycle $(\psi^{20},\psi^{11},\psi^{02})\in 
Z^2_{\check\partial s}(\pi(\U),{\mathcal S})$ and functions $\{\gamma^{10},\gamma^{01}\}$ 
satisfying an analogous set of identities to the ones in the statement of this proposition. 
Moreover, it is the case that $\Xi_{\U,s}[CT(X,\delta),\alpha]=[\psi^{20},\psi^{11},\psi^{02}]$.
We thus need
$[\psi^{20},\psi^{11},\psi^{02}]=[\phi^{20},\phi^{11},\phi^{02}]$.

By taking a common refinement of $\U^1$ and $\mathcal{V}^1$ (and leaving $\U^0$ untouched), 
we may as well assume $\mathcal{V}^1=\U^1$. We need to show that $(\tau^{10}(\gamma^{10})^*,
\tau^{01}(\gamma^{01})^*)$ is a well-defined cochain in 
$C^1_{\check\partial s}(\pi(\U),{\mathcal S})$ such that
$(\phi^{20},\phi^{11},\phi^{02})=(\psi^{20},\psi^{11},\psi^{02})
D_{\check\partial s}(\tau^{10}(\gamma^{10})^*,\tau^{01}(\gamma^{01})^*).$
First, identity (\ref{TuDimRedId1}) implies that
\begin{align*}
\prod_{1\leq i<j\leq n}\phi^{02}_{\lambda_0}(z)_{ij}^{m_il_j}
=&\tau^{01}_{\lambda_0\lambda_{12}}(l,z)\gamma^{01}_{\lambda_0\lambda_{12}}(l,z)^*
\tau^{01}_{\lambda_0\lambda_{01}}(m,z)\gamma^{01}_{\lambda_0\lambda_{01}}(m,z)^*\\
&\quad\times \tau^{01}_{\lambda_0\lambda_{02}}(m+l,z)^*\gamma^{01}_{\lambda_0\lambda_{02}}
(m+l,z) \prod_{1\leq i<j\leq n}\psi^{02}_{\lambda_0}(z)_{ij}^{m_il_j}.
\end{align*}
Using a symmetry argument, we see this implies
\[\prod_{1\leq i<j\leq n}\phi^{02}_{\lambda_0}(z)_{ij}^{m_il_j-m_jl_i}=\prod_{1\leq i<j\leq n}
\psi^{02}_{\lambda_0}(z)_{ij}^{m_il_j-m_jl_i}.\]
Therefore $\phi^{02}=\psi^{02}$, and
\begin{equation}\label{taugammastar}
\tau^{01}_{\lambda_0\lambda_{12}}(l,z)\gamma^{01}_{\lambda_0\lambda_{12}}(l,z)^*
\tau^{01}_{\lambda_0\lambda_{01}}(m,z)\gamma^{01}_{\lambda_0\lambda_{01}}(m,z)^*
\tau^{01}_{\lambda_0\lambda_{02}}(m+l,z)^*\gamma^{01}_{\lambda_0\lambda_{02}}(m+l,z)=1.
\end{equation}
Notice that Equation (\ref{taugammastar}) implies
\[\tau^{01}_{\lambda_0\lambda_{12}}(l,z)\gamma^{01}_{\lambda_0\lambda_{12}}(l,z)^*=
\tau^{01}_{\lambda_0\lambda_{01}}(m,z)^*\gamma^{01}_{\lambda_0\lambda_{01}}(m,z)
\tau^{01}_{\lambda_0\lambda_{02}}(m+l,z)\gamma^{01}_{\lambda_0\lambda_{02}}(m+l,z)^*,\]
and since the right hand side is independent of the index $\lambda_{12}$, the left hand side must 
be as well. It follows that the map
$ (m,z)\mapsto \tau^{01}_{\lambda_0\bullet}(m,z)\gamma^{01}_{\lambda_0\bullet}(m,z)^*$, for 
$z\in\pi(U_{\lambda_0}),$
which we denote by $(\tau\gamma^*)^{01}$, is a well-defined 0-cochain in 
$\check{C}^0(\pi(\U),\hat{{\mathcal N}})$.

Second, identity (\ref{TuDimRedId3}) implies
\begin{align*}
&\phi^{20}_{\lambda_0\lambda_1\lambda_2}(z)\tau^{10}_{\lambda_1\lambda_2\lambda_{12}}(z)^*
\tau^{10}_{\lambda_0\lambda_1\lambda_{01}}(z)^*\tau^{01}_{\lambda_0\lambda_{00}}(-\check\partial 
s_{\lambda_0\lambda_1\lambda_2}(z),z)\tau^{10}_{\lambda_0\lambda_2\lambda_{02}}(z) \\
=&\psi^{20}_{\lambda_0\lambda_1\lambda_2}(z)\gamma^{10}_{\lambda_1\lambda_2\lambda_{12}}(z)^*
\gamma^{10}_{\lambda_0\lambda_1\lambda_{01}}(z)^*\gamma^{01}_{\lambda_0\lambda_{00}}
(-\check\partial s_{\lambda_0\lambda_1\lambda_2}(z),z)
\gamma^{10}_{\lambda_0\lambda_2\lambda_{02}}(z).
\end{align*}
Thus
\begin{align*}
\tau^{10}_{\lambda_1\lambda_2\lambda_{12}}(z)\gamma^{10}_{\lambda_1\lambda_2
\lambda_{12}}(z)^*
=&\phi^{20}_{\lambda_0\lambda_1\lambda_2}(z)^*\psi^{20}_{\lambda_0\lambda_1\lambda_2}(z)
\tau^{10}_{\lambda_0\lambda_1\lambda_{01}}(z)\gamma^{10}_{\lambda_0\lambda_1\lambda_{01}}(z)^*
\tau^{01}_{\lambda_0\lambda_{00}}(-\check\partial s_{\lambda_0\lambda_1\lambda_2}(z),z)^*
\\
&\quad\times\gamma^{01}_{\lambda_0\lambda_{00}}(-\check\partial s_{\lambda_0\lambda_1
\lambda_2}(z),z)\tau^{10}_{\lambda_0\lambda_2\lambda_{02}}(z)^*
\gamma^{10}_{\lambda_0\lambda_2\lambda_{02}}(z),
\end{align*}
and since the right hand side is independent of the index $\lambda_{12}$, so too is the left hand side.
It follows that the map
$ z\mapsto \tau^{10}_{\lambda_0\lambda_1\bullet}(z)\gamma^{01}_{\lambda_0\lambda_1\bullet}(z)^*$,  
for $z\in\pi(U_{\lambda_0})\cap \pi(U_{\lambda_0}),$
which we denote $(\tau\gamma^*)^{10}$, is a well-defined 1-cochain in $\check{C}^1(\pi(\U),{\mathcal S})$.

Therefore, we have shown that $((\tau\gamma^*)^{10},(\tau\gamma^*)^{01})\in 
C^1_{\check\partial s}(\pi(\U),{\mathcal S})$, and it is straightforward to see that
$\phi^{02}=\psi^{02},$ and that
\begin{align*}
\phi^{11}_{\lambda_0\lambda_1}(m,z)=&\psi^{11}_{\lambda_0\lambda_1}(m,z)
(\tau\gamma^*)^{01}_{\lambda_1}(m,z)(\tau\gamma^*)^{01}_{\lambda_0}(m,z)^*,\\
\phi^{20}_{\lambda_0\lambda_1\lambda_2}(z)=&\psi^{20}_{\lambda_0\lambda_1\lambda_2}(z)
(\tau\gamma^*)^{10}_{\lambda_1\lambda_2}(z)(\tau\gamma^*)^{10}_{\lambda_0\lambda_1}(z)
(\tau\gamma^*)_{\lambda_0\lambda_2}(z)^*(\tau\gamma^*)^{01}_{\lambda_0}(\check\partial 
s_{\lambda_0\lambda_1\lambda_2}(z),z).
\end{align*}
Consequently
$\Xi_{\U,s}[CT(X,\delta),\alpha]=[\psi^{20},\psi^{11},\psi^{02}]=[\phi^{20},\phi^{11},\phi^{02}].$
\end{proof}

Proposition \ref{TuDimRed} gives us a strategy for the surjectivity proof but first
we need a preliminary result.

\begin{lemma}\label{Rnvaluedlifts}
Let $(X,\U,s)$ be in the standard setup. Define $\R^n/\Z^n$-equivariant functions 
$w_{\lambda_0}:\pi^{-1}(\pi(U_{\lambda_0}))\to \R^n/\Z^n$ by $w_{\lambda_0}(x)^{-1}
x:=\sigma_{\lambda_0}(\pi(x))$. 
Then for all indices ${\lambda_0}$ there is a continuous function $\tilde{w}_{\lambda_0}:
U_{\lambda_0}\to \R^n$ such that for all $x\in U_{\lambda_0}$,
$[\tilde{w}_{\lambda_0}(x)]_{\R^n/\Z^n}=w_{\lambda_0}(x).$
\end{lemma}

\begin{proof}
Consider the exact sequence
$C(U_{\lambda_0},\R^n)\to C(U_{\lambda_0},\T^n)\to \check{H}^1(U_{\lambda_0},\Z^n).$
Then $w_{\lambda_0}|_{U_{\lambda_0}}$ has a lift to a continuous function 
$\tilde{w}_{\lambda_0}:U_{\lambda_0}\to \R^n$ if and only if the image of $w_{\lambda_0}$ 
under the connecting homomorphism
$C(U_{\lambda_0},\T^n)\to \check{H}^1(U_{\lambda_0},\Z^n)$
is trivial. But, we have a locally constant sheaf over a contractible space $U_{\lambda_0}$,  
because $\U$ is 
good, which establishes the conclusion.
\end{proof}

\begin{cor}\label{surjectiveproof}
For $(X,\U,s)$ in the standard setup the map $\Xi_{\U,s}$  is surjective.
\end{cor}

\begin{proof}
Fix $(\phi^{20},\phi^{10},\phi^{02})$. By Proposition \ref{TuDimRed} we need  an open 
cover $\U^1$ of $\R^n\times X$, a Tu-\v{C}ech cocycle $\varphi\in Z^2(\sigma\U^\bullet,{\mathcal S})$ 
and functions $(\tau^{10},\tau^{01})$ satisfying the required identities. By Lemma \ref{Rnvaluedlifts} 
there are continuous functions $\tilde{w}_{\lambda_0}:U_{\lambda_0}\to \R^n/\Z^n$ that satisfy 
$[\tilde{w}_{\lambda_0}]_{\R^n/\Z^n}=w_{\lambda_0}$ and they can be chosen to also satisfy
\begin{align}\label{wchoice0}
&\tilde{w}_{\lambda_0}(\sigma_{\lambda_0}(z))=0\in\R^n.
\end{align}
We now claim that the function
$m_{\lambda_0\lambda_1}(s,x):=s_{\lambda_0\lambda_1}(\pi(x))-\tilde{w}_{\lambda_1}(x)+
\tilde{w}_{\lambda_0}(s^{-1}x)+s$
is continuous on $\{(s,x)\in \R^n\times X:x\in U_{\lambda_0}^0, s^{-1}x\in U_{\lambda_1}^0\}$, 
and takes values in $\Z^n$. Indeed, continuity follows from the fact that it is the sum of four 
continuous functions, and is defined on their common domain. To see that 
$m_{\lambda_0\lambda_1}(s,x)\in\Z^n$ we just use the fact that $w_{\lambda_0}$ is equivariant 
and recall that, for all $x\in \pi^{-1}(\pi(U_{\lambda_0}))\cap \pi^{-1}(\pi(U_{\lambda_1}))$ we have
\[[s_{\lambda_0\lambda_1}(\pi(x))]_{\R^n/\Z^n}w_{\lambda_1}(x)^{-1}w_{\lambda_0}(x)=[0]_{\R^n/\Z^n}.\]
Then we may now calculate:
$[s_{\lambda_0\lambda_1}(z)-\tilde{w}_{\lambda_1}(x)+\tilde{w}_{\lambda_0}(s^{-1}x)+s]_{\R^n/\Z^n}$
\begin{align*}
&\quad=[s_{\lambda_0\lambda_1}(z)]_{\R^n/\Z^n}w_{\lambda_1}(x)^{-1}
w_{\lambda_0}(s^{-1}x)[s]_{\R^n/\Z^n}=[0]_{\R^n/\Z^n}.
\end{align*}
For later use, note that (\ref{wchoice0}) implies
\begin{align}
\label{mwiths1}m_{\lambda_0\lambda_1}(-s_{\lambda_0\lambda_1}(z),\sigma_{\lambda_1}(z))=0, 
\mbox{ and}\\
\label{mwiths2}m_{\lambda_1\lambda_1}(m,\sigma_{\lambda_1}(z))=m,\quad m\in\Z^n.
\end{align}

Now we define an open cover $\U^1$ of $\R^n\ltimes X$ with open sets 
$U^1_{\lambda_0\lambda_1 l}$ indexed by $\mathcal{I}^1:=\mathcal{I}\times\mathcal{I}\times \Z^n 
$ (where the open cover $\U$ of $X$ is $\U=\{U_{\lambda_0}\}_{\lambda_0\in \mathcal{I}}$)  
and we set 
\[U^1_{\lambda_0\lambda_1 l}:=\{(s,x)\in \R^n\times X:s^{-1}x\in 
U_{\lambda_0}, x\in U_{\lambda_1}, m_{\lambda_0\lambda_1}(s,x)=l\}.\]
We next construct a cocycle $\varphi\in\check{Z}^2(\sigma\U^\bullet,{\mathcal S})$ with 
image $[\phi^{20},\phi^{11},\phi^{02}]\in H^2_F(\pi(\U),{\mathcal S})$. 
To simplify our exposition, we use the notation:
$\phi^{02}(m,l,z):=\prod_{1\leq i<j\leq n}\phi^{02}(z)^{m_il_j}_{ij},$ and
$F_{\lambda_0\lambda_1\lambda_2}(z):=\check\partial s_{\lambda_0\lambda_1\lambda_2}(z).$

Then we define a Tu-\v{C}ech 2-cocycle $\varphi\in\check{Z}^2(\sigma\U^\bullet,{\mathcal S})$ by
\begin{align*}
\varphi_{\lambda_0\lambda_1\lambda_2 (\lambda_0'\lambda_1' l_1)
(\lambda_2'\lambda_3'l_2)(\lambda_4'\lambda_5' l_3)}(s,t,x)
:=&\phi^{20}_{\lambda_0\lambda_1\lambda_2}(\pi(x))\phi^{11}_{\lambda_0\lambda_1}
(m_{\lambda_1\lambda_2}(t,x),\pi(x))\\
&\quad\times\phi^{02}(m_{\lambda_1\lambda_2}(t,x),m_{\lambda_0\lambda_1}(s,t^{-1}x),\pi(x))^*\\
&\quad\times\phi^{02}(F_{\lambda_0\lambda_1\lambda_2}(\pi(x)),m_{\lambda_0\lambda_2}(s+t,x),\pi(x)).
\end{align*}
We now take the Tu-\v{C}ech differential
term by term. Thus, suppressing the indices from $\U^1$ (since they do not determine 
the value of $\varphi$, only its domain) we define
\begin{align*}
&\varphi^a_{\lambda_0\lambda_1\lambda_2}(s,t,x):=
\phi^{20}_{\lambda_0\lambda_1\lambda_2}(\pi(x)),\quad\mbox{and}\quad
\varphi^b_{\lambda_0\lambda_1\lambda_2}(s,t,x):=
\phi^{11}_{\lambda_0\lambda_1}(m_{\lambda_1\lambda_2}(t,x),\pi(x)),\\
&\varphi^c_{\lambda_0\lambda_1\lambda_2}(s,t,x):=
\phi^{02}(m_{\lambda_1\lambda_2}(t,x),m_{\lambda_0\lambda_1}(s,t^{-1}x),\pi(x))^*,\\
&\varphi^d_{\lambda_0\lambda_1\lambda_2}(s,t,x):=
\phi^{02}(F_{\lambda_0\lambda_1\lambda_2}(z),m_{\lambda_0\lambda_2}(s+t,x),\pi(x)).
\end{align*}
Write $z:=\pi(x)$, and then the differentials of each of these terms are as follows.
\begin{align*}
\partial_{Tu}\varphi^a_{\lambda_0\lambda_1\lambda_2\lambda_3}(r,s,t,x)
=&\phi^{20}_{\lambda_1\lambda_2\lambda_3}(z)\phi^{20}_{\lambda_0\lambda_2\lambda_3}(z)^*
\phi^{20}_{\lambda_0\lambda_1\lambda_3}(z)\phi^{20}_{\lambda_0\lambda_1\lambda_2}(z)^*,\\
\partial_{Tu}\varphi^b_{\lambda_0\lambda_1\lambda_2\lambda_3}(r,s,t,x)
=&\phi^{11}_{\lambda_1\lambda_2}(m_{\lambda_2\lambda_3}(t,x),z)
\phi^{11}_{\lambda_0\lambda_2}(m_{\lambda_2\lambda_3}(t,x),z)^*\\
&\quad\times\phi^{11}_{\lambda_0\lambda_1}(m_{\lambda_1\lambda_3}(s+t,x),z)
\phi^{11}_{\lambda_0\lambda_1}(m_{\lambda_1\lambda_2}(s,t^{-1}x),z)^*\\
=&\phi^{11}_{\lambda_0\lambda_1}(F_{\lambda_1\lambda_2\lambda_3}(z),z)^*\check\partial
\phi^{11}_{\lambda_0\lambda_1\lambda_2}(\cdot,z)(m_{\lambda_2\lambda_3}(t,x)),
\end{align*}\begin{align*}
&\partial_{Tu}\varphi_{\lambda_0\lambda_1\lambda_2\lambda_3}^c(r,s,t,x)\\
&=\phi^{02}(m_{\lambda_2\lambda_3}(t,x),m_{\lambda_1\lambda_2}(s,t^{-1}x),z)^*\phi^{02}
(m_{\lambda_2\lambda_3}(t,x),m_{\lambda_0\lambda_2}(r+s,t^{-1}x),z)\\
&\quad\times\phi^{02}(m_{\lambda_1\lambda_3}(s+t,x),m_{\lambda_0\lambda_1}(r,s^{-1}t^{-1}x),z)^*
\phi^{02}(m_{\lambda_1\lambda_2}(s,t^{-1}x),m_{\lambda_0\lambda_1}(r,s^{-1}t^{-1}x),z)\\
&=\phi^{02}(m_{\lambda_2\lambda_3}(t,x),F_{\lambda_0\lambda_1\lambda_2}(z),z)^*\phi^{02}
(F_{\lambda_1\lambda_2\lambda_3}(z),m_{\lambda_0\lambda_1}(r,s^{-1}t^{-1}x),z),
\end{align*}
\begin{align*}
\mbox{and }\quad&\partial_{Tu}\varphi_{\lambda_0\lambda_1\lambda_2\lambda_3}^d(r,s,t,x)\\
&=\phi^{02}(F_{\lambda_1\lambda_2\lambda_3}(z),m_{\lambda_1\lambda_3}(s+t,x),z)\phi^{02}
(F_{\lambda_0\lambda_2\lambda_3}(z),m_{\lambda_0\lambda_3}(r+s+t,x),z)^*\\
&\quad\times\phi^{02}(F_{\lambda_0\lambda_1\lambda_3}(z),m_{\lambda_0\lambda_3}(r+s+t,x),z)
\phi^{02}(F_{\lambda_0\lambda_1\lambda_2}(z),m_{\lambda_0\lambda_2}(r+s,t^{-1}x),z)^*\\
&=\phi^{02}(F_{\lambda_1\lambda_2\lambda_3}(z),F_{\lambda_0\lambda_1\lambda_3}(z),z)
\phi^{02}(F_{\lambda_0\lambda_1\lambda_2}(z),F_{\lambda_0\lambda_2\lambda_3}(z),z)^*\\
&\quad\times\phi^{02}(F_{\lambda_0\lambda_1\lambda_2}(z),m_{\lambda_2\lambda_3}(t,x),z)
\phi^{02}(F_{\lambda_1\lambda_2\lambda_3}(z),m_{\lambda_0\lambda_1}(r,s^{-1}t^{-1}x),z)^*.
\end{align*}
We also need to recall a couple of facts from the $D_{F}$-cocycle identity:
\begin{align*}
&\check\partial\phi^{11}_{\lambda_0\lambda_1\lambda_2}(\cdot,z)(m_{\lambda_2\lambda_3}(t,x))=
\phi^{02}_{\lambda_3}(F_{\lambda_0\lambda_1\lambda_2}(z),m_{\lambda_2\lambda_3}(t,x),z)^*
\phi^{02}_{\lambda_3}(m_{\lambda_2\lambda_3}(t,x),F_{\lambda_0\lambda_1\lambda_2}(z),z),\\
&\check\partial\phi^{20}(z)_{\lambda_0\lambda_1\lambda_2\lambda_3}=\phi^{11}_{\lambda_0
\lambda_1}(F_{\lambda_1\lambda_2\lambda_3}(z),z)\phi^{02}(F_{\lambda_0\lambda_1\lambda_2}
(z),F_{\lambda_0\lambda_2\lambda_3}(z))\phi^{02}(F_{\lambda_1\lambda_2\lambda_3}
(z),F_{\lambda_0\lambda_1\lambda_3}(z))^*.
\end{align*}
Put together, these facts imply $\partial_{Tu}\varphi$ is given by
\begin{align*}
&\check\partial\phi^{20}_{\lambda_0\lambda_1\lambda_2\lambda_3}(z)
\phi^{11}_{\lambda_0\lambda_1}(F_{\lambda_1\lambda_2\lambda_3}(z),z)^*\check\partial
\phi^{11}_{\lambda_0\lambda_1\lambda_2}(\cdot,z)(m_{\lambda_2\lambda_3}(t,x))\\
&\quad\times\phi^{02}(m_{\lambda_2\lambda_3}(t,x),F_{\lambda_0\lambda_1\lambda_2}(z),z)^*
\phi^{02}(F_{\lambda_1\lambda_2\lambda_3}(z),m_{\lambda_0\lambda_1}(r,s^{-1}t^{-1}x),z)\\
&\quad\times\phi^{02}(F_{\lambda_1\lambda_2\lambda_3}(z),F_{\lambda_0\lambda_1\lambda_3}(z),z)
\phi^{02}(F_{\lambda_0\lambda_1\lambda_2}(z),F_{\lambda_0\lambda_2\lambda_3}(z),z)^*\\
&\quad\times\phi^{02}(F_{\lambda_0\lambda_1\lambda_2}(z),m_{\lambda_2\lambda_3}(t,x),z)\phi^{02}
(F_{\lambda_1\lambda_2\lambda_3}(z),m_{\lambda_0\lambda_1}(r,s^{-1}t^{-1}x),z)^*
\end{align*}\begin{align*}
&=[\phi^{11}_{\lambda_0\lambda_1}(F_{\lambda_1\lambda_2\lambda_3}(z),z)\phi^{02}
(F_{\lambda_0\lambda_1\lambda_2}(z),F_{\lambda_0\lambda_2\lambda_3}(z))\phi^{02}
(F_{\lambda_1\lambda_2\lambda_3}(z),F_{\lambda_0\lambda_1\lambda_3}(z))^*]\\
&\quad\times \phi^{11}_{\lambda_0\lambda_1}(F_{\lambda_1\lambda_2\lambda_3}
(z),z)^*[\phi^{02}_{\lambda_3}(F_{\lambda_0\lambda_1\lambda_2}(z),m_{\lambda_2
\lambda_3}(t,x),z)^*\phi^{02}_{\lambda_3}(m_{\lambda_2\lambda_3}
(t,x),F_{\lambda_0\lambda_1\lambda_2}(z),z)]\\
&\quad\times\phi^{02}(m_{\lambda_2\lambda_3}(t,x),F_{\lambda_0\lambda_1\lambda_2}(z),z)^*
\phi^{02}(F_{\lambda_1\lambda_2\lambda_3}(z),F_{\lambda_0\lambda_1\lambda_3}(z),z)\\
&\quad\times\phi^{02}(F_{\lambda_0\lambda_1\lambda_2}(z),F_{\lambda_0\lambda_2
\lambda_3}(z),z)^*\phi^{02}(F_{\lambda_0\lambda_1\lambda_2}(z),m_{\lambda_2\lambda_3}(t,x),z)
\quad=\quad1.
\end{align*}
Therefore, the cochain $\varphi$ is closed under $\partial_{Tu}$.
We now need to find functions $(\tau^{10},\tau^{01})$ satisfying the identities from Proposition
\ref{TuDimRed}.  First, let us define
\begin{align*}
\tau^{01}_{\lambda_0}(m,z):=\phi^{20}_{\lambda_0\lambda_0\lambda_0}(z)^*\prod_{1\leq i<j\leq
 n}f(z)_{ij}^{-m_im_j},\quad\quad \mbox{and  }\quad
\tau^{10}_{\lambda_0\lambda_1}(z):=1.
\end{align*}
Then, using (\ref{mwiths1}), we see that checking identity (\ref{TuDimRedId1}) amounts to the computation
\begin{align*}
&\varphi_{\lambda_0\lambda_0\lambda_0}(m,l,\sigma_{\lambda_0}(z))
\tau^{01}_{\lambda_0\cdot}(m,z)\tau_{\lambda_0\cdot}^{01}(l,z)\tau^{01}_{\lambda_0\cdot}(m+l,z)^*\\
&=(\phi^{20}_{\lambda_0\lambda_0\lambda_0}(z)\prod_{1\leq i<j\leq n}f(z)_{ij}^{-l_im_j})
(\phi^{20}_{\lambda_0\lambda_0\lambda_0}(z)^*\prod_{1\leq i<j\leq n}f(z)_{ij}^{-m_im_j})\\
&\quad\times (\phi^{20}_{\lambda_0\lambda_0\lambda_0}(z)^*\prod_{1\leq i<j\leq n}
f(z)_{ij}^{-l_il_j})(\phi^{20}_{\lambda_0\lambda_0\lambda_0}(z)\prod_{1\leq i<j\leq n}
f(z)_{ij}^{(m+l)_i(m+l)_j})=\prod_{1\leq i<j\leq n}f(z)_{ij}^{m_il_j}.
\end{align*}
For (\ref{TuDimRedId2}), the cocycle identity for $(\phi^{20},\phi^{11},\phi^{02})$ implies 
$\phi^{20}_{\lambda_0\lambda_0\lambda_1}(z)=\phi^{20}_{\lambda_0\lambda_0\lambda_0}(z)$ 
and $\phi^{20}_{\lambda_0\lambda_1\lambda_1}(z)=\phi^{20}_{\lambda_1\lambda_1\lambda_1}(z)$. 
Therefore, using (\ref{mwiths1}) and (\ref{mwiths2}) we see that
\begin{align*}
&\varphi_{\lambda_0\lambda_1\lambda_1}(-s_{\lambda_0\lambda_1}(z),m,
\sigma_{\lambda_1}(z))\varphi_{\lambda_0\lambda_0\lambda_1}(m,-s_{\lambda_0\lambda_1}(z),
\sigma_{\lambda_1}(z))^*\tau^{01}_{\lambda_1\cdot}(m,z)\tau^{01}_{\lambda_0\cdot}(m,z)^*\\
&\quad\quad=(\phi^{20}_{\lambda_0\lambda_1\lambda_1}(z)
\phi^{11}_{\lambda_0\lambda_1}(m,z))(\phi^{20}_{\lambda_0\lambda_0\lambda_1}(z))^*
\phi^{20}_{\lambda_1\lambda_1\lambda_1}(z)^*\phi^{20}_{\lambda_0\lambda_0\lambda_0}(z)
\quad=\quad \phi^{11}_{\lambda_0\lambda_1}(m,z).
\end{align*}
Finally, for identity (\ref{TuDimRedId3})
\begin{align*}
&\varphi_{\lambda_0\lambda_1\lambda_2}(-s_{\lambda_0\lambda_1}(z),
-s_{\lambda_1\lambda_2}(z),\sigma_{\lambda_2}(z))\varphi_{\lambda_0\lambda_0\lambda_2}
(-F_{\lambda_0\lambda_1\lambda_2}(z),-s_{\lambda_0\lambda_2}(z),\sigma_{\lambda_2}(z))^*\\
&\quad\times\tau^{01}_{\lambda_0}(F_{\lambda_0\lambda_1\lambda_2}(z),z)
\tau^{10}_{\lambda_1\lambda_2}(z)\tau^{10}_{\lambda_0\lambda_1}(z)
\tau^{10}_{\lambda_0\lambda_2}(z)^*\\
&=\phi^{20}_{\lambda_0\lambda_1\lambda_2}(z)\phi^{02}(F_{\lambda_0\lambda_1\lambda_2}(z),
-F_{\lambda_0\lambda_1\lambda_2}(z),\pi(x))\\
&\quad\times\phi^{20}_{\lambda_0\lambda_0\lambda_2}(\pi(x))^*
\phi^{20}_{\lambda_0\lambda_0\lambda_0}(z)\prod_{1\leq i<j\leq n}\phi^{02}
(\pi(x))_{ij}^{F_{\lambda_0\lambda_1\lambda_2}(z)_iF_{\lambda_0\lambda_1\lambda_2}(z)_j}
\quad=\quad\phi^{20}_{\lambda_0\lambda_1\lambda_2}(z).
\end{align*}
Therefore, by Proposition \ref{TuDimRed}, there exists a class $[CT(X,\delta),\alpha]$ 
such that  $\Xi_{\U,s}:[CT(X,\delta),\alpha]\mapsto[\phi^{20},\phi^{11},\phi^{02}]$, 
implying $\Xi_{\U,s}:\operatorname{Br}_{\R^n}(X)\to {\mathbb H}^2_F(\pi(\U^0),{\mathcal S})$ is surjective.
\end{proof}

The formula for the Tu-\v{C}ech cocycle in Corollary \ref{surjectiveproof} allows us to 
calculate the Dixmier-Douady class corresponding to a cocycle $[\phi^{20},\phi^{11},\phi^{02}]$.
Corollary \ref{surjectiveproof} completes the proof of Theorem \ref{MainIsomorphism}, 
which in turn provides the left downward isomorphism from the diagram of Theorem \ref{MainSquare}. 
To complete Theorem \ref{MainSquare} we need the right downward arrow
$\check{H}^3(X,\underline{\Z})|_{\pi^{0,3}=0}\stackrel{\cong}{\longrightarrow}
{\mathbb H}^3_{\check\partial s}(\pi(\U),\underline{\Z}),$
and  commutativity of the diagram.
The first step is:

\begin{theorem}[{\cite[Thm 2.2]{MatRos06}}]
Let $\delta\in\check{H}^3(X,\underline{\Z})$. Then $\delta$ is in the image of the forgetful 
homomorphism $F:\operatorname{Br}_{\R^n}(X)\to \check{H}^3(X,\underline{\Z})$
if an only if $\pi^{0,3}(\delta)=0$.
\end{theorem}
Next, as every class in $\check{H}^3(X,\underline{\Z})|_{\pi^{0,3}=0}$ lifts to 
$\mbox{Br}_{\R^n}(X)$ we can define a map implicitly from 
$\check{H}^3(X,\underline{\Z})|_{\pi^{0,3}=0}$ to 
${\mathbb H}^3_{\check\partial s}(\pi(\U),\underline{\Z})$ as the composition

\centerline{\xymatrix{
\mathrm{Br}_{\R^n}(X)\ar[r]^{\Xi_{\U,s}}&{\mathbb H}^2_{\check\partial s}(\pi(\U),{\mathcal S})\ar[d]\\
\check{H}^3(X,\underline{\Z})|_{\pi^{0,3}=0}\ar[u]^{F^{-1}}&
{\mathbb H}^3_{\check\partial s}(\pi(\U),\underline{\Z}),
}}

\noindent provided the image in ${\mathbb H}^3_{\check\partial s}(\pi(\U),\underline{\Z})$ is 
independent of the choice of the lift\\ $\check{H}^3(X,\underline{\Z})|_{\pi^{0,3}=0}\to 
\mathrm{Br}_{\R^n}(X)$. From  {\cite[Thm 2.3]{MatRos06}} we have:

\begin{prop}\label{BrauerKernel}
There is an exact sequence
\[H^2_M(\R^n,C(X,\T))\to \mathrm{Br}_{\R^n}(X)\to \check{H}^3(X,\underline{\Z})\vline_{\pi^{0,3}=0}\to 0.\]
\end{prop}
To find the image of $H^2_M(\R^n,C(X,\T))$ under the composition
\[H^2_M(\R^n,C(X,\T))\to \mathrm{Br}_{\R^n}(X)
\stackrel{\Xi_{\U,s}}{\longrightarrow}{\mathbb H}^2_{\check\partial s}(\pi(\U),{\mathcal S}),\]
we now describe $H^2_M(\R^n,C(X,\T))\to \mathrm{Br}_{\R^n}(X)$ in more detail. 
First, we know from \cite[Lemma 2.1]{MatRos05} that there is an isomorphism 
$C(Z,\wedge^2\R^n) \cong H^2_M(\R^n,C(X,\T))$ taking $g\in C(Z,M_n^u(\R))$ 
$\cong C(Z,\wedge^2\R^n)$, to  the cocycle $\tilde{g}$ in $Z^2_M(\R^n,C(X,\T))$ defined by
\[(s,t)\mapsto \left(x\mapsto \left[\sum_{1\leq i<j\leq n}g(\pi(x))_{ij}t_is_j\right]_{\R/\Z}\right).\]
We find the image of this in $\mathrm{Br}_{\R^n}(X)$ as follows. Let $\H=L^2(\R^n)$, 
and define a continuous map $L_{\tilde{g}}:\R^n\to C(X,U(L^2(\R^n)))$ with the formula
\[[L_{\tilde{g}}(s)(x)](\xi)(t):=\tilde{g}(s,t-s)(x)\xi(t-s),\quad s,t\in\R^n, x\in X, \xi\in L^2(\R^n).\]
Notice in particular that
\begin{align*}
L_{\tilde{g}}(s)(x)[L_{\tilde{g}}(t)(-s\cdot x)\xi](r)=&\tilde{g}(s,r-s)[L_{\tilde{g}}(t)(-s\cdot x)\xi](r-s) \\
=&\tilde{g}(s,r-s)(x)\tilde{g}(t,r-s-t)(-s\cdot x)\xi(r-s-t) \\
=&\tilde{g}(s,r-s)(x)\tilde{g}(t,r-s-t)(x)\xi(r-s-t)\quad\quad \\
=&\tilde{g}(s,t)(x)\tilde{g}(s+t,r-s-t)(x)\xi(r-s-t)\quad\quad \\
=&\tilde{g}(s,t)(x)[L_{\tilde{g}}(s+t)(x)\xi](r).
\end{align*}
It follows that the map $H^2_M(\R^n,C(X,\T))\cong C(Z,M_n^u(\R))\to \mathrm{Br}_{\R^n}(X)$ 
is given by (cf. \cite[Thm 5.1]{CroKumRaeWil97}):
$g\mapsto [C_0(X,\K),\operatorname{Ad}L_{\tilde{g}}\circ \tau]$.

\begin{lemma}\label{Liftindependence}
Let $(X,\U,s)$ be in the standard setup. Then the composition
$$C(Z,M_n^u(\R))\rightarrow \mathrm{Br}_{\R^n}(X)\stackrel{\Xi_{\U,s}}{\rightarrow} 
{\mathbb H}^2_{\check\partial s}(\pi(\U),{\mathcal S})\rightarrow
 {\mathbb H}^3_{\check\partial s}(\pi(\U),\underline{\Z})$$
\noindent is the zero map.
\end{lemma}

\begin{proof}
Let $g\in C(Z,M_n^u(\R))$ have image $(C_0(X,\K),\operatorname{Ad}
L_{\tilde{g}}\circ \tau)\in\mathfrak{Br}_{\R^n}(X)$. Then, if we take the identity 
maps as the local trivialisations, we have
\[\beta^{\alpha,\Phi}_{(\lambda_0(m,\sigma_{\lambda_0}(z))\lambda_0)}=
\operatorname{Ad}L_{\tilde{g}(\cdot,\cdot)(\sigma_{\lambda_0}(z))}(m)
\mbox{and}\
\beta^{\alpha,\Phi}_{(\lambda_0(-s_{\lambda_0\lambda_1}(z),\sigma_{\lambda_1}(z))\lambda_1)}
=\operatorname{Ad}L_{\tilde{g}(\cdot,\cdot)(\sigma_{\lambda_1}(z))}(-s_{\lambda_0\lambda_1}(z)).\]
Therefore we can define
\[u_{\lambda_0}^m(z):=(L_{\tilde{g}(\cdot,\cdot)(\sigma_{\lambda_0}(z))}(e_1))^{m_1}
\dots(L_{\tilde{g}(\cdot,\cdot)(\sigma_{\lambda_0}(z))}(e_n))^{m_n}
\mbox{and}\
u_{\lambda_0\lambda_1}(z):=L_{\tilde{g}(\cdot,\cdot)(\sigma_{\lambda_1}(z))}
(-s_{\lambda_0\lambda_1}(z)).\]
We can then calculate the image $[\phi(\operatorname{Ad}L\circ\tau)^{20},
\phi(\operatorname{Ad}L\circ\tau)^{11},\phi(\operatorname{Ad}L\circ\tau)^{02}]$ of $(C_0(X,\K),
\operatorname{Ad}L_{\tilde{g}}\circ \tau)$ in ${\mathbb H}^2_{\check\partial s}(\pi(\U),{\mathcal S})$ 
using
\begin{equation}\label{Lequation}
L_{\tilde{g}}(s)(x)L_{\tilde{g}}(t)(-s\cdot x)=\tilde{g}(s,t)(x)L_{\tilde{g}}(s+t)(x).
\end{equation}
as follows
\begin{align*}
\phi(\operatorname{Ad}L\circ\tau)^{02}_{\lambda_0}(z)_{ij}
&\quad=L_{\tilde{g}(\cdot,\cdot)(\sigma_{\lambda_0}(z))}(e_j)L_{\tilde{g}(\cdot,\cdot)
(\sigma_{\lambda_0}(z))}(e_i)L_{\tilde{g}(\cdot,\cdot)(\sigma_{\lambda_0}(z))}(e_i+e_j)^*\\
&\quad=\tilde{g}(e_j,e_i)(\sigma_{\lambda_0}(z))\quad=\quad [g(z)_{ij}]_{\R/\Z},\quad\quad\\
\phi(\operatorname{Ad}L\circ\tau)^{11}_{\lambda_0\lambda_1}(m,z)
&\quad=u_{\lambda_1}^m(z)u _{\lambda_0\lambda_1}(z)u_{\lambda_0}^m(z)^*
u_{\lambda_0\lambda_1}(z)^* \quad\mbox{ by } (\ref{Lequation})\times 2\\
&\quad=L_{\tilde{g}(\cdot,\cdot)(\sigma_{\lambda_1}(z))}(m)
L_{\tilde{g}(\cdot,\cdot)(\sigma_{\lambda_1}(z))}(-s_{\lambda_0\lambda_1}(z))\\
&\quad\quad\times L_{\tilde{g}(\cdot,\cdot)(\sigma_{\lambda_0}(z))}(m)^*
L_{\tilde{g}(\cdot,\cdot)(\sigma_{\lambda_1}(z))}(-s_{\lambda_0\lambda_1}(z))^*\\
&\quad=\tilde{g}(m,-s_{\lambda_0\lambda_1}(z))(\sigma_{\lambda_1}(z))\tilde{g}(-
s_{\lambda_0\lambda_1}(z),m)(\sigma_{\lambda_1}(z))^*\\
&\quad=\left[\sum_{1\leq i<j\leq n}g(z)_{ij}(m_is_{\lambda_0\lambda_1}(z)_j-s_{\lambda_0\lambda_1}
(z)_im_j)\right]_{\R/\Z},\quad\quad\quad\quad
\end{align*}\begin{align*}
\phi(\operatorname{Ad}L\circ\tau)^{20}_{\lambda_0\lambda_1\lambda_2}(z)
&\quad=u _{\lambda_1\lambda_2}(z)u _{\lambda_0\lambda_1}(z)
u_{\lambda_0}^{-\check\partial s_{\lambda_0\lambda_1\lambda_2}(z)}(z)
u _{\lambda_0\lambda_2}(z)^*\\
&\quad=L_{\tilde{g}(\cdot,\cdot)(\sigma_{\lambda_2}(z))}(-s_{\lambda_1\lambda_2}(z))
L_{\tilde{g}(\cdot,\cdot)(\sigma_{\lambda_1}(z))}(-s_{\lambda_0\lambda_1}(z))\\
&\quad\quad\times L_{\tilde{g}(\cdot,\cdot)(\sigma_{\lambda_0}(z))}(-\check\partial 
s_{\lambda_0\lambda_1\lambda_2}(z))^*L_{\tilde{g}(\cdot,\cdot)(\sigma_{\lambda_2}(z))}
(-s_{\lambda_0\lambda_2}(z))^*\\
&\quad=\tilde{g}(-s_{\lambda_1\lambda_2}(z),-s_{\lambda_0\lambda_1}(z))\tilde{g}
(-s_{\lambda_0\lambda_2}(z),-\check\partial s_{\lambda_0\lambda_1\lambda_2}(z))^*\\
&\quad=\left[\sum_{1\leq i<j\leq n}g(z)_{ij}(s_{\lambda_0\lambda_1}(z)_is_{\lambda_1\lambda_2}(z)_j-
\check\partial s_{\lambda_0\lambda_1\lambda_2}(z)_is_{\lambda_0\lambda_2}(z)_j)\right]_{\R/\Z}.
\end{align*}
Now refer to the construction of the connecting homomorphism after 
Proposition \ref{cohomologylongexactsequence} to find 
$\Delta(\phi(\operatorname{Ad}L\tau)^{20},\phi(\operatorname{Ad}L\tau)^{11},
\phi(\operatorname{Ad}L\tau)^{02})\in {\mathbb H}^3_{\check\partial s}(\pi(\U),\underline{\Z})$, 
which will be given by a triple
\begin{align*}
&\big(\Delta(\phi(\operatorname{Ad}L\tau)^{20},\phi(\operatorname{Ad}L\tau)^{11},
\phi(\operatorname{Ad}L\tau)^{02})^{30}, \\
&\quad \Delta(\phi(\operatorname{Ad}L\tau)^{20},\phi(\operatorname{Ad}L\tau)^{11},
\phi(\operatorname{Ad}L\tau)^{02})^{21},\\
&\quad\quad\Delta(\phi(\operatorname{Ad}L\tau)^{21},\phi(\operatorname{Ad}L\tau)^{11},
\phi(\operatorname{Ad}L\tau)^{02})^{12}\big).
\end{align*}
The latter two terms are relatively easy to compute:
\begin{align*}
&\Delta(\phi(\operatorname{Ad}L\tau)^{20},\phi(\operatorname{Ad}L\tau)^{11},
\phi(\operatorname{Ad}L\tau)^{02})_{\lambda_0\lambda_1\lambda_2}^{21}(z)_l\\
&\quad=\check{\partial}\left[\sum_{1\leq i<j\leq n}g_{ij}(z)((e_l)_is_{\cdot\cdot}(z)_j-
s_{\cdot\cdot}(z)_i(e_l)_j)\right]_{_{\lambda_0\lambda_1\lambda_2}}\\
& \quad\quad\quad +\sum_{1\leq i<j\leq n} g(z)_{ij}(F_{\lambda_0\lambda_1\lambda_2}(z)_i(e_l)_j-
(e_l)_iF_{\lambda_0\lambda_1\lambda_2}(z)_j)=0.
\end{align*}
$$\Delta(\phi(\operatorname{Ad}L\tau)^{20},\phi(\operatorname{Ad}L\tau)^{11},
\phi(\operatorname{Ad}L\tau)^{02})_{\lambda_0\lambda_1}^{12}(z)_{ij}=g(z)_{ij}-g(z)_{ij}=0.$$
The first term
is similar, but requires one to compute the \v Cech differential $\check\partial$ of
\[\bigcap_{k=0}^2 \pi(U_{\lambda_k})\ni z\mapsto \sum_{1\leq i<j\leq n}g(z)_{ij}
(s_{\lambda_0\lambda_1}(z)_is_{\lambda_1\lambda_2}(z)_j-\check\partial 
s_{\lambda_0\lambda_1\lambda_2}(z)_is_{\lambda_0\lambda_2}(z)_j).\]
We find this differential has the formula
\begin{align*}\bigcap_{k=0}^3 \pi(U_{\lambda_k})\ni z\mapsto &\sum_{1\leq i<j\leq n}
g(z)_{ij}\left[\check\partial s_{\lambda_0\lambda_1\lambda_2}(z)_i\check\partial 
s_{\lambda_0\lambda_2\lambda_3}(z)_j-\check\partial s_{\lambda_1\lambda_2\lambda_3}(z)_i
\check\partial s_{\lambda_0\lambda_1\lambda_3}(z)_j\right.\\
&\quad \left. \check\partial s_{\lambda_1\lambda_2\lambda_3}(z)_i s_{\lambda_0\lambda_1}(z)_j-
s_{\lambda_0\lambda_1}(z)_i\check\partial s_{\lambda_1\lambda_2\lambda_3}(z)_j\right].
\end{align*}
Then we see that
\begin{align*}
&\Delta(\phi(\operatorname{Ad}L\tau)^{20},\phi(\operatorname{Ad}L\tau)^{11},
\phi(\operatorname{Ad}L\tau)^{02})_{\lambda_0\lambda_1\lambda_2\lambda_3}^{30}(z)\\
=&\sum_{1\leq i<j\leq n}g(z)_{ij}\left[\check\partial s_{\lambda_0\lambda_1
\lambda_2}(z)_i\check\partial s_{\lambda_0\lambda_2\lambda_3}(z)_j-
\check\partial s_{\lambda_1\lambda_2\lambda_3}(z)_i\check\partial 
s_{\lambda_0\lambda_1\lambda_3}(z)_j\right.\\
&\quad \left. +\check\partial s_{\lambda_1\lambda_2\lambda_3}(z)_i 
s_{\lambda_0\lambda_1}(z)_j-s_{\lambda_0\lambda_1}(z)_i\check\partial 
s_{\lambda_1\lambda_2\lambda_3}(z)_j\right]\\
&\quad-\sum_{1\leq i<j\leq n}g(z)_{ij}(\check\partial 
s_{\lambda_1\lambda_2\lambda_3}(z)_i
s_{\lambda_0\lambda_1}(z)_j-s_{\lambda_0\lambda_1}(z)_i\check\partial 
s_{\lambda_1\lambda_2\lambda_3}(z)_j)\\
&\quad-\sum_{1\leq i<j\leq n}g(z)_{ij}(\check\partial 
s_{\lambda_0\lambda_1\lambda_2}(z)_i\check\partial 
s_{\lambda_0\lambda_2\lambda_3}(z)_j-\check\partial 
s_{\lambda_1\lambda_2\lambda_3}(z)_i\check\partial s_{\lambda_0\lambda_1\lambda_3}(z)_j)
\quad =0.
\end{align*}
\end{proof}

\begin{cor}\label{NonclassicalLiftIndependence}
Let $(X,\U,s)$ be in the standard setup, and suppose $\delta\in 
\check{H}^3(X,\underline{\Z})|_{\pi^{0,3}=0}$. Then the image of $\delta$ under the composition

\centerline{\xymatrix{
\mathrm{Br}_{\R^n}(X)\ar[r]^{\Xi_{\U,s}}&{\mathbb H}^2_{\check\partial s}(\pi(\U),{\mathcal S})\ar[d]&\\
\check{H}^3(X,\underline{\Z})|_{\pi^{0,3}=0}\ar[u]^{F^{-1}}&
{\mathbb H}^3_{\check\partial s}(\pi(\U),\underline{\Z})
}}

\noindent is independent of the choice of lift $\check{H}^3(X,\underline{\Z})|_{\pi^{0,3}=0}
\to \mathrm{Br}_{\R^n}(X)$.
\end{cor}
\begin{proof}
Follows from Proposition \ref{BrauerKernel} and Lemma \ref{Liftindependence}.
\end{proof}

\begin{theorem}\label{DimReductThm}
Let $(X,\U,s)$ be in the standard setup. Then we have an isomorphism of groups
$\check{H}^3(X,\underline{\Z})|_{\pi^{0,3}=0}\cong 
{\mathbb H}^3_{\check\partial s}(\pi(\U),\underline{\Z})$.
\end{theorem}
\begin{proof}
Corollary \ref{NonclassicalLiftIndependence} gives a well-defined homomorphism
$\check{H}^3(X,\underline{\Z})|_{\pi^{0,3}=0}\mapsto 
{\mathbb H}^3_{\check\partial s}(\pi(\U),\underline{\Z})$.
Then, Corollary \ref{dimreducedlongexact}, Proposition 
\ref{BrauerKernel} and Theorem \ref{MainIsomorphism} give a diagram with exact rows
\centerline{\xymatrix{
C(Z,M_n^u(\R))\ar[r]\ar[d]^{\operatorname{id}}&\operatorname{Br}_{\R^n}(X)
\ar[r]\ar[d]^{\cong}_{\Xi_{\U,s}}&\check{H}^3(X,\underline{\Z})|_{\pi^{0,3}=0}\ar[d]\ar[r]&
0\ar[d]\ar[r]&0\ar[d]\\
C(Z,M_n^u(\R))\ar[r]&{\mathbb H}^2_{\check\partial s}(\pi(\U),{\mathcal S})
\ar[r]&{\mathbb H}^3_{\check\partial s}(\pi(\U),\underline{\Z})\ar[r]&0\ar[r]&0
}}
\noindent The squares are commutative by definition (because the maps on the 
bottom row are defined implicitly by going clockwise around the square). 
The isomorphism then follows from the Five Lemma.
\end{proof}

The proof of Theorem \ref{MainSquare} then follows from Theorem \ref{DimReductThm} and its proof.

% ==============================================================================

% ==============================================================================

\end{document}